\documentclass[UTF-8,reqno]{amsart}
\usepackage{enumerate}
\usepackage{mhequ}
\usepackage{graphicx}
\usepackage[margin=1.2in]{geometry}
\usepackage{amssymb,url,color, booktabs,nccmath}
\usepackage{mathrsfs}
\usepackage{enumitem}
\usepackage{graphicx}
\usepackage{tikz}
\usetikzlibrary{shapes,snakes}
\usetikzlibrary{calc}
\usetikzlibrary{decorations.shapes}
\usepackage{color}
\usepackage[colorlinks=true]{hyperref}
\hypersetup{
    linkcolor=blue,          
    citecolor=red,        
    filecolor=blue,      
    urlcolor=cyan
}

\definecolor{darkergreen}{rgb}{0.0, 0.5, 0.0}

\numberwithin{equation}{section}
\def\theequation{\arabic{section}.\arabic{equation}}
\newcommand{\be}{\begin{eqnarray}}
\newcommand{\ee}{\end{eqnarray}}
\newcommand{\ce}{\begin{eqnarray*}}
\newcommand{\de}{\end{eqnarray*}}
\newtheorem{theorem}{Theorem}[section]
\newtheorem{lemma}[theorem]{Lemma}
\newtheorem{remark}[theorem]{Remark}
\newtheorem{definition}[theorem]{Definition}
\newtheorem{proposition}[theorem]{Proposition}
\newtheorem{Examples}[theorem]{Example}
\newtheorem{corollary}[theorem]{Corollary}

\newcommand{\LL}{\mathscr{L}}

\def\Wick#1{\,\boldsymbol{\colon}\!\! #1 \boldsymbol{\colon}\!}
\def\WickC#1{\,\boldsymbol{\colon}\!\! #1 \boldsymbol{\colon}\!\!_{C}}

\def\PPhi{\mathbf{\Phi}}

\def\eps{\varepsilon}

\def\p{\partial}

\def\[{{\Big[}}
\def\]{{\Big]}}
\def\<{{\langle}}
\def\>{{\rangle}}
\def\({{\Big(}}
\def\){{\Big)}}

\def\by{{\mathbf{y}}}
\def\bx{{\mathbf{x}}}

\def\dif{{\mathord{{\rm d}}}}

\def\no{\nonumber}
\def\={&\!\!=\!\!&}

 \newcommand{\eqdef}{\;\stackrel{\mbox{\tiny def}}{=}\;}

\def\bC{{\mathbf C}}

\def\cE{{\mathcal E}}
\def\cF{{\mathcal F}}
\def\cG{{\mathcal G}}
\def\cH{{\mathcal H}}
\def\cI{{\mathcal I}}

\def\cN{{\mathcal N}}

\def\cP{{\mathcal P}}
\def\cQ{{\mathcal Q}}

\def\cS{{\mathcal S}}

\def\mN{{\mathbb N}}

\def\mR{{\mathbb R}}

\def\mT{{\mathbb T}}

\def\mZ{{\mathbb Z}}

\def\bP{{\mathbf P}}

\def\1{{\mathbf{1}}}
\def\mm{{\mathbf{m}}}

\def\sD{{\mathscr D}}

\def\sL{{\mathscr L}}

\def\E{\mathbf E}

\def\geq{\geqslant}
\def\leq{\leqslant}
\def\ge{\geqslant}

\def\eps{\varepsilon}

\def\p{\partial}

\def\[{{\Big[}}
\def\]{{\Big]}}
\def\<{{\langle}}
\def\>{{\rangle}}
\def\({{\Big(}}
\def\){{\Big)}}

\def\by{{\mathbf{y}}}
\def\bx{{\mathbf{x}}}

\def\dif{{\mathord{{\rm d}}}}

\def\no{\nonumber}
\def\={&\!\!=\!\!&}
\def\bt{\begin{theorem}}
\def\et{\end{theorem}}
\def\bl{\begin{lemma}}
\def\el{\end{lemma}}
\def\br{\begin{remark}}
\def\er{\end{remark}}
\def\bx{\begin{Examples}}
\def\ex{\end{Examples}}
\def\bd{\begin{definition}}
\def\ed{\end{definition}}
\def\bp{\begin{proposition}}
\def\ep{\end{proposition}}
\def\bc{\begin{corollary}}
\def\ec{\end{corollary}}

\def\geq{\geqslant}
\def\leq{\leqslant}
\def\ge{\geqslant}

\def\bP{{\mathbf P}}

 \def\R{\mathbb R}
 \def\R{\mathbb R}    
\def\N{\mathbb N}  \def\m{{\bf m}}
   
\def\<{\langle} \def\>{\rangle}

\allowdisplaybreaks

\newenvironment{nouppercase}{%
  \renewcommand{\uppercasenonmath}[1]{}}{}

\tikzset{
        dot/.style={circle,fill=black,inner sep=0pt, outer sep=0.7pt, minimum size=1mm},
        Phi/.style={white!40!red,thick,snake=coil,segment amplitude=0.6pt, segment length=2pt},
         Z/.style={black!40!green,thick,snake=coil,segment amplitude=0.6pt, segment length=2pt},
        C/.style={thick,black},
          Cr/.style={thick,black!20!red},
            Cg/.style={thick,black!40!green},
            K/.style={ultra thick,white!20!blue},
       }

\begin{document}

\title[Large $N$ limit and $1/N$ expansion of invariant observables in $O(N)$ linear $\sigma$-model via SPDE]{\LARGE Large $N$ limit and $1/N$ expansion of invariant observables in $O(N)$ linear $\sigma$-model via SPDE}

\author[Hao Shen]{\large Hao Shen}
\address[H. Shen]{Department of Mathematics, University of Wisconsin - Madison, USA}
\email{pkushenhao@gmail.com}

\author[Rongchan Zhu]{\large Rongchan Zhu}
\address[R. Zhu]{Department of Mathematics, Beijing Institute of Technology, Beijing 100081, China}
\email{zhurongchan@126.com}
\author[Xiangchan Zhu]{\large Xiangchan Zhu}
\address[X. Zhu]{ Academy of Mathematics and Systems Science,
Chinese Academy of Sciences, Beijing 100190, China}
\email{zhuxiangchan@126.com}

\begin{abstract}
In this paper we continue the study of large $N$ problems
for the Wick renormalized  linear sigma model, i.e. $N$-component $\Phi^4$ model, in two spatial dimensions, using stochastic quantization methods and Dyson--Schwinger equations.
We identify the large $N$ limiting law of a collection of Wick renormalized $O(N)$ invariant observables. In particular, under a suitable scaling, the quadratic
observables converge in the large $N$ limit to a mean-zero (singular) Gaussian field denoted by $\cQ$ with an explicit covariance;
and the observables which are $2n$-th renormalized powers of the fields
converge in the large $N$ limit to suitably renormalized
$n$-th powers of $\cQ$. 
The quartic interaction term of the model has no effect on the large $N$ limit of the field $\Phi$, but has nontrivial contributions to the limiting law of the observables, and the renormalization of the $n$-th powers of $\cQ$ in the limit has an interesting finite shift from the standard one.

Furthermore, we derive the $1/N$ asymtotic expansion for the $k$-point functions of the quadratic observables by  employing graph representations and analyzing the order of each graph from Dyson--Schwinger equations.
Finally, turning to the stationary solutions to the stochastic quantization equations, 
with the Ornstein--Uhlenbeck process being  the large $N$ limiting dynamic, 
we derive here its next order correction in stationarity, as described by an SPDE with the right-hand side having explicit fixed-time marginal law which involves the above field $\cQ$.
\end{abstract}


\date{\today}

\begin{nouppercase}
\maketitle
\end{nouppercase}

\setcounter{tocdepth}{1}
\tableofcontents

\section{Introduction}
\label{sec:intro}

The $O(N)$ linear $\sigma$-model is the first paradigm for studying large $N$ problems in quantum field theory (QFT) with $N$ interacting fields.
The model was first
introduced and studied by Wilson \cite{wilson1973quantum} and Coleman--Jackiw--Politzer \cite{coleman1974},
which is an $N$-component generalization of the $\Phi^4_d$ model, given by the  (formal) measure
\begin{equation}\label{e:Phi_i-measure}
	\dif\nu^N(\Phi)\eqdef \frac{1}{C_N}\exp\bigg(-\frac12\int_{\mathbb T^d} \sum_{j=1}^N|\nabla \Phi_j|^2+\frac{\mm}2 \sum_{j=1}^N\Phi_j^2
	+\frac{1}{4N} \Big(\sum_{j=1}^N\Phi_j^2\Big)^2 \dif x\bigg)\mathcal D \Phi
\end{equation}
over $\R^N$ valued fields $\Phi=(\Phi_1,\Phi_2,...,\Phi_N)$ and $C_N$ is a normalization constant.
We will consider the case where the underlying space is a torus and $d=2$ in this paper. In this setting the  interaction should be Wick renormalized   $\Wick{\big(\sum_{j=1}^N\Phi_j^2\big)^2}$
for the measure to be rigorously defined.
The model has an $O(N)$ symmetry which will play an important role throughout the paper: that is, the measure is invariant
under any rotation of the $N$ components of $\Phi$.
In large $N$ problems one aims to show properties of the model as $N\to \infty$.

For every fixed $N\in \N$, the measure
$\nu^N$ can be constructed rigorously as the unique invariant measure of the following
system of equations on the two-dimensional torus $\mathbb T^2$
\begin{equation}\label{eq:21i}
	\LL \Phi_i= -\frac{1}{N}\sum_{j=1}^N \Wick{\Phi_j^2\Phi_i}+\sqrt 2\xi_i,
\end{equation}
where $\LL=\p_t-\Delta+\mm$  with $\mm\geq0$, and $i\in\{1,\cdots,N\}$.
The collection $(\xi_i)_{i=1}^N$ consists of $N $ independent space-time white noises on a stochastic basis, i.e. $(\Omega,\mathcal{F},\mathbf{P})$ with a filtration.
The Wick products will be precisely reviewed in Section~\ref{s:uni}.

The above connection between a quantum field theory \eqref{e:Phi_i-measure} and a stochastic PDE \eqref{eq:21i} is the well-known stochastic quantization. This connection brings novel techniques into the study of large $N$ problems in QFT, such as singular SPDE theories, PDE a priori (uniform) estimates,  mean field limit technique, etc. which were first developed in \cite{SSZZ2d} and \cite{SZZ3d}. We refer to \cite[Section~1]{SSZZ2d} or \cite{MR4472829} for more discussion on the background, motivation, and previous results for large $N$ problems in quantum field theory.
With these new techniques  \cite{SSZZ2d}  proved that the large $N$ limit of $\nu^N$ is given by the Gaussian free field, i.e. the $k$-marginal distribution of $\nu^N$ converges to the Gaussian field $\nu^{\otimes k}$ with $\nu=\cN(0,(\mm-\Delta)^{-1})$.
This result was also extended to 3D in \cite{SZZ3d}.
This means that by only observing the field $\Phi_i$, one does not see any effect from the  interaction term $\Wick{\frac1N\big(\sum_{j=1}^N\Phi_j^2\big)^2}$ in the large $N$ limit.

In this paper we proceed to study the large $N$ behavior of the  following  observables, also known as composite fields
\begin{equ}\label{e:ob}
	\frac{1}{N^{n/2}} \Wick{\Big( \sum_{i=1}^N \Phi_i^2 \Big)^{n}} \;,
\end{equ}
with $n\geq1$, $\Phi=(\Phi_i)_{1\leq i\leq N}\thicksim \nu^N$ for the invariant measure $\nu^N$. The precise definitions of these observables (in particular the Wick product) are given in \eqref{ob0}-\eqref{obn} below. They are basic and natural observables of the model since they are $O(N)$ invariant.
When $n=1,2,\cdots$, we will call them the ``first observable'', ``second observable'', etc.
Remark that  \cite{MR578040} and \cite[Theorem~1.3]{SSZZ2d}
derived explicit formulas for
 the large $N$ limit of the two-point correlations of $\frac{1}{N^{1/2}}\sum_{i=1}^N \Wick{\Phi_i^2}$ and the expectations of $\frac{1}{N} \Wick{\big( \sum_{i=1}^N \Phi_i^2 \big)^2}$, which are shown to be different from those correlations when $\Phi_i$ is replaced by the Gaussian field $\cN(0,(\mm-\Delta)^{-1})$;
this already indicates that these observables receive nontrivial corrections from the interaction term in the  large $N$ limit.

We obtain here the complete description of the
limiting distributions of the observables \eqref{e:ob} as $N\to \infty$,
as  stated in the following main theorem. It relies on deriving exact formulas for correlations in the large $N$ limit.
The availability of such exact formulas is an example that such models have exact solvability  in the large $N$ limit.
Set
$$
\Wick{(\PPhi^2)^n} \eqdef \Wick{\Big(\sum_{i=1}^N\Phi_i^2\Big)^n}.
$$
Let  $H^s:=B^s_{2,2}$ for $s\in\mR$ where $B^s_{2,2}$ is the Besov space which we review in the appendix.

\bt\label{main:1} There exists $m_0>0$ such that for $\mm\geq m_0$, the following statements hold.
\begin{enumerate}
	\item
	As $N\to \infty$,
	the observables $(\frac1{\sqrt N}\Wick{\PPhi^2})_N$ converge in law in $H^{-\kappa}$ for any $\kappa>0$ to  a  mean zero Gaussian field $\cQ$
	 with covariance  $G(x-y)$ determined by\footnote{Compared to \cite{SSZZ2d} we have an extra factor $\frac12$ in front of $|\nabla \Phi_j|^2$ here, and as a result we do not have an extra factor $2$ in \eqref{e:G}.}
	 \begin{align}\label{e:G}
	 C^2*G+G=2C^2,
	 \end{align}
	 where $C=(\mm-\Delta)^{-1}$.
	\item
	As $N\to \infty$,
	the observables $(\frac1N\Wick{(\PPhi^2)^2})_N$ converge in law in $H^{-\kappa}$ for any $\kappa>0$ to  the random field
	$$\WickC{\cQ^2}\eqdef \lim_{\eps\to0}(\cQ^2_\eps-2C^2_\eps(0))\;,$$
where $\cQ_\eps=\cQ*\rho_\eps$	and $C_\eps=C*\rho_\eps$ with some mollifier $\rho_\eps$ and  the limit is understood in $\bC^{-\kappa}$ $\bP$-a.s. for $\kappa>0$.
	\item
	For $\mathbf{n}=(n_1,\dots,n_m)\in\mN^m, m\in\mN$, as $N\to \infty$,
	$$
	\Big\{\Big(\frac1{ N^{n_1/2}}\Wick{(\PPhi^2)^{n_1}} \;,\;
	\cdots \;,\;
	\frac1{ N^{n_m/2}}\Wick{(\PPhi^2)^{n_m}}\Big)\Big\}_N$$
	converge jointly in law to
	$$(\WickC{\cQ^{n_1}} \;,\; \cdots \;,\; \WickC{\cQ^{n_m}})$$
	in $(H^{-\kappa})^m$ for $\kappa>0$ with
\begin{align}\label{lim:n1}
	\WickC{\cQ^{n}}\eqdef \lim_{\eps\to0} (2C_\eps^2(0))^{n/2}H_n((2C_\eps^2(0))^{-1/2}\cQ_\eps)\qquad n\in\mN.
\end{align}
where the limit is understood in $\bC^{-\kappa}$ $\bP$-a.s. for $\kappa>0$ and $H_n$ is the Hermite polynomial defined in \eqref{def:Her}. 
\end{enumerate}
\et


We will prove the convergence of the above limit \eqref{lim:n1} in \eqref{lim:n}.
These results can be derived from Theorem \ref{th:1} and Theorem \ref{th:o}. Note that part (3) is the most general case, while choosing $n=1$ and $n=2$ recovers cases (1) and (2) respectively.

Theorem~\ref{main:1} reveals interesting structures of these $O(N)$ invariant observables  in the large $N$ limit, which we briefly explain here. First of all,
in part (1), $2C^2$ and $G$ are both singular at the origin,
but we will show that their difference $C^2*G$ is smooth everywhere including the origin.
We recall that $2C^2(x-y)$ is the correlation
for $\frac1{\sqrt N} \Wick{\sum_{i=1}^N Z_i^2}$ where $Z_i \sim \cN(0,(\mm-\Delta)^{-1})$ are i.i.d. Gaussian fields;
so there is
indeed a nontrivial (order one) interaction between
the observables $(\frac1{\sqrt N}\Wick{\PPhi^2})_N$
and the non-Gaussian term $\frac1N\Wick{(\PPhi^2)^2}$ of the measure $\nu^N$.

Also, one has $\WickC{\cQ^2}=\Wick{\cQ^2}-C^2*G(0)$, where $\Wick{\cQ^2}$ is the mean-zero Wick product (formally ``renormalized by $G(0)$'') defined in Proposition \ref{pro:1}. Consequently,   part (2) is consistent with (but much stronger than) the formula for $\lim_{N\to \infty} \E \frac{1}{N} \Wick{\big( \sum_{i=1}^N \Phi_i^2 \big)^2}$ obtained in \cite[Theorem~1.3]{SSZZ2d}.

Moreover, formally we can view $\frac{1}{N^{n/2}}\Wick{(\PPhi^2)^n}$ as the ``$n$-th power'' of $\frac{1}{\sqrt{N}}\Wick{\PPhi^2}$, except that one has to take into account suitable renormalization.    To elaborate further, we present the following formal calculations using Wiener chaos decomposition (c.f. \cite[Lemma 10.3]{Hairer14})  for the square of $\Wick{\PPhi^2}$ 
\begin{equ}[e:formal4]
		\frac1{N}\Wick{(\PPhi^2)^2}=\Big(\frac1{\sqrt N}\Wick{\PPhi^2}\Big)^2-2C^2(0)-4C(0)\frac1N \Wick{\PPhi^2}
		\quad \to \quad \cQ^2-2C^2(0),
\end{equ}
where we pretend that $(\Phi_i)$ are  independent Gaussian free fields, which is indeed the case in the  large $N$ limit, and we use law of large numbers to formally deduce that the last term in the first line vanishes.  This observation suggests that the appropriate renormalization constant in the large N limit should be given by $2C^2(0)$  instead of $G(0)$, i.e. The renormalization constant is still obtained by substituting $\Phi_i$ with its large $N$ limit the Gaussian free field $Z_i$, while the correlation of the observable is influenced by the nonlinear interaction. The general heuristic argument for any $n\in \mathbb{N}$ can be derived through induction. For more detailed explanations, we refer to Section \ref{sec:4.3}. However, it is important to note that the above calculation is purely formal, as we treat $\Phi$ as its large $N$ limit, i.e.,  independent Gaussian free fields, and apply Gaussian property. Additionally, it is worth mentioning that $C^2(0)$ is infinite, which renders it nonsensical in the context of the above calculation.

We prove these results by combining uniform in $N$ estimates from the stochastic quantization equations \eqref{eq:21i} and integration by parts formula (also called Dyson-Schwinger equations). This strategy has been shown to be fruitful in studying large $N$ problems  in \cite{SSZZ2d,SZZ3d}, as well as perturbation theory and small scale behavior of correlations in \cite{SZZ21}, and it will be further developed in this paper.

More precisely, to establish these results, we start by improving uniform estimates for the $H^{-\kappa}$ norm of $\frac{1}{N^{n/2}}\Wick{(\PPhi^2)^n}$ for $n\in\mathbb{N}$ and $\kappa>0$ using stochastic quantization equations \eqref{eq:21i}  (as stated in Proposition \ref{prop1}), where we also upgrade from the second moment bounds in \cite{SSZZ2d} to higher moment bounds. These moment bounds  play two crucial roles:
First, they are important for proving tightness of the observables in \eqref{e:ob}.
The uniform estimates ensure that the observables remain bounded as $N$ tends to infinity, which is essential for studying their limiting behaviors.
Second, these estimates allow us to  prove that the error terms from integration by parts formula converge to zero as $N$ approaches infinity.

Integration by parts (IBP for short) is another crucial component in our analysis.
\footnote{Dyson-Schwinger equations were also used by physicists to study large $N$ problems, see \cite{symanzik1977}, on a formal level. The aformentioned rigorous result in \cite[Section~3]{MR578040} used integration by parts twice, and in this paper we perform IBP in a more recursive and systematic way, especially in the proof of Theorem~\ref{main} below.}
 By selecting appropriate test functions, we can derive a set of recursive equations (as given in \eqref{eq:inductive2}) for the large $N$ limit of the $k$-point functions associated with the observables in \eqref{e:ob}.
These recursive equations provide a systematic way to compute the limiting behavior of these $k$-point functions as $N$ becomes large. In particular, we can identify the solutions of these recursive equations as the $k$-point functions of $\WickC{\cQ^n}$.
 This allows us to understand the statistical properties and correlations among the observables as $N$ tends to infinity.
 
\medskip



The above strategy, with extra effort, allows us to study more intricate properties of the model and its observables.  As the second objective of this paper we prove a  $1/N$ asymptotic expansion for the $k$-point correlations of $\frac{1}{\sqrt{N}}\Wick{\PPhi^2}$. This $1/N$ expansion provides valuable insights into the behavior of the $k$-point correlations of these observables and helps us understand the leading and subleading terms as $N$ becomes large. In physics such $1/N$ expansions are often  powerful calculation tools   in the study of critical phenomena, where in the vicinity of critical points
 the usual perturbative expansions (in the coupling constant of the model)
are not informative (see e.g. \cite[Section~1]{Klebanov2014} for some discussion). Our theorem below gives an error bound on 
this $1/N$ expansion.

We define the $k$-point correlations.
Denote $\Phi_\eps\sim \nu_\eps^N$ with $\nu_\eps^N$ being a lattice approximation to \eqref{e:Phi_i-measure} (see \eqref{eq:Phiin}). For $k\in\mN$ we
define the $k$-point correlation function for the observable $\frac1{\sqrt{N}}\Wick{\PPhi^2}$ as
\begin{align}\label{def:fkN}
	\<f_k^N,\varphi\>=\lim_{\eps\to0}\int \E\Big(\prod_{i=1}^k\cE^\eps\frac{1}{\sqrt N}\Wick{\PPhi^2_\eps}(y_i)\Big)\varphi(y_1,\dots,y_k)\prod_{i=1}^k\dif y_i.
\end{align}
for every $\varphi\in \cS(\mT^{2k})$, where $\cE^\eps$ is the extension operator from discrete to continuum introduced in \eqref{def:E}. Our second main result is the following $1/N$ expansion of $f_k^N$.

\bt\label{main} Let $\m$ be as in Theorem \ref{main:1}.
For any $k \ge 1$ and $p\geq1$
we have the following representation for the $k$-point correlation
\begin{equ}
	f^N_{k}  = \sum_{n=0}^p  \frac{1}{N^{n}} F_{n}^{k,1} +  \frac{1}{N^{p+1}} R_{p+1}^{k,1},\quad k\in 2\mN,
\end{equ}
and
\begin{equ}
	f^N_{k}  = \sum_{n=0}^p  \frac{1}{N^{n+1/2}} F_{n}^{k,2} +  \frac{1}{N^{p+3/2}} R_{p+1}^{k,2},\quad k\in 2\mN-1,
\end{equ}
where $F_{n}^{k,1}, F_{n}^{k,2}$ only depend on the Green's function $C=(\mm-\Delta)^{-1}$ and are independent of $N$,  and
\begin{align*}
	\|R_{p+1}^{k,1}\|_{(H^{-\kappa})^k}+\|R_{p+1}^{k,2}\|_{(H^{-\kappa})^k}\lesssim1,
\end{align*}
with the proportional constant independent of $N$.
\et

Based on the results, we observe that the odd moments of $\frac{1}{\sqrt{N}}\Wick{\PPhi^2}$ are of at least the order of $\frac{1}{\sqrt{N}}$, while the even moments are of order $1$.  This aligns with the order of the moments when replacing $\Phi$ by its large $N$ limit and also is consistent with the formal Taylor  expansion  for the density $\frac{\dif \nu^N}{\dif \nu^{\otimes N}}\propto\exp(-\frac1{4N}\int\Wick{(Z_j^2)^2}\dif x)$ of $\nu^N$ w.r.t. the Gaussian free field (see \cite[Section 1.3]{SZZ3d} for more discussion).
The above bound on $R_{p+1}^{k,1}$, $R_{p+1}^{k,2}$ shows that the expansions are asymptotic.
We note that a  $1/N$ expansion for the infinite volume pressure i.e. vacuum energy  was  derived in \cite[Theorem~1]{MR578040}; the proof used a ``dual'' representation of the partition function and tools such as chessboard estimates from constructive quantum field theory (see also  \cite{MR661137} and more recent paper \cite{ferdinand2022borel} for $1/N$ expansion and Borel summability for Schwinger functions via similar approach or loop vertex expansion). Here, in Theorem \ref{main}, we are considering the $k$-point functions
of the first observable 
for {\it arbitrary} $k$. 
Our approach  combining the Dyson--Schwinger equations with the stochastic quantization equations 
 is new and powerful once all the a priori estimates are available, and the method can be easily extended to obtain a $1/N$ expansion for the $k$-point functions of more general observables.
	

Theorem \ref{main} follows from Theorem \ref{m:5} and Proposition \ref{pr:N}.
A more delicate analysis than that for Theorem \ref{main:1}
 is required for the $k$-point correlation functions. Instead of directly taking the limit as $N\to\infty$, we need to perform a more careful analysis.

Our strategy involves introducing a graphic representation for each term obtained from the integration by parts (IBP) procedure. This graphic representation allows us to visualize and analyze the structure of the terms. To control these graphs, we apply a priori estimates based on uniform estimates from stochastic quantization equations, similar to the approach used in \cite{SZZ21}.

However, the procedure becomes more complicated in our case because each error term is of order $\frac{1}{\sqrt{N}}$, and the approach in \cite{SZZ21} would only provide a ``$1/\sqrt{N}$ expansion''. To obtain a $1/N$ expansion, we need to analyze the structure of each term obtained from IBP, develop more useful recursions in Section~\ref{sec:IBP} and classify the terms into two categories, and iterate IBP in a way that depends on the parity of occurrences of $\PPhi^2$. See Section~\ref{sec:5} for details.

\medskip

In the third part of the paper
we turn to questions about the dynamics.
The earlier work \cite{SSZZ2d} constructed stationary processes $(\Phi_i,Z_i)_{1\leq i\leq N}$ such that $\Phi_i$ and $Z_i$ are stationary solutions to \eqref{eq:21i} and the linear equations
\begin{align*}
	\sL Z_i=\sqrt 2\xi_i,
\end{align*}
respectively. It was shown that with decomposition $\Phi_i=Y_i+Z_i$ one has $\E\|Y_i\|_{H^1}^2\lesssim \frac1N$ (see \cite[Lemma 6.2]{SSZZ2d} and also Lemma \ref{th:m2} below). Since the stationary distribution of $Z_i$ is given by the Gaussian free field $\nu=\cN(0,(\mm-\Delta)^{-1})$, a corollary of the above result is the convergence of $k$-marginal distribution of $\nu^N$ to $\nu^{\otimes k}$.

Given the above result,
it is  natural to ask what is the ``next order correction'' to the limiting dynamic.
 In this paper we consider the large $N$ limit behavior of the stationary process $\sqrt N(\Phi_i-Z_i)$ and obtain the following result.

\bt\label{main:2}  Let $\mm$ be as in Theorem \ref{main:1}. For every $k\in \mN$, $\{(\sqrt N(\Phi_1-Z_1),\dots,\sqrt N(\Phi_k-Z_k))\}_N$ is tight in $$\Big(L^2_{\text{loc}}(\mR^+;\bC^\kappa)\cap L^2_{\text{loc}}(\mR^+;H^\delta)\cap C(\mR^+;H^{-2\kappa})\Big)^k,$$
for some small $\kappa>0, 0<\delta<1$.
Every tight limit $u_i$ is a stationary process and satisfies the following equation
\begin{align}
	\sL u_i=\cP_i
\end{align}
in the analytic weak sense, where $\{\cP_1,\dots,\cP_k\}$ is stationary process  with the time marginal distribution $\{X_1\cQ,\dots,X_k\cQ\}$. Here $X_i, i=1,\dots,k,$ and $\cQ$ are independent, $X_i=^dZ_i$ and $\cQ$ is the Gaussian field as in Theorem \ref{main:1}.
\et

Theorem \ref{main:2} follows from Theorem \ref{th:5.1}, Theorem \ref{th:3} and Theorem \ref{th:4}. The key step of the proof is to identify the joint distributions of $\cP_i$, which again follows from a combination of IBP and uniform estimates from stochastic quantization equations \eqref{eq:21i}. By applying IBP and utilizing the uniform estimates, we can analyze the correlations and dependencies among the $\cP_i$ variables. This allows us to identify the joint distributions of these variables.

Remark that it would be interesting to apply and extend our methodology to three spatial dimensions.
 The work \cite{SZZ3d} proved the convergence of the coupled $\Phi^4_3$ field to the Gaussian free field and the tightness of the first observable. However, the uniform estimates for the stochastic quantization equations in three dimensions as presented in \cite{SZZ3d}, are currently insufficient to obtain the statistical behavior as $N\to \infty$ (e.g. extend results of  the present paper to 3D). Further investigations are required in this direction and will be very interesting.	When $N=1$, remarkable results have been obtained  recently for stochastic quantization of $\Phi^4_3$, including the construction of local solutions \cite{Hairer14,GIP15,MR3846835}, global solutions and estimates \cite{MW18,GH18,AK17,MoinatCPAM, GH18a, JagannathPerkowski2021}, as well as a priori bounds in fractional dimension $d<4$ by \cite{CMW19,GR2023}. These new techniques could potentially be helpful in improving the a priori bounds as required in the study of large $N$ problems in three dimensions.

As more far-reaching goals, it would be interesting to study large $N$ problems beyond $\Phi^4$ type models, for either  invariant measures or observables or the associated stochastic dynamics.
For instance, the coupled KPZ systems \cite{MR3653951}, random loops in $N$ dimensional manifolds \cite{hairer2019geo,hairer2016motion, RWZZ17, CWZZ18} and
	Yang-Mills type models \cite{CCHS20,CCHS_3D,Shen2018Abelian,Chevyrev22YM} where the dimension of the Lie group or its representation space tends to infinity. 
In the last case,  the Yang--Mills measure in 2D is known to converge to a deterministic limit called the master field \cite{Levy2011,DN2020yang,DL2022I,DL2022II} which satisfies the Makeenko--Migdal equations; on lattice much more results can be proved, see \cite{MR3919447,chatterjee2016} and dynamic approach \cite{SSZloop,SZZ22}.

\subsection*{Structure of the paper}
This paper is structured as follows.
In Section \ref{s:uni}, we obtain uniform in $N$ estimates  for the stationary solutions to the equations \eqref{eq:21i}. These estimates are essential for subsequent analysis and provide the foundation for our results.
Section \ref{sec:IBP} is dedicated to deriving the recursive relation for the $k$-point functions of the observables through the integration by parts (IBP) technique. This recursive relation allows us to study the correlations and dependencies among the observables.
In Section \ref{sec:large}, we investigate the behavior of the observables in the large $N$ limit. Specifically, in Section \ref{sec:4.1}, we prove that the large $N$ limit of the first observable corresponds to a Gaussian field $\cQ$. In Section \ref{sec:4.3}, we establish that the general observables converge to suitably renormalized powers of $\cQ$ by solving the recursive equations obtained earlier.
Section \ref{sec:5} is dedicated to deriving the $1/N$ expansion for the $k$-point functions of the first observables. Finally, in Section \ref{sec:6}, we study the behavior of the fluctuations of the dynamics in the large $N$ limit and prove Theorem~\ref{main:2}.

\subsection*{Notations.} 
Set $\mN_0\eqdef \mN\cup \{0\}$.  Throughout the paper, we use the notation $a\lesssim b$ if there exists a constant $c>0$ such that $a\leq cb$, and we write $a\simeq b$ if $a\lesssim b$ and $b\lesssim a$. 

Let $\mathcal{S}'(\mT^d)$ and $\mathcal{S}(\mT^d)$ be the space of  distributions and its dual on $\mathbb{T}^d$ with $\<\cdot,\cdot\>$ as the usual duality between $\mathcal{S}'(\mT^d)$ and $\mathcal{S}(\mT^d)$. 
We denote by $B^\alpha_{p,q}$ the Besov spaces on the torus with general indices $\alpha\in \R$, $p,q\in[1,\infty]$.
 The H\"{o}lder-Besov space
$\bC^\alpha$ is given by $\bC^\alpha=B^\alpha_{\infty,\infty}$ and we will often write $\|\cdot\|_{\bC^\alpha}$ instead of $\|\cdot\|_{B^\alpha_{\infty,\infty}}$. We refer to Appendix \ref{A.1} for the definition of Besov spaces.
Set $\Lambda= (1-\Delta)^{\frac{1}{2}}$.
For $s\geq0$, $p\in [1,+\infty]$ we use $H^{s}_p$ to denote the subspace of $L^p$, consisting of all  $f$   which can be written in the form $f=\Lambda^{-s}g$ with $g\in L^p$, and the $H^{s}_p$ norm of $f$ is defined to be the $L^p$ norm of $g$, i.e. $\|f\|_{H^{s}_p}:=\|\Lambda^s f\|_{L^p}$. For $s<0$, $p\in (1,\infty)$, $H^s_p$ is the dual space of $H^{-s}_q$ with $\frac{1}{p}+\frac{1}{q}=1$.
 Set $H^s:=B^s_{2,2}$ for $s\in\mR$. We also use $B^{\alpha,\eps}_{p,q}, \alpha\in\mR, p,q\in [1,\infty]$ and $H^{s,\eps}, s\in\mR$ to denote discrete Besov sapces and discrete Sobolev spaces on $\Lambda_\eps\eqdef\eps \mZ^2\cap \mT^2$.  Their definitions are also given in Appendix \ref{A.1}. We also denote by $\widehat{f}$  the Fourier transform of $f$ on $\mT^2$.

Given a Banach space $E$ with a norm $\|\cdot\|_E$, for an interval
 $B\subseteq\mR$ we write $C(B;E)$ for the space of continuous functions from $B$ to $E$. Set $C_TE=C([0,T];E)$ for every $T>0$.
 For $p\in [1,\infty]$ we write $L^p([0,T];E)=L^p_TE$ for the space of $L^p$-integrable functions from $[0,T]$ to $E$, equipped with the usual $L^p$-norm. We also use $L^p_{\mathrm{loc}}(\mR;E)$ to denote the space of functions $f$ from $\mR$ to $E$ satisfying $f|_{[0,T]}\in L^p((0,T); E)$ for all $ T>0$. 
 We also denote $L^p=L^p(\mathbb{T}^2)$.

\subsection*{Acknowledgements} We would like to thank Scott Smith for many helpful discussions on large N problems.
H.S. gratefully acknowledges support by NSF via the grants DMS-1954091 and  CAREER DMS-2044415.  R.Z. and X.Z. are grateful to the financial supports by National Key R\&D Program of China (No. 2022YFA1006300). R.Z. gratefully acknowledges financial support from the NSFC (No.  12271030) and BIT Science and Technology Innovation Program Project 2022CX01001. X.Z. is grateful to the financial supports in part by National Key R\&D Program of China
	(No. 2020YFA0712700) and the NSFC (No. 12090014, 12288201) and the support by key Lab of Random Complex
	Structures and Data Science, Youth Innovation Promotion Association (2020003), Chinese Academy of Science.
	R.Z. and X.Z. are grateful to
	the financial supports  by the Deutsche Forschungsgemeinschaft (DFG, German Research Foundation) – Project-ID 317210226--SFB 1283.

 \section{Uniform estimates from SPDEs}\label{s:uni}
In this section, we will revisit uniform estimates in  \cite{SSZZ2d} and also establish new and stronger estimates for the stationary solutions of the following system of SPDEs
\begin{equation}\label{eq:Phi2d}
	\LL \Phi_i=-\frac{1}{N}\sum_{j=1}^N \Wick{\Phi_j^2\Phi_i}+\sqrt 2\xi_i.
\end{equation}
These estimates are essential for the analysis in the subsequent sections.

Let $Z_{i}$ be the stationary solution to the linear SPDE
\begin{equation}\label{eq:li1}
	\LL Z_i=\sqrt 2\xi_i,
\end{equation}
and let  $Y_{i}$ be the stationary solution to the  equation
\begin{equ}\label{eq:22}
	\LL Y_i =-\frac{1}{N}\sum_{j=1}^N(Y_j^2Y_i+Y_j^2Z_i+2Y_jY_iZ_j+2Y_j\Wick{Z_iZ_j}
	+ \Wick{Z_j^2}Y_i+\Wick{Z_i Z_j^2 }).
\end{equ}
The notations $\Wick{Z_iZ_j} $, $ \Wick{Z_j^2}$ and $\Wick{Z_i Z_j^2}$ denote renormalized products of Wick type which are defined as follows:
let $Z_{i,\varepsilon}$ be a space-time mollification of $Z_i$ and define
\begin{equ}[e:wick-tilde]
	\Wick{Z_iZ_j}=
	\begin{cases}
		\lim\limits_{\varepsilon\to0}(Z_{i,\varepsilon}^2-a_\varepsilon)  &  (i=j)\\
		\lim\limits_{\varepsilon\to0}Z_{i,\varepsilon}Z_{j,\varepsilon} & (i\neq j)
	\end{cases}
	\quad
	\Wick{Z_iZ_j^2} =
	\begin{cases}
		\lim\limits_{\varepsilon\to0}(Z_{i,\varepsilon}^3-3a_\varepsilon Z_{i,\varepsilon})   & (i=j)\\
		\lim\limits_{\varepsilon\to0}(Z_{i,\varepsilon}Z_{j,\varepsilon}^2-a_\varepsilon Z_{i,\varepsilon}) & (i\neq j)
	\end{cases}
\end{equ}
where $a_\varepsilon=\mathbf{E}[ Z_{i,\varepsilon}^2]$ is a diverging constant independent of $i$ and the limits are understood in $C_T\bC^{-\kappa}$ $\bP$-a.s. for $\kappa>0$. For fixed $N\in\mN$, using \cite[Lemma 5.7]{SSZZ2d} we have constructed a stationary process $(\Phi_i^N,Z_i)_{1\leq i\leq N}$ such that the components $\Phi^N_i$ and $Z_i$ are stationary solutions to \eqref{eq:Phi2d} and \eqref{eq:li1}, respectively.  In the following we omit the superscript $N$ for simplicity. We know that $Y_i=\Phi_i-Z_i$ is a stationary solution to \eqref{eq:22}.  We recall the following uniform in $N$ bounds for the stationary solutions $Y_i$ to equation \eqref{eq:22} from \cite[Lemma 6.2]{SSZZ2d}.

\bl\label{th:m2} There exists an $m_{0}$ such that for $\mm \geq m_{0}$ and $q\geq 1$
\begin{align}
	&\E \bigg [ \bigg (\sum_{i=1}^N\|Y_i\|_{L^2}^2 \bigg)^q \bigg ]+\E \bigg[\bigg(\sum_{i=1}^N\|Y_i\|_{L^2}^2+1\bigg)^{q}\bigg(\sum_{i=1}^N\|\nabla Y_i\|_{L^2}^2\bigg) \bigg] \lesssim 1,\label{se19} \\
	&\E \bigg[\bigg(\sum_{i=1}^N\|Y_i\|_{L^2}^2+1\bigg)^{q}\bigg\|\sum_{i=1}^NY_i^{2} \bigg\|_{L^2}^2\bigg]\lesssim 1,\label{se16}
\end{align}
where the implicit constants are independent of $N$.
\el

Let $\nu^N$ be the unique invariant measure of \eqref{eq:Phi2d}.
Recall that our main objects of interest are the observables
\begin{equ}\label{ob}
	\frac{1}{N^{n/2}} \Wick{\Big( \sum_{i=1}^N \Phi_i^2 \Big)^{n}}
\end{equ}
with $n\geq1$, and $\Phi=(\Phi_i)_{1\leq i\leq N}\thicksim \nu^N$.
These observables are defined precisely using
 the stationary solutions $Z_i$ and $Y_i$ to \eqref{eq:li1} and \eqref{eq:22}, for instance,
 \begin{equs}
 	\frac{1}{\sqrt N} \Wick{ \sum_{i=1}^N \Phi_i^2 }
 	& \eqdef
 	\frac{1}{\sqrt N}\sum_{i=1}^N \Big(Y_i^2+2Y_iZ_i+\Wick{Z_i^2}\Big),\label{ob0}
\\
	\frac{1}{N} \Wick{\Big( \sum_{i=1}^N \Phi_i^2 \Big)^2}
	& \eqdef
	\frac{1}{N}\sum_{i,j=1}^N \Big(Y_i^2Y_j^2+4Y_i^2Y_jZ_j+2Y_i^2\Wick{Z_j^2} \label{ob1}
	\\&
	\qquad+\Wick{Z_i^2Z_j^2}+4Y_i\Wick{Z_iZ_j^2}+4Y_iY_j\Wick{Z_iZ_j}\Big),\label{ob2}
\end{equs}
and more generally,
\begin{equs}[obn]
	\frac{1}{N^{n/2}} \Wick{\Big( \sum_{i=1}^N \Phi_i^2 \Big)^n}
	& \eqdef
	\frac{1}{N^{n/2}}\sum_{k=0}^n\sum_{m=0}^{n-k}\frac{n!}{k!m!(n-k-m)!} \Big(\sum_{i=1}^NY_i^2\Big)^k  \\
	&\qquad\qquad\qquad\qquad
	\Wick{\Big(2\sum_{j=1}^NY_jZ_j\Big)^m\Big(\sum_{\ell=1}^NZ_\ell^2\Big)^{n-k-m}}\;.
\end{equs}
Here the Wick products are canonically defined as in \eqref{e:wick-tilde} with $a_\varepsilon=\E[Z_{i,\varepsilon}^2]$, in particular
\begin{equ}[e:Zi2Zj2]
	\Wick{Z_i^2Z_j^2} =
	\begin{cases}
		\lim\limits_{\varepsilon\to0}(Z_{i,\varepsilon}^4-6a_\varepsilon Z_{i,\varepsilon}^2+3a_\varepsilon^2)   & (i=j)\\
		\lim\limits_{\varepsilon\to0}(Z_{i,\varepsilon}^2-a_\varepsilon)(Z_{j,\varepsilon}^2-a_\varepsilon) & (i\neq j),
	\end{cases}
\end{equ}
where the limits are understood in $C_T\bC^{-\kappa}$ $\bP$-a.s. for $\kappa>0$. 
The Wick product of $(2\sum_{j=1}^NY_jZ_j)^m$ and
$\big(\sum_{\ell=1}^NZ_\ell^2\big)^l$ means the the canonical Wick products for $Z_j, Z_\ell$ and the usual product with $Y_j$.
We also define the Wick product inductively by
\begin{equ}[e:Zik]
	\Wick{Z_{i}^n}\eqdef\lim_{\eps\to0}\Wick{Z_{i,\eps}^n} \eqdef\lim_{\eps\to0}a_\eps^{n/2}H_n(a_\eps^{-1/2}Z_{i,\eps}),
\end{equ}
and
\begin{equ}[e:ZiZn]
	\Wick{\prod_{k=1}^pZ_{i_k}^{n_k}Z_j^m} =
	\begin{cases}
	\lim\limits_{\eps\to0}
	(\Wick{\prod_{k=1,k\neq \ell}^pZ_{i_k,\eps}^{n_k}}\Wick{Z_{i_\ell,\eps}^{n_\ell+m}} )  & (i_\ell=j, \mbox{ for some } \ell ),\\
		\lim\limits_{\varepsilon\to0}(\Wick{\prod_{k=1}^pZ_{i_k,\eps}^{n_k}}\Wick{Z_{j,\eps}^{m}}) & (i_k\neq j, \forall k=1,\dots p),
	\end{cases}
\end{equ}
where $n,m, p, n_k\in\mN$, $i_k,j\in\{1,\dots,N\}$, $i_1,\dots,i_k$ are all distinct and $H_n$ with $n\geq1$ are Hermite polynomials explicitly given by
\begin{align}\label{def:Her}
	H_n(x)=\sum_{j=0}^{\lfloor n/2\rfloor}(-1)^j\frac{n!}{(n-2j)!j!2^j}x^{n-2j}.
\end{align}

Although we use the stationary solutions $Y_i$ and $Z_i$ to define  observables in \eqref{ob}, the laws of them only depend on $\nu^N$ (see \cite[Remark 6.1]{SSZZ2d}).

We also recall the following classical results for the above Wick products.

\bl\label{lem:z} For $\ell\geq1, \kappa>0$ it holds that
\begin{align*}
	\E\Big\|\Wick{\prod_{k=1}^pZ_{i_k}^{n_k}}\Big\|_{\bC^{-\kappa}}^\ell
	\lesssim 1,
\end{align*}
with $p, n_k\in\mN$, $i_k\in\{1,\dots,N\}$ and the proportional constant is independent of 
the indices $i_k$.
\el

As we will see, by repeatedly applying  integration by parts
to the above observables, we will obtain terms which are the main contribution to the large $N$ limits, as well as ``remainder'' or ``error'' terms.
All the remainder terms for the observable $\frac{1}{N^{1/2}}\sum_{i=1}^N \Wick{\Phi_i^2}$  from IBP will be controlled by  terms of the following form:
\begin{align*}
	\E A_1^{\ell_1}A_2^{\ell_2}A_3^{\ell_3},
\end{align*}
with $\ell_i\geq0,$ where  (for $s>3\kappa>0$ small enough)
\begin{align*}
	A_1 &\eqdef \|Y_1\|_{L^2}+\|Z_1\|_{\bC^{-\frac\kappa 3}},\\
	A_2 &\eqdef  \bigg\|\sum_{i=1}^NY_i^2\bigg \|_{L^{1}} +\sum_{i=1}^N\|Y_i\|_{H^{\kappa}}\|Z_i\|_{\bC^{-\frac\kappa2}}+\bigg\|\sum_{i=1}^N\Wick{Z_i^2} \bigg \|_{H^{-\frac\kappa2}},\\
	A_3 &\eqdef \bigg\|Y_1\sum_{i=1}^NY_i^2 \bigg \|_{L^{1+s}}+\bigg \|\Lambda^{-s}(Z_1\sum_{i=1}^NY_i^2) \bigg \|_{L^{1+s}}
	\\&\quad+\bigg \|\Lambda^{-s}(Y_1\sum_{i=1}^NY_iZ_i) \bigg \|_{L^{1+s}}+\bigg \|\Lambda^{-s}(\sum_{i=1}^NY_i\Wick{Z_1Z_i}) \bigg \|_{L^{1+s}}
	\\&\quad+\bigg \|\Lambda^{-s}(Y_1\sum_{i=1}^N\Wick{Z_i^2}) \bigg \|_{L^{1+s}}+\bigg \|\sum_{i=1}^N\Wick{Z_1Z_i^2} \bigg \|_{H^{-s}}
	.
\end{align*}

For $A_1, A_2$ and $A_3$ we recall the following moment estimates from \cite[Lemma 6.6]{SSZZ2d}.

\bl\label{lem:zmm} Let $\mm$  as in Lemma \ref{th:m2}. For each $\ell_i\geq0$, $\kappa, s>0$ with $\ell_2\frac\kappa2+3\ell_3 s<1$, it holds that
\begin{align*}
	\E A_1^{\ell_1} A_2^{\ell_2} A_3^{\ell_3} \lesssim N^{(\ell_2+\ell_3)/2}
\end{align*}
where the proportional constant is independent of $N$.
\el

Lemma \ref{lem:zmm} is sufficient to control the remainders  from IBP to determine the large $N$ limit of  the first observables $\frac1{N^{1/2}}\sum_{i=1}^N\Wick{\Phi_i^2}$. To determine the large $N$ limit of {\it all} the observables in \eqref{ob}, we derive a uniform in $N$  higher moments estimate for the $\bC^\kappa$-norm of $Y_i$.
Below we write $S_t=e^{t(\Delta-\mm)}$.

\bl\label{lem:zin}
Let $\mm$  as in Lemma \ref{th:m2}.
 For $\ell\geq1, \kappa>0$ with $12\ell \kappa<1$ it holds that
\begin{align*}
	\E\|Y_i\|_{L^\infty}^\ell\lesssim\E\|Y_i\|_{\bC^\kappa}^\ell \lesssim \frac1{N^{\ell/2}},
\end{align*}
and for $24\ell\kappa<1$
\begin{align*}
	\E\Big(\sum_{i=1}^N\|Y_i\|_{\bC^\kappa}^2\Big)^\ell\lesssim 1,
\end{align*}
where the proportional constants may depend on $\ell$ but  are independent of $N$.
\el
\begin{proof}
	Write  \eqref{eq:22} into the mild form:
	\begin{align*}
		Y_i(t)=S_tY_i(0)-\frac1N\int_0^tS_{t-r}\Big(\sum_{j=1}^N(Y_j^2Y_i+Y_j^2Z_i+2Y_jY_iZ_j+2Y_j\Wick{Z_iZ_j}
		+ \Wick{Z_j^2}Y_i+\Wick{Z_i Z_j^2 })\Big)\dif r.
	\end{align*}
	For the first term involving initial data, we use  smooth effect of the heat operator to have for $\kappa>0$ small enough
	\begin{align*}
		\|S_tY_i(0)\|_{\bC^\kappa}\leq e^{-\mm t}\|Y_i(0)\|_{\bC^\kappa}.
	\end{align*}
We then
 apply  Besov embedding Lemma \ref{lem:emb} to have  $B^{-s}_{1+s,\infty}\subset\bC^{2\kappa-2}$ for $s>s^2+2\kappa(1+s)$ and Lemma \ref{lem:sch} to  find  that for $s>0$ small enough and $\delta>0$
	\begin{align*}
		\|Y_i\|_{L_T^p\bC^\kappa}^p
		&\leq \frac1{\mm p}(1-e^{-\mm Tp})(1+\delta)\|Y_i(0)\|_{\bC^\kappa}^p
		\\&\quad + \frac{C_{\delta,p}}{N^p}\bigg[\bigg\|Y_i\sum_{j=1}^NY_j^2 \bigg \|_{L_T^pL^{1+s}}^p+\bigg \|\Lambda^{-s}(Z_i\sum_{j=1}^NY_j^2) \bigg \|_{L_T^pL^{1+s}}^p
		\\&\quad+\bigg \|\Lambda^{-s}(Y_i\sum_{j=1}^NY_jZ_j) \bigg \|_{L_T^pL^{1+s}}^p+\bigg \|\Lambda^{-s}(\sum_{j=1}^NY_j\Wick{Z_iZ_j}) \bigg \|_{L_T^pL^{1+s}}^p
		\\&\quad+\bigg \|\Lambda^{-s}(Y_i\sum_{j=1}^N\Wick{Z_j^2}) \bigg \|_{L_T^pL^{1+s}}^p+\bigg \|\sum_{j=1}^N\Wick{Z_iZ_j^2} \bigg \|_{L_T^pH^{-s}}^p\bigg]
	\end{align*}
where we used the elementary inequality $|A+B|^p\leq (1+\delta)A^p+C_{\delta,p}B^p$.
For fixed $T\geq2$  we can choose $p$ large and use the stationarity of $Y_i$ to deduce that the LHS can absorb the first term on the RHS.

Taking expectation on both sides, we note that the quantity in the bracket  has the same expectation as $A_3^pT$. Hence, by Lemma \ref{lem:zmm} we obtain for $\ell\geq1$ with $3\ell s<1$ and $s=4\kappa$
\begin{align*}
	\E\|Y_i\|_{L^\infty}^\ell\lesssim \E\|Y_i\|_{\bC^\kappa}^\ell\lesssim\frac1{N^{\ell}}\E A_3^\ell\lesssim \frac1{N^{\ell/2}},
\end{align*}
which implies the first bound.
The second bound follows by the first bound together with H\"older's inequality applied to the sum over $i$.
\end{proof}

%

Using Lemma \ref{lem:zin} we can improve the bounds in Lemma \ref{lem:zmm}. In particular we have the following results. We introduce the following short hand notation
\begin{align*}
	\Wick{\PPhi^2} & \eqdef \sum_{i=1}^N\Wick{\Phi_i^2}\;,
	\qquad\qquad
	 \Wick{\Phi_1\PPhi^2}\eqdef \sum_{i=1}^N\Wick{\Phi_1\Phi_i^2}\;,
	 \\
	 \Wick{(\PPhi^2)^n} &\eqdef \Wick{\Big(\sum_{i=1}^N\Phi_i^2\Big)^n}\;,
	 \qquad \Wick{\Phi_1(\PPhi^2)^n}\eqdef \Wick{\Phi_1\Big(\sum_{i=1}^N\Phi_i^2\Big)^n}
\end{align*}
where $\sum_{i=1}^N\Wick{\Phi_1\Phi_i^2}$ is defined as the sum in the brackets of \eqref{eq:22} with $i=1$. Here $\Wick{\Phi_1(\sum_{i=1}^N\Phi_i^2)^n}$ is defined similarly as in \eqref{obn} with $\Phi_1, \Phi_i$ replaced by $(Y_1+Z_1)$ and $Y_i+Z_i$, respectively, where the products between different $Z_i, Z_j$ are the Wick products defined in \eqref{e:Zik}-\eqref{e:ZiZn}.

\bp\label{prop1}
Let $\mm$ be as in Lemma \ref{th:m2}.
For $\ell_i\geq0$, $n_i\in\mN$, $i=1,\dots,n$,  $n\in\mN$
and $\kappa>0$, one has 
\begin{equs}
	\E\Big(\prod_{i=1}^n \Big\|\Wick{(\PPhi^2)^{n_i}}\Big\|_{H^{-\kappa}}^{\ell_i}\Big) & \lesssim N^{\frac12\sum_{i=1}^nn_i\ell_i},
\\
	\E\Big(\prod_{i=1}^n \Big\|\Wick{\Phi_1(\PPhi^2)^{n_i}}\Big\|_{H^{-\kappa}}^{\ell_i}\Big) &\lesssim N^{\frac12\sum_{i=1}^nn_i\ell_i},
\end{equs}
where the proportional constants are independent of $N$.
\ep

\begin{proof}
It suffices to prove that for $n\in\mN$, $\ell\geq0$ with $48\ell n\kappa <1$
\begin{align}\label{est:phi2}
	\E\Big( \Big\|\Wick{(\PPhi^2)^n}\Big\|_{H^{-\kappa}}^{\ell}\Big)\lesssim N^{n\ell/2},
\end{align}	
and for $\ell\geq0$ with $96n\ell\kappa<1$
\begin{align}\label{est:phi3}
	\E\Big( \Big\|\Wick{\Phi_1(\PPhi^2)^n}\Big\|_{H^{-\kappa}}^{\ell}\Big)\lesssim N^{n\ell/2}.
\end{align}	
The general results stated in the proposition then follow from H\"older's inequality. For general $\kappa>0$ and $\ell_i$ we can always find $0<\kappa'<\kappa$ such that $\sum_{i}96\ell_i\kappa' n_i<1$. Since $\|f\|_{H^{-\kappa}}\leq \|f\|_{H^{-\kappa'}}$ the  results hold with general $\kappa$.

By definition of $\Wick{(\PPhi^2)^n}$ we want to estimate the following general terms for $0\leq k\leq n, 0\leq m\leq n-k$
\begin{align*}
	&\Big\|\Big(\sum_{i=1}^NY_i^2\Big)^k\Wick{\Big(\sum_{j=1}^NY_jZ_j\Big)^m\Big(\sum_{\ell=1}^NZ_\ell^2\Big)^{n-k-m}}\Big\|_{H^{-\kappa}}
	\\ &\lesssim 	\Big\|\Big(\sum_{i=1}^NY_i^2\Big)^k\Big\|_{\bC^{2\kappa}}
	\;\Big\|\Wick{\Big(\sum_{j=1}^NY_jZ_j\Big)^m\Big(\sum_{\ell=1}^NZ_\ell^2\Big)^{n-k-m}}\Big\|_{H^{-\kappa}}.
\end{align*}
For the first factor on the RHS, we use  Lemma \ref{lem:zin} and Lemma \ref{lem:multi} to have for $\ell>1$ with $48\kappa\ell k<1$
\begin{align}\label{est:yn}
	\E\Big\|\Big(\sum_{i=1}^NY_i^2\Big)^k\Big\|_{\bC^{2\kappa}}^\ell\lesssim \E\Big\|\sum_{i=1}^NY_i^2\Big\|_{\bC^{2\kappa}}^{\ell k}\lesssim 1.
\end{align}
Turning to the other factor,
for the case $m=0$ we also have
\begin{align*}
	\E\Big\|\Wick{\Big(\sum_{\ell=1}^NZ_\ell^2\Big)^{n-k}}\Big\|_{H^{-\kappa}}^2 & =\E\Big\<\Lambda^{-\kappa}\Wick{\Big(\sum_{\ell=1}^NZ_\ell^2\Big)^{n-k}}\; ,\;\Lambda^{-\kappa}\Wick{\Big(\sum_{\ell=1}^NZ_\ell^2\Big)^{n-k}}\Big\>
	\\
	&=\sum_{\substack{\ell_i=1 \\ i=1,\dots,n-k}}^N\sum_{\substack{p_j=1 \\ j=1,\dots,n-k}}^N\E\Big\<\Lambda^{-\kappa}\Wick{\prod_{i=1}^{n-k}Z_{\ell_i}^2},\Lambda^{-\kappa}\Wick{\prod_{j=1}^{n-k}Z_{p_j}^2}\Big\>.
\end{align*}
Now we have $2(n-k)$ indices summing from $1$ to $N$ and these indices $\ell_i$ and $\ell_j$ might be different. If $\{\ell_1,\dots,\ell_{n-k},p_1,\dots,p_{n-k}\}$ involves more than $n-k+1$ different elements from $\{1,\dots,N\}$, we have
\begin{align*}
	\E\Big\<\Lambda^{-\kappa}\Wick{\prod_{i=1}^{n-k}Z_{\ell_i}^2},\Lambda^{-\kappa}\Wick{\prod_{j=1}^{n-k}Z_{p_j}^2}\Big\>=0.
\end{align*}
Hence, we  have at most  $n-k$ distinct indices summing from $1$ to $N$, and thus by Lemma \ref{lem:z} we have
	\begin{align}\label{est:zn}
		\E\Big\|\Wick{\Big(\sum_{\ell=1}^NZ_\ell^2\Big)^{n-k}}\Big\|_{H^{-\kappa}}^2\lesssim N^{n-k}.
\end{align}
Using Gaussian hypercontractivity and the fact that $\Wick{\Big(\sum_{\ell=1}^NZ_\ell^2\Big)^{n-k}}$ is a random variable with finite Wiener chaos decomposition, we have for $\ell\geq1$
	\begin{align*}
	\E\Big\|\Wick{\Big(\sum_{\ell=1}^NZ_\ell^2\Big)^{n-k}}\Big\|_{H^{-\kappa}}^{\ell}\lesssim N^{(n-k)\ell/2}.
\end{align*}
For general $m>0$ we have
\begin{align*}
	\Wick{\Big(\sum_{j=1}^NY_jZ_j\Big)^m\Big(\sum_{\ell=1}^NZ_\ell^2\Big)^{n-k-m}}
	=\sum_{\substack{j_i=1 \\ i=1,\dots,m}}^N\prod_{i=1}^m Y_{j_i}\Wick{\prod_{i=1}^mZ_{j_i}\Big(\sum_{\ell=1}^NZ_\ell^2\Big)^{n-k-m}}.
\end{align*}
Using Lemma \ref{lem:multi} we find that
 the $H^{-\kappa}$-norm of the above term is bounded by
\begin{align*}
	\sum_{\substack{j_i=1 \\ i=1,\dots,m}}^N\prod_{i=1}^m\|Y_{j_i}\|_{\bC^{2\kappa}}\Big\|\Wick{\prod_{i=1}^mZ_{j_i}\Big(\sum_{\ell=1}^NZ_\ell^2\Big)^{n-k-m}}\Big\|_{H^{-\kappa}}.
\end{align*}
Using H\"older's inequality w.r.t. the sum over $j_i$ we get
\begin{align*}
	\Big(\sum_{j=1}^N\|Y_{j}\|_{\bC^{2\kappa}}^2\Big)^{m/2}\Big(\sum_{\substack{j_i=1\\i=1,\dots,m}}^N\Big\|\Wick{\prod_{i=1}^mZ_{j_i}\Big(\sum_{\ell=1}^NZ_\ell^2\Big)^{n-k-m}}\Big\|_{H^{-\kappa}}^2\Big)^{1/2}.
\end{align*}
Using Lemma \ref{lem:zin} we find for $\ell\geq1$ with $24\kappa\ell m<1$
\begin{align}\label{est:yn1}
	\E\Big(\sum_{j=1}^N\|Y_{j}\|_{\bC^{2\kappa}}^2\Big)^{m\ell/2}\lesssim1.
\end{align}
We then consider the second term:
\begin{align*}
	&\sum_{\substack{j_i=1 \\ i=1,\dots,m}}^N\Big\|\Wick{\prod_{i=1}^mZ_{j_i}\Big(\sum_{\ell=1}^NZ_\ell^2\Big)^{n-k-m}}\Big\|_{H^{-\kappa}}^2
	\\&=\sum_{\stackrel{\ell_p,m_q=1}{p,q=1,\dots,n-k-m}}^N\sum_{\stackrel{j_i=1}{i=1,\dots,m}}^N\Big\<\Lambda^{-\kappa}\Wick{\prod_{i=1}^mZ_{j_i}\prod_{p=1}^{n-k-m}Z_{\ell_p}^2} \;,\; \Lambda^{-\kappa}\Wick{\prod_{i=1}^mZ_{j_i}\prod_{q=1}^{n-k-m}Z_{m_q}^2}\Big\>.
\end{align*}
We have $2n-2k-m$ indices being summed from $1$ to $N$.   For fixed $j_1,\dots,j_m$, if the set
$$\{\ell_1,\dots,\ell_{n-k-m},m_1,\dots,m_{n-k-m}\}$$ does not include $j_i$ ($i=1,\dots,m$) but involves more than $n-k-m+1$ different indices, we find
\begin{align*}
	\E\Big\<\Lambda^{-\kappa}\Wick{\prod_{i=1}^mZ_{j_i}\prod_{p=1}^{n-k-m}Z_{\ell_p}^2},\Lambda^{-\kappa}\Wick{\prod_{i=1}^mZ_{j_i}\prod_{q=1}^{n-k-m}Z_{m_q}^2}\Big\>=0.
\end{align*}
Hence, we have to take sum for $n-k$ different indices, which combined with Lemma \ref{lem:z} implies
\begin{align*}
	\sum_{j_i=1,i=1,\dots,m}^N\E\Big\|\Wick{\prod_{i=1}^mZ_{j_i}\Big(\sum_{\ell=1}^NZ_\ell^2\Big)^{n-k-m}}\Big\|_{H^{-\kappa}}^2\lesssim N^{n-k}.
\end{align*}
Using Gaussian hypercontractivity we have for $\ell\geq1$
\begin{align}\label{est:zn2}
	\E\Big(\sum_{j_i=1,i=1,\dots,m}^N\Big\|\Wick{\prod_{i=1}^mZ_{j_i}\Big(\sum_{\ell=1}^NZ_\ell^2\Big)^{n-k-m}}\Big\|_{H^{-\kappa}}^2\Big)^{\ell/2}\lesssim N^{(n-k)\ell/2}.
\end{align}
Combining \eqref{est:yn}-\eqref{est:zn2} and applying H\"older's inequality, \eqref{est:phi2} follows, which implies the first bound in the proposition.

In the following we prove \eqref{est:phi3}.  Since $\Phi_1=Y_1+Z_1$ we can write
\begin{align*}
	\Wick{\Phi_1(\PPhi^2)^n}=Y_1\Wick{(\PPhi^2)^n}+\Wick{Z_1(\PPhi^2)^n}.
\end{align*}	
For the first term on the RHS we use Lemma \ref{lem:zin} and \eqref{est:phi2} to have for $96n\ell \kappa<1$
\begin{align*}
	\E\|Y_1\Wick{(\PPhi^2)^n}\|_{H^{-\kappa}}^\ell
	&\lesssim\E \|Y_1\|_{\bC^{2\kappa}}^\ell\|\Wick{(\PPhi^2)^n}\|_{H^{-\kappa}}^\ell
	\\
	&\lesssim\Big(\E \|Y_1\|_{\bC^{2\kappa}}^{2\ell}\Big)^{1/2}\Big(\E\|\Wick{(\PPhi^2)^n}\|_{H^{-\kappa}}^{2\ell}\Big)^{1/2}\lesssim N^{\frac{n\ell}2}.
\end{align*}
For the second term we use the defintion of
$\Wick{(\PPhi^2)^n}$  to estimate the following general terms for $0\leq k\leq n, 0\leq m\leq n-k$
\begin{align*}
	&\Big\|\Big(\sum_{i=1}^NY_i^2\Big)^k\Wick{Z_1\Big(\sum_{j=1}^NY_jZ_j\Big)^m\Big(\sum_{\ell=1}^NZ_\ell^2\Big)^{n-k-m}}\Big\|_{H^{-\kappa}}
	\\
	&\lesssim 	\Big\|\Big(\sum_{i=1}^NY_i^2\Big)^k\Big\|_{\bC^{2\kappa}}\Big\|\Wick{Z_1\Big(\sum_{j=1}^NY_jZ_j\Big)^m\Big(\sum_{\ell=1}^NZ_\ell^2\Big)^{n-k-m}}\Big\|_{H^{-\kappa}}.
\end{align*}
Compared to the estimates for the corresponding terms  to derive \eqref{est:phi2}, we have extra $Z_1$ in each Wick product involving $Z$ part. All the estimates follow in the exactly the same way.
Hence, \eqref{est:phi3} follows
which implies the second bound
in the proposition.
	\end{proof}

The moments estimates obtained in Proposition \ref{prop1} are important in bounding various error terms in the following sections.
The proof given above has a flavor of mean field limit technique; in particular, the use of Hilbert spaces $H^{-\kappa}$ is crucial (see \cite[Sec.~1 and Sec.~4]{SSZZ2d} for more motivation).
We will also use the following shorthand notation for the norms appeared in  Proposition \ref{prop1}: for $i\in\mN$
\begin{equ}\label{no:B}
	B_{2i}\eqdef \|\Wick{(\PPhi^2)^i}\|_{H^{-\kappa}},\qquad B_{2i-1}\eqdef \|\Wick{\Phi_1(\PPhi^2)^{i-1}}\|_{H^{-\kappa}}.
\end{equ}
Here the definition of $B_\cdot$ depends on $\kappa$ and we could apply Proposition \ref{prop1}.

\section{Integration by parts and recursive formula}
\label{sec:IBP}

In this section we will employ
 Dyson--Schwinger equations, i.e.
integration by parts formula (``IBP'' for short in the sequel) to derive the recursive equations of the $k$-point functions for the observables introduced in \eqref{ob}. The Dyson-Schwinger equations provide a powerful tool to relate higher-order correlation functions to lower-order ones. By applying IBP to the $k$-point functions of the observables, we can derive the recursive equations governing the behavior of the $k$-point functions.
These recursive equations allow us to systematically analyze the large $N$ limit of the observables in Section \ref{sec:large}.
To this end, we first recall the Dyson--Schwinger equations for the $O(N)$ model on the lattice $\Lambda_\eps\eqdef\eps \mZ^2\cap \mT^2$ 
where $\eps=2^{-M}$ with $M\in \mN$.

We start with Gibbs measures $(\nu_{\eps})_{\eps}$ on $\Lambda_{\eps}$ given by
\begin{align}\label{eq:Phiin}
	\dif \nu_{\eps}^N\varpropto\exp\Big\{-\eps^2\sum_{\Lambda_{\eps}}\Big[\frac{1}{4N}\Big(\sum_{i=1}^N\Phi_i^2\Big)^2
	+\frac12\Big(-\frac{N+2}N a^{\eps}+\m\Big)\sum_{i=1}^N\Phi_i^2+\frac12\sum_{i=1}^N|\nabla_\eps\Phi_i|^2\Big]\Big\}\prod_{x\in \Lambda_{\eps}}\dif \Phi(x),\end{align}
where $$\nabla_\eps f(x)=\Big(\frac{f(x+\eps e_i)-f(x)}{\eps}\Big)_{i=1,2} $$ denotes the discrete gradient and $a^{\eps}$ are renormalization constants defined below. Here $(e_i)_{i=1,2}$ is the canonical basis in $\mR^2$. We write
$$\Delta_\eps f(x)=\eps^{-2}(f(x+\eps e_i)+f(x-\eps e_i)-2f(x)), \quad  x\in \Lambda_\eps$$
as the discrete Laplacian on $\Lambda_{\eps}$ and
$\sL_\eps:=\p_t+\m-\Delta_\eps$.
We also use
$\Phi_{\eps}$ to denote the random  distribution on $\Lambda_\eps$ with $\Phi_\eps=(\Phi_{i,\eps})_{i=1}^N=^d\nu_\eps^N$.

Let   $C_{\eps}(x) \eqdef (\m - \Delta_\eps )^{-1}(x)$ be the Green's function 
where $\Delta_{\eps}$
is the discrete Laplacian on $\Lambda_{\eps}$.
Choosing $a^{\eps}=C_{\eps}(0)$ as the (discrete) Wick constant, 
we recall
\begin{align*}
	\Wick{\Phi_{i,\eps}\PPhi_{\eps}^2}&\eqdef\Phi_{i,\eps}\PPhi_{\eps}^2-a^{\eps}(N+2)\Phi_{i,\eps}\eqdef \Phi_{i,\eps}\Big(\sum_{j=1}^N\Phi_{j,\eps}^2\Big)-a^{\eps}(N+2)\Phi_{i,\eps},\\
	\Wick{\PPhi_{\eps}^2}&\eqdef\sum_{i=1}^N\Wick{\Phi_{i,\eps}^2}\eqdef\sum_{i=1}^N\Phi_{i,\eps}^2-N a^{\eps}.
\end{align*}
Similarly we  also define  $\Wick{(\PPhi_{\eps}^2)^2}$ and $\Wick{(\PPhi_{\eps}^2)^n}, n\geq 2$ as in \eqref{ob1}-\eqref{obn} with the RHS $Y_i, Z_i$ replaced by the related discrete version $Y_{i,\eps}$ and $Z_{i,\eps}$, which satisfy the discrete version of equations \eqref{eq:li1} and \eqref{eq:22} on $\Lambda_\eps$, i.e. for $i=1,\dots,N$
\begin{align*}
	\sL_\eps Z_{i,\eps}=\sqrt 2 \,\xi_{i,\eps},
\end{align*}
and
\begin{align*}
	\sL_\eps Y_{i,\eps}=-\frac1N\sum_{j=1}^N(Y_{j,\eps}^2Y_{i,\eps}+Y_{j,\eps}^2Z_{i,\eps}+2Y_{j,\eps}Y_{i,\eps}Z_{j,\eps}+2Y_{j,\eps}\Wick{Z_{i,\eps}Z_{j,\eps}}+\Wick{Z_{j,\eps}^2}Y_{i,\eps}+\Wick{Z_{i,\eps}Z_{j,\eps}^2}).
\end{align*}
Here $\xi_{i,\eps}$ is a discrete approximation of a space-time white noise $\xi_i$ on $\mR^+\times \mT^2$  constructed as follows:
$$\xi_{i,\eps}(t,x):=\eps^{-2}\<\xi_i(t,\cdot),1_{|\cdot-x|_\infty\leq \eps/2}\>,
\qquad (t,x)\in \mR^+\times \Lambda_{\eps},$$
where $|x|_\infty=|x_1|\vee |x_2|$ for $x=(x_1,x_2)$, and $\Wick{Z_{i,\eps}Z_{j,\eps}}, \Wick{Z_{j,\eps}^2}$ and $\Wick{Z_{i,\eps}Z_{j,\eps}^2}$ are defined similarly as in \eqref{e:wick-tilde}
(with a minor abuse of notation that $Z_{i,\eps}$ denotes the discretized $Z_i$ here and below while  in Section \ref{s:uni} $Z_{i,\eps}$ was smooth approximation).

Recall the Dyson--Schwinger equations (IBP) for $\Phi_\eps$
\begin{align}\label{eq:ibp}
	\E \Big(\frac{\delta F (\Phi_\eps)}{\delta \Phi_{1,\eps}(x)}\Big)
	= \E \Big( (\m - \Delta_\eps ) \Phi_{1,\eps}(x) F(\Phi_\eps)\Big)
	+ \frac{1}{N}  \E \Big( F(\Phi_\eps) \Wick{\Phi_{1,\eps}(x) \PPhi_\eps(x)^2} \Big),
\end{align}
where
$$\frac{\delta F (\Phi_{\eps})}{\delta \Phi_{1,\eps}(x)}=\lim_{\eta\to0}\frac1\eta(F(\Phi_{1,\eps}+\eta \frac{e_x}{\eps^2})-F(\Phi_{1,\eps}))$$
with $e_x:\Lambda_{\eps}\to [0,1]$, $e_x(x)=1$ and $e_x(y)=0$ for $y\neq x$.

We
write \eqref{eq:ibp} in terms of Green's function $C_{\eps}$:
\begin{equation}\label{e:IBP-CN}
	\int_{\Lambda_{\eps}} C_\eps(x-z ) \E \Big(\frac{\delta F (\Phi_\eps)}{\delta \Phi_{1,\eps}(z)}\Big) \dif z
	= \E \Big( \Phi_{1,\eps}(x) F(\Phi_\eps)\Big)
	+ \frac{1}{N}  \int_{\Lambda_\eps} C_\eps(x-z)
	\E \Big( F(\Phi_\eps) \Wick{\Phi_{1,\eps} \PPhi_\eps^2}(z) \Big) \dif z
\end{equation}
for any $x\in {\Lambda_{\eps}}$   and we write $\int_{\Lambda_{\eps}}f(z)\dif z\eqdef \eps^2\sum_{z\in \Lambda_{\eps}}f(z)$. We will also use the  shorthand notation
$$\cI_\eps f(x)\eqdef\int_{\Lambda_\eps} C_\eps(x-z )f(z)\dif z.$$
%
%
%

We first prove two useful results using IBP \eqref{e:IBP-CN}.
The first one is Lemma~\ref{lem:IBP}, which allows us to
rewrite correlation functions involving
certain non-$O(N)$-invariant objects of the form
$\Wick{\Phi_{1,\eps}(\PPhi^2)^{m}_\eps}$
into correlation functions with only $O(N)$-invariant observables,
up to certain error terms.
This is then used in Lemma~\ref{lem:32},
where we derive important finite $N$ recursive formulas for
correlation functions of $O(N)$-invariant observables, up to error terms.

\bl\label{lem:IBP}
Consider for $m_1, m_2, n_i\in \mN_0$, $x_1, x_2, y_i\in \Lambda_\eps$, $i=1,\dots,k$
\begin{align}\label{Ib:phi}
	F(\Phi_\eps)=\Wick{ \Phi_{1,\eps}(\PPhi^2_\eps)^{m_1}}(x_1)\;
	\Wick{\Phi_{1,\eps}(\PPhi^2_\eps)^{m_2}}(x_2)\;
	\prod_{i=1}^k \Wick{ (\PPhi^2_\eps)^{n_i}}(y_i)
\end{align}
and
\begin{align*}
	G(\Phi_\eps)=\frac{N+2m_2}{N}C_\eps(x_1-x_2)\Wick{(\PPhi^2_\eps)^{m_1}}(x_1)\Wick{(\PPhi^2_\eps)^{m_2}}(x_2)\prod_{i=1}^k \Wick{ (\PPhi^2_\eps)^{n_i}}(y_i).
\end{align*}
It holds that
\begin{align*}
	\E(F(\Phi_\eps))=\E(G(\Phi_\eps))+O_N^\eps
\end{align*}
where $O_N^\eps$ is given in the proof.
\el
\begin{proof}
	We omit $\eps$ to simplify notation.
	Applying IBP \eqref{e:IBP-CN} to the following test function
	\begin{align*}
		F_1(\Phi)=\Wick{ (\PPhi^2)^{m_1}}(x_1)\Wick{\Phi_{1}(\PPhi^2)^{m_2}}(x_2)\prod_{i=1}^k \Wick{ (\PPhi^2)^{n_i}}(y_i),
	\end{align*}
	with $x$ in \eqref{e:IBP-CN} given by $x_1$, we obtain
\begin{equs}
		{}&\frac{N+2m_2}{N}C(x_1-x_2)
		\,\E\Big(\Wick{(\PPhi^2)^{m_1}}(x_1)\,\Wick{(\PPhi^2)^{m_2}}(x_2)\,\prod_{i=1}^k \Wick{ (\PPhi^2)^{n_i}}(y_i)\Big)
		\\&+2\sum_{j=1}^kn_jC(x_1-y_j)\E\Big(\Wick{ (\PPhi^2)^{m_1}}(x_1)\Wick{\Phi_{1}(\PPhi^2)^{m_2}}(x_2)\Wick{\Phi_1(\PPhi^2)^{n_j-1}}(y_j)\prod_{i=1,i\neq j}^k \Wick{ (\PPhi^2)^{n_i}}(y_i)\Big)
		\\&=\E(F(\Phi))+\frac1N\E\Big(F_1(\Phi)\cI(\Wick{\Phi_1\PPhi^2})(x_1)\Big). \label{e:defO}
	\end{equs}
Here, the first term on the LHS arises from
	$$\frac{\delta \Wick{\Phi_1(\PPhi^2)^{m_2}}}{\delta\Phi_1(x_2)}=\Big(\Wick{(\PPhi^2)^{m_2}}+2m_2\Wick{\Phi_1^2(\PPhi^2)^{m_2-1}}\Big)(x_2)\frac{e_{x_2}}{\eps^2},$$
	and using $O(N)$ symmetry to replace $\Wick{\Phi_1^2(\PPhi^2)^{m_2-1}}$ by $\frac1N\Wick{(\PPhi^2)^{m_2}}$.
The second term on the LHS arises from
	\begin{align}\label{ibp:1}
		\frac{\delta \Big(\Wick{\prod_{i=1}^k(\PPhi^2)^{n_i}}\Big)}{\delta\Phi_1(y_j)}=2n_j\Wick{\Phi_1(\PPhi^2)^{n_j-1}}(y_j)\Wick{\!\!\prod_{i=1,i\neq j}^k(\PPhi^2)^{n_i}}\frac{e_{y_j}}{\eps^2}
	\end{align}
and we used $$\Wick{\Phi_1(\PPhi^2)^{m_1}}=\Phi_1\Wick{(\PPhi^2)^{m_1}}-2m_1C(0)\Wick{\Phi_1(\PPhi^2)^{m_1-1}}$$
	 to absorb the term from $\frac{\delta \Wick{(\PPhi^2)^{m_1}}}{\delta\Phi_1}$.
	Notice that the first term on the LHS precisely gives $G(\Phi_\eps)$.
Defining 	
the other two terms (i.e. the 2nd term on LHS and the 2nd term on RHS) in \eqref{e:defO} as $O_N^\eps$, the result follows.
\end{proof}

 According to Proposition \ref{prop1}, the terms defined as $O_N^\eps$ are smaller than $F(\Phi_\eps)$ and $G(\Phi_\eps)$ by a factor $O(\frac{1}{\sqrt{N}})$. The above lemma shows that under expectation, the main contribution
to  $F(\Phi_\eps)$ is obtained by replacing the two incidences of $\Phi_1$ at $x_1$ and $x_2$ by $C_\eps(x_1-x_2)$.

Furthermore, the terms defined as $O_N^\eps$ have the same structure as $F(\Phi_\eps)$, namely, each of them contains exactly two incidences of $\Phi_1$.

In the following we consider
$$f_{\mathbf{n},k,\eps}^{N}(y_1,\dots,y_{k})\eqdef\frac1{N^{\sum_{i=1}^kn_i/2}}\E\Big(\prod_{i=1}^k \Wick{ (\PPhi^2_\eps)^{n_i}}(y_i)\Big),$$
for $\mathbf{n}=(n_1,\dots,n_k)$ with $n_i\in \mN\cup \{0\}$ and $y_i\in \Lambda_\eps$, $i=1,\dots,k$. We also set $\Wick{ (\PPhi^2_\eps)^{0}}=1$.

\bl \label{lem:32}
For $\mathbf{n}=(n_1,\dots,n_k)\in \mN^k$, $y_i\in \Lambda_\eps$, $i=1,\dots,k$, one has
\begin{equation}\label{eq:inductive}
	\aligned
	&f_{\mathbf{n},k,\eps}^{N}(y_1,\dots,y_k)+\frac{N+2}N\int_{\Lambda_\eps} C^2_\eps(y_1-z)	f_{\hat{\mathbf{n}},k+1,\eps}^{N}(z,y_1,\dots,y_k)\dif z
	\\
	=&\sum_{j=2}^k2n_j\frac{N+2(n_j-1)}NC^2_\eps(y_1-y_j)f_{\widetilde{\mathbf{n}}_j,k,\eps}^{N}(y_1,\dots,y_k)
	+N^{-(\sum_{i=1}^kn_i/2)+1}Q_{N,\mathbf{n}}^\eps
	\endaligned
\end{equation}
with $\hat{\mathbf{n}}=(1,n_1-1,n_2,\dots,n_k)$ and $\widetilde{\mathbf{n}}_j=(n_1-1,n_2,\dots,n_{j-1},n_j-1,n_{j+1},\dots,n_k)$ and $Q_{N,\mathbf{n}}^\eps$ given in \eqref{de:Q} below.
\el
\begin{proof}
We omit $\eps$ as above.
Choosing test function in  IBP \eqref{e:IBP-CN} as
	\begin{equ}
		F(\Phi)=\Wick{ \Phi_1(\PPhi^2)^{n_1-1}}(y_1) \prod_{i=2}^k \Wick{ (\PPhi^2)^{n_i}}(y_i),
	\end{equ}
we have
\begin{equs}
		&\sum_{j=2}^k2n_jC(y_1-y_j)\E\Big(\Wick{ \Phi_1(\PPhi^2)^{n_1-1}}(y_1)\Wick{\Phi_1(\PPhi^2)^{n_j-1}}(y_j)\prod_{i=2,i\neq j}^k \Wick{ (\PPhi^2)^{n_i}}(y_i)\Big)
		\\
		&=\E\Big(\Wick{ \Phi_1^2(\PPhi^2)^{n_1-1}}(y_1) \prod_{i=2}^k \Wick{ (\PPhi^2)^{n_i}}(y_i)\Big)+\frac1N\E\Big(F(\Phi)\cI(\Wick{\Phi_1\PPhi^2})(y_1)\Big)\label{e:same}
		\\
		&=\frac1N\E\Big(\prod_{i=1}^k \Wick{ (\PPhi^2)^{n_i}}(y_i)\Big)+\frac1N\E\Big(F(\Phi)\cI(\Wick{\Phi_1\PPhi^2})(y_1)\Big),
\end{equs}
	where we used \eqref{ibp:1} to have the first term and we also used
	$$
	\Wick{ \Phi_1^2(\PPhi^2)^{n_1-1}}=\Phi_1\Wick{ \Phi_1(\PPhi^2)^{n_1-1}}-\Big(C(0)\Wick{ (\PPhi^2)^{n_1-1}}+2(n_1-1)C(0)\Wick{ \Phi_1^2(\PPhi^2)^{n_1-2}}\Big)
	$$
	to absorb the term from $\frac{\delta}{\delta\Phi_1} (\Wick{ \Phi_1(\PPhi^2)^{n_1-1}})$.
In view of  Proposition \ref{prop1},
all the terms  on the LHS and the RHS of \eqref{e:same} are of the same order.
		
For 	each term in \eqref{e:same} having two incidences of $\Phi_1$ (namely the term in the first line  and the last term in the last line),
we can apply  Lemma~\ref{lem:IBP}, which  replaces the two $\Phi_1$ by  the Green's function $C$,  and some error terms. Hence we obtain
	\begin{align*}
		&\frac1N\E\Big(\prod_{i=1}^k \Wick{ (\PPhi^2)^{n_i}}(y_i)\Big)
		\\
		&=\sum_{j=2}^k2n_j\frac{N+2(n_j-1)}NC^2(y_1-y_j)\E\Big(\Wick{ (\PPhi^2)^{n_1-1}}(y_1)\Wick{(\PPhi^2)^{n_j-1}}(y_j)\prod_{i=2,i\neq j}^k \Wick{ (\PPhi^2)^{n_i}}(y_i)\Big)
		\\
		&\quad -\frac1N\frac{N+2}N\int C^2(y_1-z)\E\Big(\Wick{ (\PPhi^2)^{n_1-1}}(y_1) \prod_{i=2}^k \Wick{ (\PPhi^2)^{n_i}}(y_i)\Wick{\PPhi^2}(z)\Big)\dif z+Q_{N,\mathbf{n}}^\eps,
	\end{align*}
	where 
	\begin{equ}\label{de:Q}
		Q_{N,\mathbf{n}}^\eps=\sum_{i=1}^4 Q_{N,\mathbf{n}}^{i,\eps},
	\end{equ}
	with  (remark that $Q_{N,\mathbf{n}}^{2,\eps}$ is from the first term in \eqref{e:same} and $Q_{N,\mathbf{n}}^{3,\eps}$ is from the last term in \eqref{e:same}, and they turn out to be the same)
	\begin{equs}
		Q_{N,\mathbf{n}}^{1,\eps}&=\sum_{j,m=2,j\neq m}^k4n_jn_mC(y_1-y_j)C(y_1-y_m)
		\E\Big(\Wick{ (\PPhi^2)^{n_1-1}}(y_1)
		\\&\qquad\qquad\qquad \times\Wick{\Phi_1(\PPhi^2)^{n_j-1}}(y_j)\Wick{\Phi_1 (\PPhi^2)^{n_m-1}}(y_m)\prod_{i=2,i\neq j,m}^k \Wick{ (\PPhi^2)^{n_i}}(y_i)\Big),
		\\
		Q_{N,\mathbf{n}}^{2,\eps}=Q_{N,\mathbf{n}}^{3,\eps}&=\frac1N\sum_{j=2}^k2n_jC(y_1-y_j)
		\E\Big(\Wick{ (\PPhi^2)^{n_1-1}}(y_1)\cI(\Wick{\Phi_1\PPhi^2})(y_1)
		\\&\qquad\qquad\qquad\times\Wick{\Phi_1(\PPhi^2)^{n_j-1}}(y_j)\prod_{i=2,i\neq j}^k \Wick{ (\PPhi^2)^{n_i}}(y_i)\Big),
		\\
		Q_{N,\mathbf{n}}^{4,\eps}&=\frac1{N^2}
		\E\Big(\Wick{ (\PPhi^2)^{n_1-1}}(y_1)\cI(\Wick{\Phi_1\PPhi^2})(y_1)^2\prod_{i=2}^k \Wick{ (\PPhi^2)^{n_i}}(y_i)\Big).
	\end{equs}
	Hence, \eqref{eq:inductive} follows by multiplying both sides by $N^{-(\sum_{i=1}^kn_i/2)+1}$.
\end{proof}

Now we pass the above results to the continuum limit $\eps\to 0$.
Since $\Phi_\eps$ satisfies discrete version of the SPDEs \eqref{eq:Phi2d}, we also decompose $\Phi_\eps=Y_\eps+Z_\eps$ where $Y_\eps=(Y_{i,\eps})$ satisfies discrete version of the equations \eqref{eq:22}. Moreover, the uniform in $N$ estimates Lemma \ref{th:m2} hold for $(Y_{i,\eps})$ and Proposition \ref{prop1} hold for $\Phi_\eps$. More precisely,
Let $\mm$ be as in Lemma \ref{th:m2}. For $\ell_i\geq0$, $n_i\in\mN$, $i=1,\dots,n$, $n\in\mN$
and $\kappa>0$, one has 
\begin{align}\label{uni:d1}
	\E\Big(\prod_{i=1}^n \Big\|\Wick{(\PPhi^2_\eps)^{n_i}}\Big\|_{H^{-\kappa,\eps}}^{\ell_i}\Big)\lesssim N^{\frac12\sum_{i=1}^n{n_i}\ell_i}
\end{align}
and
\begin{align}\label{uni:d2}
	\E\Big(\prod_{i=1}^n \Big\|\Wick{\Phi_{1,\eps}(\PPhi^2)_\eps^{n_i}}\Big\|_{H^{-\kappa,\eps}}^{\ell_i}\Big)\lesssim N^{\frac12\sum_{i=1}^n{n_i}\ell_i}
\end{align}
where the proportional constants are independent of $N$ and $\eps$.
 We refer to \cite{SZZ21}, \cite{SSZZ2d} and \cite{GH18a} for more details.

Furthermore, for fixed $N$, and fixed $t\geq0$, the sequence
$(\cE^\eps\Phi_{{\eps}}(t))_{\eps}=(\cE^\eps\Phi_{{\eps,i}}(t))_{\eps}$ is tight in $(H^{-\kappa})^N$ for $\kappa>0$, where $\cE^\eps$ is the extension operator defined in Appendix \eqref{def:E}. Moreover,  every tight limit $\mu$ is an invariant measure of \eqref{eq:Phi2d}. By \cite[Lemma 5.7]{SSZZ2d} for fixed $N$ the invariant measure of equations \eqref{eq:Phi2d} is unique.  Hence, the whole sequence $(\cE^\eps\Phi_{{\eps}}(t))_{\eps}$ converges to $\Phi(t)$ weakly in $(H^{-\kappa})^N$, where $\Phi$ is the stationary solution to equations \eqref{eq:Phi2d}.
For $k\in\mN$ and every $\mathbf{n}=(n_1,\dots,n_k)\in \mN^k$, we  define the continuous version of $f_{\mathbf{n},k,\eps}^N$
\begin{align*}
	\<f_{\mathbf{n},k}^N,\varphi\>\eqdef\lim_{\eps\to0}\int \E\Big(\prod_{i=1}^k\cE^\eps\frac{1}{ N^{n_i/2}}\Wick{(\PPhi^2_\eps)^{n_i}}(y_i)\Big)\varphi(y_1,\dots,y_k)\prod_{i=1}^k\dif y_i
\end{align*}
for every $\varphi\in \cS(\mT^{2k})$, $i=1,\dots, k$. Since for fixed $N$ the law of $\frac{1}{ N^{n_i/2}}\Wick{(\PPhi^2)^{n_i}}$  are uniquely determined, we have
\begin{align*}
	\<f_{\mathbf{n},k}^N,\otimes_{i=1}^{k}\varphi_i\>= \E\Big(\prod_{i=1}^{k}\Big\<\frac{1}{ N^{n_i/2}}\Wick{(\PPhi^2)^{n_i}},\varphi_i\Big\>\Big),
\end{align*}
for $\varphi_i\in \cS(\mT^2)$, $i=1,\dots, k$.

Now we send $\eps\to0$ on both sides of
the formula \eqref{eq:inductive} obtained in Lemma~\ref{lem:32}. To this end, we introduce the following integral operators:
\begin{equation}\label{def:I}
	\aligned
	\cI f=(\m-\Delta)^{-1}f, \qquad (\cI_1f)(y)=\int C^2(y-z)f(z)\dif z,\quad
	f\in C^\infty(\mT^2),
	\endaligned
\end{equation}
where $C$ is the Green function of $\Delta-\m$.
By Lemma \ref{lem:el} in Appendix we can extend these operators from Sobolev spaces $H^\alpha$ to $H^{\alpha+2-\kappa}$, $\alpha<0, \kappa>0$.

\bl\label{le:fkN} It holds that
for $\mathbf{n}=(n_1,\dots,n_k)$ with $n_i\in \mN$
\begin{equs}[eq:inductive1]
	{}&\<f_{\mathbf{n},k}^{N},\varphi^{\otimes k}\>+\frac{N+2}N\int C^2(y_1-z)\prod_{i=1}^k\varphi_i(y_i)	f_{\hat{\mathbf{n}},k+1}^{N}(\dif z,\dif y_1,\dots,\dif y_k)
	\\=&\sum_{j=2}^k2n_j\frac{N+2(n_j-1)}N\int C^2(y_1-y_j)\prod_{i=1}^k\varphi_i(y_i)f_{\widetilde{\mathbf{n}}_j,k}^{N}(\dif y_1,\dots,\dif y_k)+\frac1{N^{(\sum_{i=1}^kn_i/2)-1}}\<Q_{N,\mathbf{n}},\varphi^{\otimes k}\>,
\end{equs}
with $\hat{\mathbf{n}}=(1,n_1-1,n_2,\dots,n_k)$ and $\widetilde{\mathbf{n}}_j=(n_1-1,n_2,\dots,n_{j-1},n_j-1,n_{j+1},\dots,n_k)$ and $\varphi^{\otimes k}=\otimes_{i=1}^k\varphi_i$ for $\varphi_i\in \cS(\mT^2)$, $i=1,\dots,k$,  where
$Q_{N,\mathbf{n}}=\sum_{i=1}^4 Q_{N,\mathbf{n}}^{i}$
with
\begin{align*}
&\<Q_{N,\mathbf{n}}^{1},\varphi^{\otimes k}\>
\\&=\sum_{\substack{j,m=2 \\ j\neq m}}^k \!\! 4n_jn_m
	\E\Big(\Big\<\Wick{ (\PPhi^2)^{n_1-1}}\cI(\Wick{\Phi_1(\PPhi^2)^{n_j-1}}\varphi_j)\cI(\Wick{\Phi_1 (\PPhi^2)^{n_m-1}}\varphi_m),\varphi_1\Big\> 
\!\!\! \prod_{\substack{i=2 \\ i\notin\{j,m\}}}^k \!\!\! \<\Wick{ (\PPhi^2)^{n_i}},\varphi_i\>\Big),
\\
&\<Q_{N,\mathbf{n}}^{2},\varphi^{\otimes k}\>=\<Q_{N,\mathbf{n}}^{3},\varphi^{\otimes k}\>
\\	&=\frac1N\sum_{j=2}^k2n_j
	\E\Big(\Big\<\Wick{ (\PPhi^2)^{n_1-1}}\cI(\Wick{\Phi_1\PPhi^2})\cI(\Wick{\Phi_1(\PPhi^2)^{n_j-1}}\varphi_j),\varphi_1\Big\>
	\prod_{i=2,i\neq j}^k \<\Wick{ (\PPhi^2)^{n_i}},\varphi_i\>\Big),
\\
	&\<Q_{N,\mathbf{n}}^{4},\varphi^{\otimes k}\>=\frac1{N^2}
	\E\Big(\Big\<\Wick{ (\PPhi^2)^{n_1-1}}(\cI(\Wick{\Phi_1\PPhi^2}))^2,\varphi_1\Big\>\prod_{i=2}^k \<\Wick{ (\PPhi^2)^{n_i}},\varphi_i\>\Big).
\end{align*}
\el
\begin{proof} By  \cite[Lemma 2.24]{MP19}, we know that $\cE^\eps$ is uniformly  bounded in $\eps$ operators from discrete Besov space to the continuous Besov space. 
	Using Lemma \ref{lem:Epro}, Lemma \ref{Cepro} and the fact that Proposition \ref{prop1} also holds in the discrete setting i.e. \eqref{uni:d1} and \eqref{uni:d2}, we have, for fixed $N\in\mN$
	$$\<\cE^\eps Q_{N,\eps}^i,\otimes_{i=1}^k \varphi_i\>\to \<Q_N^i,\otimes_{i=1}^k  \varphi_i\>,\quad \eps\to0.$$
	We also refer to \cite[Section 4.2]{SZZ21} for more details on the proof of the above convergence.
The lemma then follows from Lemma~\ref{lem:32} and sending $\eps\to 0$.
\end{proof}

Later we will take the large $N$ limit on both sides of \eqref{eq:inductive1}.
The following lemma shows the large $N$ behavior of the ``error term'' $Q_{N,\mathbf{n}}$.

\bl\label{lem:wqnn} Let $\m$ be as in Lemma \ref{th:m2}. It holds that
\begin{align*}
	\frac1{N^{(\sum_{i=1}^kn_i/2)-1}}|\<{Q}_{N,\mathbf{n}}, \varphi^{\otimes k}\>|\lesssim \frac1{\sqrt N}
\end{align*}
for $\varphi^{\otimes k}=\otimes_{i=1}^k  \varphi_i$ with $\varphi_i\in\cS(\mT^2)$, $i=1,\dots,k$, where the proportional constant is independent of $N$.
\el
\begin{proof}
	Using  Lemma \ref{lem:emb} and \eqref{eq:sch}  we find that for $\kappa>0$
	\begin{align}\label{es}
		\|\cI(f)\|_{\bC^{1-\kappa}}\lesssim \|f\|_{H^{-\kappa}},
	\end{align}
	which combined with Lemma \ref{lem:multi} leads to
	\begin{align*}
		\Big|\Big\<\Wick{ (\PPhi^2)^{n_1-1}}\cI(\Wick{\Phi_1(\PPhi^2)^{n_j-1}}\varphi_j)\cI(\Wick{\Phi_1 (\PPhi^2)^{n_m-1}}\varphi_m),\varphi_1\Big\>\Big|\lesssim B_{2n_1-2}B_{2n_j-1}B_{2n_m-1},
	\end{align*}
where we recall \eqref{no:B} for the notation $B$.
This implies that
\begin{equ}
\frac1{N^{(\sum_{i=1}^kn_i/2)-1}}|\<Q_{N,\mathbf{n}}^1,\varphi^{\otimes k}\>|\lesssim 
\frac1{N^{(\sum_{i=1}^kn_i/2)-1}}
\E \Big(B_{2n_1-2}B_{2n_j-1}B_{2n_m-1}
\prod_{\substack{i=2 \\ i \notin \{j,m\}}}^kB_{2n_i}\Big).
\end{equ}
By Proposition \ref{prop1},
	\begin{align*}
		\frac1{N^{(\sum_{i=1}^kn_i/2)-1}}|\<Q_{N,\mathbf{n}}^1,\varphi^{\otimes k}\>|\lesssim \frac1{\sqrt N}.
	\end{align*}
	For the first part in $\<Q_{N,\mathbf{n}}^2,\varphi^{\otimes k}\>$
	we also use \eqref{es} and Lemma \ref{lem:multi} to have
	\begin{align*}
		\Big|\Big\<\Wick{ (\PPhi^2)^{n_1-1}}\cI(\Wick{\Phi_1(\PPhi^2)^{n_j-1}}\varphi_j)\cI(\Wick{\Phi_1 \PPhi^2}),\varphi_1\Big\>\Big|\lesssim B_{2n_1-2}B_{2n_j-1}B_{3}.
	\end{align*}
	Hence, we use Proposition \ref{prop1} to have
	\begin{align*}
		&\frac1{N^{(\sum_{i=1}^kn_i/2)-1}}|\<Q_{N,\mathbf{n}}^2,\varphi^{\otimes k}\>|\lesssim \frac1{N^{\sum_{i=1}^kn_i/2}}\E \Big(B_{2n_1-2}B_{2n_j-1}B_{3}\prod_{i=2,i\neq j}^kB_{2n_i}\Big)
		\lesssim\frac1{\sqrt N}.
	\end{align*}
	For the first part in $\<Q_{N,\mathbf{n}}^4,\varphi^{\otimes k}\>$ we use \eqref{es} and Lemma \ref{lem:multi} to have
	\begin{align*}
		\Big|\Big\<\Wick{ (\PPhi^2)^{n_1-1}}(\cI(\Wick{\Phi_1 \PPhi^2}))^2,\varphi_1\Big\>\Big|\lesssim B_{2n_1-2}B_{3}^2.
	\end{align*}
	Hence, by Proposition \ref{prop1},
	\begin{align*}
		&\frac1{N^{(\sum_{i=1}^kn_i/2)-1}}|\<Q_{N,\mathbf{n}}^4,\varphi^{\otimes k}\>|\lesssim \frac1{N^{(\sum_{i=1}^kn_i/2)+1}}\E \Big(B_{2n_1-2}B_{3}^2\prod_{i=2}^kB_{2n_i}\Big)
		\lesssim\frac1{\sqrt N}.
	\end{align*}
\end{proof}

By Proposition \ref{prop1} and Lemma \ref{th:m2}, for fixed $k\in\mN$, $\{f_{\mathbf{n},k}^N\}_N$ are uniformly bounded in $(H^{-\kappa})^k$. By Lemma \ref{le:fkN} and Lemma \ref{lem:wqnn}, every subsequential limit $f_{\mathbf{n},k}$  satisfies the following equation
\begin{equation}\label{eq:inductive2}
	\aligned
	&\<f_{\mathbf{n},k},\varphi^{\otimes k}\>+\int C^2(y_1-z)\prod_{i=1}^k\varphi_i(y_i)	f_{\hat{\mathbf{n}},k+1}(\dif z,\dif y_1,\dots,\dif y_k)
	\\&=\sum_{j=2}^k2n_j\int C^2(y_1-y_j)\prod_{i=1}^k\varphi_i(y_i)f_{\widetilde{\mathbf{n}}_j,k}(\dif y_1,\dots,\dif y_k),
	\endaligned
\end{equation}
with $\hat{\mathbf{n}}=(1,n_1-1,n_2,\dots,n_k)$ and $\widetilde{\mathbf{n}}_j=(n_1-1,n_2,\dots,n_{j-1},n_j-1,n_{j+1},\dots,n_k)$. Intuitively the second term on the LHS of \eqref{eq:inductive2} comes from the interaction between the observables and $\Wick{\Phi_1\PPhi^2}$. 

 \section{Large $N$ limits of observables}\label{sec:large}

 In this section, we focus on analyzing the behavior of the observables \eqref{ob} in the large $N$ limit.  We first obtain the large $N$ limit of the first observables $\frac1{\sqrt N}\Wick{\PPhi^2}$ using the recursive relation derived in \eqref{eq:inductive2} in Section \ref{sec:4.1}. We then derive in Section \ref{sec:4.3} the large $N$ limit of the observables $\frac1{ N^{n/2}}\Wick{(\PPhi^2)^n}$ by identifying the solutions to the recursive equations \eqref{eq:inductive2} as the $k$-point functions of $\WickC{\cQ^n}$.

 \subsection{Large $N$ limit of $\frac1{\sqrt N}\Wick{\PPhi^2}$}\label{sec:4.1}

Let
 $$
 f_{k,\eps}^N(y_1,\dots,y_k)
 \eqdef
 \frac1{N^{k/2}}\E\Big(\prod_{j=1}^{k}\Wick{\PPhi_\eps^2}(y_j)\Big), \quad y_i\in \Lambda_\eps=\eps \mZ^2\cap \mT^2, \quad i=1,\dots,k.
 $$
 Choosing $\mathbf{n}=(1,\dots,1)$ in
 the formula \eqref{eq:inductive} of Lemma~\ref{lem:32}
 we obtain the following result.

 \bl\label{lem:3.1} Let $\mm$ be as in Lemma \ref{th:m2}. It holds that
\begin{equs}\label{eq:fkN}
	f_{k,\eps}^N(y_1,\dots,y_k)&+\int_{\Lambda_{\eps}} C_\eps^2(y_1-z)f_{k,\eps}^N(z,y_2,\dots,y_k)\dif z
	\\&= 2\sum_{m=2}^k C_\eps^2(y_1-y_m) f_{k-2,\eps}^N(y_2,\dots,y_k\backslash\{y_m\})+Q_{N,\eps}
\end{equs}
where $Q_{N,\eps}$ is of order $\frac1{\sqrt N}$.
 \el

Here and in the sequel,  notation such as
$f_{k-2,\eps}^N(y_2,\dots,y_k\backslash\{y_m\})$
is understood in the obvious way, namely 
$f_{k-2,\eps}^N(y_2,\dots,y_{m-1},y_{m+1},\cdots, y_k)$.
 Since for fixed $N$ the law of $\frac{1}{\sqrt N}\Wick{\PPhi^2}$ is uniquely determined, we have
\begin{align*}
	\<f_k^N,\otimes_{i=1}^k\varphi_i\>= \E\prod_{i=1}^k\Big\<\frac{1}{\sqrt N}\Wick{\PPhi^2},\varphi_i\Big\>,
\end{align*}
for $\varphi_i\in \cS(\mT^2)$, $i=1,\dots, k$.


By \eqref{eq:inductive2} 
we have that
every subsequential limit $f_k$ of $f_k^N, N\to\infty$ satisfies the following equations (recall the notation \eqref{def:I})
\begin{equation}\label{eq:fk}
	\aligned
	&\<f_{k},\otimes_{i=1}^k\varphi_i\>+\<\cI_1(f_{k}),\otimes_{i=1}^k\varphi_i\>
	\\&= 2\sum_{m=2}^k \int C^2(y_1-y_m)\varphi_1(y_1)\varphi_m(y_m)\dif y_1\dif y_m\<f_{k-2},\otimes_{i=2,i\neq m}^k\varphi_i\>,
	\endaligned
\end{equation}
where $\cI_1(f_{k})$ only acts on the first variable and $\cI_1(f_{k})=\int C^2(y_1-z)f_k(\dif z,\dif y_2,\dots,\dif y_k)$.
By iteration and Fourier transform, 
it is easy to see that the solution to \eqref{eq:fk} is unique. In the following we give an explicit formula for the solution to \eqref{eq:fk}. To this end, we start with the case that $k$ is odd.

\bl\label{lem:1}  Let $\mm$ be as in Lemma \ref{th:m2}. One has
$f_k\equiv0$ if $k$ is odd.
\el

\begin{proof} 
In the proof we first apply IBP \eqref{e:IBP-CN} to calculate $\E\Wick{\PPhi^2}$, which is a constant by translation invariance. We also omit $\eps$ for notation simplicity.
	Choosing $F(\Phi) = \Phi_1(x)$ in \eqref{e:IBP-CN},
	the LHS of \eqref{e:IBP-CN} only gives a Wick constant which can be absorbed into the RHS, which implies that
	\begin{align}\label{e:phi2}
		0=\E \Big( \Wick{\Phi_1(x)^2}\Big)
		+ \frac{1}{N} \int C(x-z )
		\E \Big(  \Phi_1(x) \Wick{\Phi_1(z) \PPhi(z)^2} \Big) \dif z.
	\end{align}
By symmetry we find the first term on the RHS is
$\frac1N\E ( \Wick{\PPhi(x)^2})$. For the second term on the RHS we  choose
$F(\Phi)= \int C(x-z ) \Wick{\Phi_1(z) \PPhi(z)^2}  \dif z$ in \eqref{e:IBP-CN} to obtain
	\begin{equ}
		(1+\frac2N)\int C^2(x-z)  \E \Big( \Wick{\PPhi(z)^2}\Big) \dif z
		=
		\E \Big( \Phi_1(x) \cI(\Wick{   \Phi_1 \PPhi^2})(x) \Big)
		+
		\frac{1}{N}
		\E \Big(  \cI(\Wick{\Phi_1\PPhi^2})(x)^2 \Big).
	\end{equ}
	Substituting \eqref{e:phi2} into the above equality and using symmetry, we find
	\begin{equs}
		(1+\frac2N)\int C^2(x-z)  \E \Big( \Wick{\PPhi(z)^2}\Big) \dif z
		&= - \E \Big( \Wick{\PPhi(x)^2}\Big)
		+
	\frac{1}{N}
	\E \Big(  \cI(\Wick{\Phi_1\PPhi^2})(x)^2  \Big).
	\end{equs}
	As $\E \big( \Wick{\PPhi(x)^2}\big)$ is independent of $x$, we can write the above equality as
	\begin{equs}\label{Phi2}
		&\E \Big( \Wick{\PPhi(x)^2}\Big)\Big( (1+\frac2N)\int C^2(x-z)  \dif z+1\Big)
		= 	
		\frac{1}{N}
		\E \Big(   \cI (\Wick{\Phi_1\PPhi^2})(x)^2  \Big).
	\end{equs}
Now we  use the extension operator $\cE^\eps$ to extend both sides to the functions on $\mT^2$. Letting $\eps\to0$ and applying \cite[Lemma A.9, Lemma A.10]{SZZ21}, we obtain \eqref{Phi2} holds in the continuum. 
By Proposition \ref{prop1}, the RHS of \eqref{Phi2} is of order $1$. Hence,
\begin{align*}
	|f_1^N|=\frac1{\sqrt N}|\E(\Wick{\PPhi^2})|\lesssim \frac1{\sqrt N}.
\end{align*}
Letting $N\to\infty$, we find that $f_1\equiv0$. Substituting $f_1$ into \eqref{eq:fk}, using uniqueness of the solutions to \eqref{eq:fk} and by induction, the result follows.
\end{proof}

By \eqref{eq:fk} with $k=2$ (see also \cite[Theorem 6.5]{SSZZ2d}) and translation invariance we obtain that $f_2(y_1,y_2)=G(y_1-y_2)$ with $G\in L^2(\mT^2) $ satisfying \eqref{e:G}, or more explicitly
$$\widehat G= \frac{2\widehat{C^2}}{1+\widehat{C^2}},$$
where $\widehat f$ denotes Fourier transform of $f$ \footnote{Compared to \cite[Theorem 6.5]{SSZZ2d} we have extra $\frac12$ in the coefficient of $|\nabla \phi_\eps|^2$ in $\nu_\eps^N$, which makes that there's no extra $2$ in the denominator.}. In the sequel we also write $G(y_1-y_2)=G(y_1,y_2)$. In the following lemma we use $f_2$ to give an explicit formula of general $f_k$.

\bl\label{lem:2}  Let $\m$ be as in Lemma \ref{th:m2}. It holds that
\begin{align}\label{def:fk}
f_{k}(y_1,\dots,y_{k})
= \sum_{\pi}\prod_{j=1}^{k/2}f_2(y_{\pi(2j-1)},y_{\pi(2j)}), \qquad k\in 2\mN, \quad k\geq 4,
\end{align}
where $\pi$ runs through pairing permutations of $\{ 1, \dots, k\}$.
\el
Here, pairing permutations are simply permutations modulo possibly swapping the values of  $\pi(2j-1)$ and $\pi(2j)$ for any $j$.
\begin{proof} 
We claim that $f_k$ given in \eqref{def:fk} satisfies \eqref{eq:fk}.
We prove by induction. 
Note that for $k=2$
\begin{align*}
	f_2(y_1,y_2)=G(y_1-y_2),
\end{align*}
with $G$ satisfying
\begin{align}\label{eq:G}
	G(x-y)=2C^2(x-y)-\int C^2(x-z)G(z-y)\dif z.
\end{align}
Suppose that $f_{k}$ given by \eqref{def:fk} satisfies \eqref{eq:fk}. We denote the RHS of \eqref{def:fk} by $\tilde f_k$. We then use \eqref{def:fk} to write
\begin{align}\label{eq:fk2}
	\tilde f_{k+2}(y_1,\dots,y_{k+2})=\sum_{m=2}^{k+2}G(y_1-y_m) f_k(y_2,\dots,y_{k+2}\backslash \{y_m\}).
\end{align}
Substituting \eqref{eq:fk2} into \eqref{eq:fk} and using \eqref{eq:G} we obtain
\begin{align*}
	&\tilde f_{k+2}(y_1,\dots,y_{k+2})+\int C^2(y_1-z)\tilde f_{k+2}(z,y_2,\dots,y_{k+2})\dif z
	\\&=
	\sum_{m=2}^{k+2}\Big(G(y_1-y_m)+\int C^2(y_1-z)G(z-y_m)\dif z\Big) f_k(y_2,\dots,y_{k+2}\backslash \{y_m\})
	\\&=2\sum_{m=2}^{k+2}C^2(y_1-y_m)f_k(y_2,\dots,y_{k+2}\backslash \{y_m\}),
\end{align*}
which proved our claim, so the result follows by
 uniqueness of the solutions to \eqref{eq:fk}.
	\end{proof}

Combining Lemma \ref{lem:1} and Lemma \ref{lem:2} we give the following characterization of the limit for the  observables
$\frac1{\sqrt N}\Wick{\PPhi^2}$.

\bt\label{th:1} Let $\m$ be as in Lemma \ref{th:m2}.  Any tight limit of $(\frac1{\sqrt N}\Wick{\PPhi^2})_N$ in $H^{-\kappa}, \kappa>0$ is a Gaussian field with mean zero and the covariance given by $G(x-y)$. Hence, the whole sequence $\frac1{\sqrt N}\Wick{\PPhi^2}$ converges in distribution to this Gaussian field in $H^{-\kappa},\kappa>0$.
\et
\begin{proof}
	By Proposition \ref{prop1} we know that $(\frac1{\sqrt N}\Wick{\PPhi^2})_N$ is tight in $H^{-\kappa}$ for every $\kappa>0$. Combining Lemma \ref{le:fkN} and Lemma \ref{lem:1} we have that the $k$-point functions of every tight limit satisfy \eqref{eq:fk}. The solutions to \eqref{eq:fk} are given in Lemma \ref{lem:1} and Lemma \ref{lem:2}. Hence, the law of every tight limit is the same and is given by the Gaussian field. Hence, the result follows.
\end{proof}

Let $\cQ$  denote the Gaussian field obtained in Theorem \ref{th:1}. We can analyze the regularity of $\cQ$ in the following result.

\bl It holds that for every $\kappa>0$, $p\geq 1$
\begin{align*}
	\E\|\cQ\|_{\bC^{-\kappa}}^p\lesssim1.
\end{align*}
\el
\begin{proof}
	Since the correlation of $\cQ$ is given by $G$, we have for $j\geq-1$
	\begin{align*}
		\E|(\Delta_j\cQ)(x)|^2= \sum_{k\in\mZ^2}\theta_j(k)^2\widehat{G}(k)\lesssim \sum_{k\in\mZ^2}\theta_j(k)^2 \widehat{C^2}(k)\lesssim 2^{j\kappa},
	\end{align*}
for $\kappa>0$,
where $(\Delta_j)_{j\geq-1}$ are the Littlewood--Paley blocks and $(\theta_j)_{j\geq-1}$ is the dyadic partition of unity.  Now it is standard to apply Gaussian hypercontractivity and Lemma \ref{lem:emb} to conclude the result.
\end{proof}

From the regularity of $\cQ$, which is not a random function but a distribution,  we cannot define $\cQ^n$ directly by using Lemma \ref{lem:multi}.
However, we can define $\Wick{\cQ^n}$ by the following result.

 \bp\label{pro:1} For every smooth mollifier $\rho_\eps$, set $G_\eps=\rho_\eps*G*\rho_\eps$ and $\cQ_\eps=\cQ*\rho_\eps$. For each $n\geq1$,
 \begin{equ}
 \Wick{\cQ_\eps^n}\eqdef G_\eps(0)^{n/2}H_n(G_\eps(0)^{-1/2}\cQ_\eps)
 \end{equ}
  converges as $\eps\to 0$ in $\bC^{-\kappa}$ $\bP$-a.s. for any $\kappa>0$ to the limit which is denoted by $\Wick{\cQ^n}$, where $H_n, n\geq1$ are  Hermite polynomials. The limit is independent of the choice of the mollifier.
\ep
\begin{proof}
 By Wick Theorem we have
\begin{align*}
	\E [|\Delta_j(\Wick{\cQ_\eps^n})(x)|^2]=\sum_{k\in\mZ^2}\theta_j(k)^2\widehat{G^n_\eps}(k)\lesssim 2^{j\kappa},
\end{align*}
where we used that $\widehat{G^n_\eps}=\widehat{G_\eps}*\widehat{G_\eps}*\dots*\widehat{G_\eps}$ to conclude that $\widehat{G^n_\eps}(k)\lesssim \frac{|k|^{\kappa}}{1+|k|^2}$ for $\kappa>0$. Now the result follows by Gaussian hypercontractivity and standard argument (cf. \cite[Section 6]{ZZ18}).
\end{proof}

\subsection{Large $N$ limit of $\frac1{N^{n/2}}\Wick{(\PPhi^2)^n}$}\label{sec:4.3}
In this section we consider  the large $N$ limit of the observables $\frac1{N^{n/2}}\Wick{ (\PPhi^2)^n}$ by finding the solutions to the recursive relation \eqref{eq:inductive2}.

We start with a discussion on the simplest cases to provide some intuition. Recall that in \cite[Theorem 6.5]{SSZZ2d}, it is shown that
\begin{equ}[e:h_1]
	\lim_{N\to\infty}\frac1N\E\Big(\Wick{(\PPhi^2)^2}\Big)=-(C^2*G)(0)\eqdef h_1.
\end{equ}
Consider the case $\mathbf{n}=(2,2)$ and in this case we write  the recursive relation \eqref{eq:inductive2} as
 	\begin{align}\label{e:h2}
 	&h_2(y_1,y_2)=4C^2(y_1-y_2)G(y_1-y_2)-\int C^2(y_1-z)f_{2,1}(y_1,z;y_2)\dif z,
 \end{align}
with
\begin{align*}
h_{2}(y_1,y_2)&\eqdef\lim_{N\to\infty}\lim_{\eps\to0}\frac1{N^{2}}\E\Big(\cE^\eps\Wick{(\PPhi^2_\eps)^2}(y_1)\cE^\eps\Wick{(\PPhi^2_\eps)^2}(y_2)\Big),
\\
f_{2,1}(y_1,z;y_2)&\eqdef\lim_{N\to\infty}\lim_{\eps\to0}\frac1{N^{2}}\E\Big(\cE^\eps\Wick{\PPhi^2_\eps}(y_1)\cE^\eps\Wick{\PPhi^2_\eps}(z)\cE^\eps\Wick{(\PPhi^2_\eps)^2}(y_2)\Big).
\end{align*}
Similarly, choosing $\mathbf{n}=(1,1,2)$ we obtain
	\begin{equs}[e:equ-f21]
	f_{2,1}& (z_1,z_2;y)
	+\int C^2(z_1-z_3) f_{2,1} (z_3,z_2;y)\dif z_3
	\\
	&=  -2C^2(z_1-z_2) \int C^2(x)G(x) \dif x
	+4C^2(z_1-y) G (z_2-y).
\end{equs}
Here \eqref{e:h2} and \eqref{e:equ-f21} are understood as in \eqref{eq:inductive2} but we omit the test functions for notation simplicity.
 In the following lemma solve the unknown $h_2$ and $f_{2,1}$ from  \eqref{e:h2} and \eqref{e:equ-f21}. 

\begin{lemma} It holds that
	\begin{equ}[e:f_21]
		f_{2,1} (z_1,z_2;y)
		= 2 G(z_1-y) G(z_2-y) -(C^2*G)(0) G(z_1-z_2),
	\end{equ}
	and
	\begin{equ}\label{e:h_2}
		h_2(y_1,y_2)=2G(y_1-y_2)^2+(C^2*G)^2(0).
	\end{equ}
\end{lemma}

\begin{proof}
It is elementary to solve \eqref{e:equ-f21}. For instance, by some speculation 
we could make an ansatz  
	\begin{equation}\label{eq:f21}
	f_{2,1} (z_1,z_2;y)
	= a G(z_1-y) G(z_2-y) + b (C^2*G)(0) G(z_1-z_2)
	\end{equation}
	and we try to solve the constants $a,b$.
	Recalling  $C^2 * G =2 C^2 - G$ from \eqref{e:G},
	\begin{equs}
		\int & C^2(z_1-z_3) f_{2,1} (z_3,z_2;y)\dif z_3
		\\ &=
		a G(z_2-y) \int  C^2(z_1-z_3) G(z_3-y)\dif z_3
		+ b (C^2*G)(0) \int  C^2(z_1-z_3) G(z_3-z_2)\dif z_3
		\\
		&= 2a G(z_2-y) C^2(z_1-y) - a G(z_1-y) G(z_2-y)
		\\&\qquad\qquad\qquad\qquad\qquad+ 2b (C^2*G)(0) C^2(z_1-z_2) -b (C^2*G)(0) G(z_1-z_2).
	\end{equs}
	Adding  the above  identity with \eqref{eq:f21}, some terms obviously cancel, and we obtain
	$$
	\mbox{LHS of } \eqref{e:equ-f21} =
	2a G(z_2-y) C^2(z_1-y)
	+ 2b (C^2*G)(0) C^2(z_1-z_2).
	$$
	So we choose
	$a =2, b =-1$ and we have proved \eqref{e:f_21}.
	
	Now by \eqref{e:h2} we also have (again using $C^2 * G = 2C^2 - G$)
	\begin{equs}
		h_2 (y_1,y_2)
		&=4C^2(y_1-y_2) G (y_1-y_2)
		-2\int C^2(y_1-z) G(y_1-y_2)G(z-y_2)\dif z
		\\&\quad +(C^2*G)(0)\int C^2(y_1-z)  G(y_1-z)\dif z
		=\eqref{e:h_2}.
	\end{equs}
\end{proof}

Before we proceed to prove the general results in Theorem~\ref{main:1}, let us give some motivation for the limit \eqref{lim:n1}. 
In this formal discussion, $C$, $G$, $\cQ$ etc. are understood via approximation as in Proposition \ref{pro:1}.

Note that $2G^2$ is the two-point correlation function of $\Wick{\cQ^2}$ (defined in Proposition \ref{pro:1}) since $\cQ$ is Gaussian. Approximately, $\frac{1}{N}\Wick{(\PPhi^2)^2}$ is the square of $\frac{1}{\sqrt{N}}\Wick{\PPhi^2}$ up to some difference in renormalization constants. Therefore, it is reasonable to speculate   that $\lim_{N}\frac{1}{N}\Wick{(\PPhi^2)^2}$ is given by a linear combination of $\Wick{\cQ^2}$ and $(C^2*G)(0)$, which then should be  $\Wick{\cQ^2}-(C^2*G)(0)$ since it has the correct expectation $h_1$ (see \eqref{e:h_1}) and the correct two-point function $h_2$ (see \eqref{e:h_2}).

To better understand the appearance of $(C^2*G)(0)$, 
recall the renormalization constant $G(0)$
in the definition of $\Wick{\cQ^2}$ in Proposition \ref{pro:1}. Also,  when we expand $\Wick{(\PPhi^2)^2}$ as in \eqref{e:formal4}, we have
 	 a renormalization constant   $2C^2(0)$. 
Recall that their difference is 
 \begin{align}\label{e:c}
 	-C^2*G(0)=G(0)-2C^2(0),
 \end{align}
 where $G(0)$ and $2C^2(0)$ are infinite quantities (understood via approximation).  
 Hence, 
 it is reasonable to speculate that 
 the large $N$ limit of $\frac{1}{N}\Wick{(\PPhi^2)^2}$ is given by $\Wick{\cQ^2}-(C^2*G)(0)$.

For  general $n\in\mN$, 
let us give a formal ``derivation'' for \eqref{lim:n1}
along the lines of \eqref{e:formal4}.
Assume that for $m\leq n$ we have 
$$\Wick{\Big(\frac1{\sqrt{N}}\PPhi^2\Big)^m}\to (2C^2(0))^{m/2}H_m((2C^2(0))^{-1/2}\cQ).$$
Then we {\it pretend} for the moment that $\Phi$ is simply its large $N$ limit $Z_i$, and
 formally apply Wiener chaos decomposition (c.f. \cite[Lemma 10.3]{Hairer14}) for the product of $\Wick{(\PPhi^2)^n}$ and $\Wick{\PPhi^2}$  to obtain
\begin{align*}
	\Wick{\Big(\frac1{\sqrt{N}}\PPhi^2\Big)^{n+1}}
	=&\Wick{\Big(\frac1{\sqrt{N}}\PPhi^2\Big)^n} \, \frac1{\sqrt{N}}\Wick{\PPhi^2}
	-C(0)\frac{4n}{N^{(n+1)/2}}\Wick{(\PPhi^2)^n}
	\\&-4C^2(0)\frac{n(n-1)}{N^{(n+1)/2}}\Wick{(\PPhi^2)^{n-1}}-2C^2(0)\frac{n}{N^{(n-1)/2}}\Wick{(\PPhi^2)^{n-1}}.
\end{align*}
On the RHS, the second term  arises from one contraction   between the new $\Wick{\PPhi^2}$ and one $\Wick{\PPhi^2}$ in $\Wick{(\PPhi^2)^n}$,
the third term arises from two contractions between  the new $\Wick{\PPhi^2}$ and two $\Wick{\PPhi^2}$ in $\Wick{(\PPhi^2)^n}$, and the last term is from two contractions between  the new $\Wick{\PPhi^2}$ and one $\Wick{\PPhi^2}$ in $\Wick{(\PPhi^2)^n}$, namely:
\begin{equ}
\begin{tikzpicture}
\node at (0,0) {$\Wick{\cdots \PPhi^2 \cdots  \PPhi^2 \cdots}$};
\node at (2,0) {$\Wick{\PPhi^2}$};
\draw (0.4,-0.3) -- (0.4,-0.5) -- (1.9,-0.5) -- (1.9,-0.3);
\end{tikzpicture}
\qquad
\begin{tikzpicture}
\node at (0,0) {$\Wick{\cdots \PPhi^2 \cdots  \PPhi^2 \cdots}$};
\node at (2,0) {$\Wick{\PPhi^2}$};
\draw (0.4,-0.3) -- (0.4,-0.4) -- (1.9,-0.4) -- (1.9,-0.3);
\draw (-0.6,-0.3) -- (-0.6,-0.5) -- (1.95,-0.5) -- (1.95,-0.3);
\end{tikzpicture}
\qquad
\begin{tikzpicture}
\node at (0,0) {$\Wick{\cdots \PPhi^2 \cdots  \PPhi^2 \cdots}$};
\node at (2,0) {$\Wick{\PPhi^2}$};
\draw (0.4,-0.3) -- (0.4,-0.4) -- (1.9,-0.4) -- (1.9,-0.3);
\draw (0.35,-0.3) -- (0.35,-0.5) -- (1.95,-0.5) -- (1.95,-0.3);
\end{tikzpicture}
\end{equ}
Since formally $\Wick{(\PPhi^2)^n}\sim N^{n/2}$,  
only the first term and the last term contribute and $\Wick{(\frac1{\sqrt{N}}\PPhi^2)^{n+1}}$ goes to
\begin{align*}
	&a^{n/2}H_n(a^{-1/2}\cQ)\cQ-na a^{(n-1)/2}H_{n-1}(a^{-1/2}\cQ)
	\\&=a^{(n+1)/2}H_n(a^{-1/2}\cQ)a^{-1/2}\cQ-na^{(n+1)/2}H_{n-1}(a^{-1/2}\cQ)
	\\&=a^{(n+1)/2}H_{n+1}(a^{-1/2}\cQ),
\end{align*}
with $a=2C^2(0)$,
where in the last equality we used $H_{n+1}(x)=xH_n(x)-nH_{n-1}(x).$
This formal calculation leads to 
``indicates'' that for any $n\in \mN$,  the limit of $\frac1{N^{n/2}}\Wick{(\PPhi^2)^n}$ is given by:
 \begin{equation}\label{lim:n}
 		\aligned
 		\WickC{\cQ^{n}}=& \lim_{\eps\to0} (2C_\eps^2(0))^{n/2}H_n((2C_\eps^2(0))^{-1/2}\cQ_\eps^n)
\\=	&\sum_{l=0}^{\lfloor n/2\rfloor} (-C^2*G(0))^l\frac{n!}{(n-2l)!l!2^l}\Wick{\cQ^{n-2l}},
\endaligned
\end{equation}
where $C_\eps=C*\rho_\eps$ and $\cQ_\eps$ and  $\rho_\eps$  are as in Proposition \ref{pro:1}. By  Proposition \ref{pro:1} the limit  in \eqref{lim:n} holds in $\bC^{-\kappa}$ $\bP$-a.s..
The second line of  \eqref{lim:n} follows from the definition of Hermite polynomials \eqref{def:Her} (see also \cite[Lemma 3.4]{RZZ17}).

 To show that the above guess is indeed true, our next step is to find explicit solutions to the recursive relation \eqref{eq:inductive2}.  In other words, we aim to provide an explicit formula for $f_{\mathbf{n},k}$, $\mathbf{n}=(n_1,\dots,n_k)\in \mN^k$ such that $f_{\mathbf{n},k}$ solves the following recursive relation
\begin{equation}\label{recur}
	f_{\mathbf{n},k}(y_1,\dots,y_k)+\int C^2(y_1-z)	f_{\hat{\mathbf{n}},k+1}(z,y_1,\dots, y_k)\dif z=
	\sum_{j=2}^k2n_j C^2(y_1-y_j)f_{\widetilde{\mathbf{n}}_j,k}( y_1,\dots, y_k),
\end{equation}
with $\hat{\mathbf{n}}=(1,n_1-1,n_2,\dots,n_k)$ and $\widetilde{\mathbf{n}}_j=(n_1-1,n_2,\dots,n_{j-1},n_j-1,n_{j+1},\dots,n_k)$.
To this end, we introduce the following notation:
denote by $\mathbf{y}$ the vector with $\sum_{i=1}^kn_i$ coordinates given by
$$
\mathbf{y}=\Big(\underbrace{y_1,\dots,y_1}_{n_1},
y_2,\dots,
\underbrace{y_j,\dots,y_j}_{n_j},\dots,y_{k-1},\underbrace{y_k,\dots,y_k}_{n_k}\Big),
$$
and we write $\mathbf{y}(j)$ for its $j$-th coordinate where $j=1,\cdots, \sum_{i=1}^kn_i$, for instance $\mathbf{y}(1)=y_1$.
For $\mathbf{n}=(n_1,\dots,n_k)\in \mN^k$, $k\in\mN$, define
  \begin{equation}\label{def:F}
  	F_{\mathbf{n},k}(y_1,\dots,y_k)
	\eqdef
	\begin{cases} \sum_\pi\prod_{j=1}^{\sum_{i=1}^kn_i/2}
	 G(\by(\pi(2j-1))-\by(\pi(2j))), &\text{ for }\sum_{i=1}^kn_i \text{ even, }
  	\\0,& \text{ for }\sum_{i=1}^kn_i \text{ odd, }
  	\end{cases}\end{equation}
where
$\pi$ runs over all the pairing permutations of $1,\dots,\sum_{i=1}^kn_i$ 
such that $\by(\pi(2j-1))\neq \by(\pi(2j))$ for $j=1,\dots,\sum_{i=1}^kn_i/2$.
In plain words, each $y_j$ can only pair with some $y_i$ with $i\neq j$.

Formally $F_{\mathbf{n},k}(y_1,\dots,y_k)$ denotes 
$$
\E\Big(\prod_{i=1}^k\Wick{\cQ^{n_i}}(y_i)\Big),
$$
and  for $\varphi_i\in \cS(\mT^2)$ with $i=1,\dots,k$,
 $$
 \<F_{\mathbf{n},k}(y_1,\dots,y_k),\varphi^{\otimes k}\>=\E\Big(\prod_{i=1}^k\<\Wick{\cQ^{n_i}},\varphi_i\>\Big).
 $$

\bl 
The functions $f_{\mathbf{n},k}(y_1,\dots,y_k), k\in\mN, \mathbf{n}=(n_1,\dots,n_k)\in \mN^k$  given by
\begin{align}\label{L1}
\sum_{l_1=0}^{\lfloor n_1/2\rfloor}
\cdots
\sum_{l_k=0}^{\lfloor n_k/2\rfloor} (-C^2*G(0))^{\sum_{i=1}^kl_i}\prod_{i=1}^k\frac{n_i!}{(n_i-2l_i)!l_i!2^{l_i}}F_{\mathbf{n}-2\mathbf{l},k}(y_1,\dots,y_k)
\end{align}
satisfies \eqref{recur},
where $\mathbf{n}-2\mathbf{l}=(n_1-2l_1,\dots,n_k-2l_k)$.
\el

Using \eqref{lim:n} and the above remark for $F_{\mathbf{n},k}$,
 it is easy to see that \eqref{L1} is formally the $k$-point function of $\WickC{\cQ^{n_i}}$, i.e.
$$
\E\Big(\prod_{i=1}^k\WickC{\cQ^{n_i}}(y_i)\Big).
$$

\begin{proof}
 We set $L_1=\eqref{L1}$.
Formally we view $F_{\mathbf{n}-2\mathbf{l},k}(y_1,\dots,y_k)$ as
$$\E\Big(\prod_{i=1}^k\Wick{\cQ^{n_i-2l_i}}(y_i)\Big).$$

Now we check \eqref{recur}.
The first term on the LHS of \eqref{recur} is 
$f_{\mathbf{n},k}$ which is now denoted to be $L_1$.

Now we consider the second term on the LHS of \eqref{recur}. Using \eqref{L1} with $\mathbf{n},k$ replaced by $\hat{\mathbf{n}},k+1$,
 $f_{\hat{\mathbf{n}},k+1}(z,y_1,\dots,y_k)$ is given by
\begin{equation}\label{eq:fnk+1}
	\aligned
\sum_{l_2=0}^{\lfloor n_2/2\rfloor}
\cdots
\sum_{l_k=0}^{\lfloor n_k/2\rfloor}
\sum_{l_1=0}^{\lfloor(n_1-1)/2\rfloor}
& (-C^2*G(0))^{\sum_{i=1}^kl_i}
\prod_{i=2}^k
\frac{n_i!}{(n_i-2l_i)!l_i!2^{l_i}}\frac{(n_1-1)!}{(n_1-1-2l_1)!l_1!2^{l_1}}
	\\&\times F_{\hat{\mathbf{n}}-2\mathbf{l},k+1}(z,y_1,\dots,y_k),
	\endaligned
\end{equation}
where $\hat{\mathbf{n}}-2\mathbf{l}=(1,n_1-1-2l_1,n_2-2l_2,\dots,n_k-2l_k)$, and formally  $F_{\hat{\mathbf{n}}-2\mathbf{l},k+1}(z,y_1,\dots,y_k)$ equals 
$$\E\Big(\cQ(z)\Wick{\cQ^{n_1-1-2l_1}}(y_1)\prod_{i=2}^k\Wick{\cQ^{n_i-2l_i}}(y_i)\Big).$$
Hence, the term $\int C^2(y_1-z)f_{\hat{\mathbf{n}},k+1}(z,y_1,\dots,y_k)\dif z$ on the LHS of \eqref{recur} can be written as
\begin{equation}\label{zxm:1}
	\aligned
	\sum_{l_2=0}^{\lfloor n_2/2\rfloor}
	\cdots
	\sum_{l_k=0}^{\lfloor n_k/2\rfloor}\sum_{l_1=0}^{\lfloor(n_1-1)/2\rfloor}& (-C^2*G(0))^{\sum_{i=1}^kl_i}\prod_{i=2}^k\frac{n_i!}{(n_i-2l_i)!l_i!2^{l_i}}\frac{(n_1-1)!}{(n_1-1-2l_1)!l_1!2^{l_1}}
	\\&\times\int C^2(y_1-z)F_{\hat{\mathbf{n}}-2\mathbf{l},k+1}(z,y_1,\dots,y_k)\dif z,
	\endaligned
\end{equation}
where the first line is the same constant as in \eqref{eq:fnk+1} and we take convolution w.r.t. $C^2$ for the second line in \eqref{eq:fnk+1}.

Recall the definition of $F_{\hat{\mathbf{n}}-2\mathbf{l},k+1}(z,y_1,\dots,y_k)$ from \eqref{def:F}, where we have different pairs for points
$$
z,\, \underbrace{y_1,\dots,y_1}_{n_1-1-2l_1},y_2,\dots,\underbrace{y_j,\dots,y_j}_{n_j-2l_j},\dots,y_{k-1},\underbrace{y_k,\dots,y_k}_{n_k-2l_k}.$$
Note that we only have one occurrence of $z$. Indeed, considering the point $z$ and its pairing with different points leads to two different cases which we now discuss:

\noindent\textbf{Case 1.} 
Suppose that the point $z$ pairs with $y_1$. Then we have $n_1-1-2l_1$ choices of different copies of $y_1$, which gives an extra factor $n_1-1-2l_1$. In this case the rest of the points pair with each other. Hence, we replace the second line of \eqref{zxm:1} by
\begin{align*}
	&(n_1-1-2l_1)\int C^2(y_1-z)G(z-y_1)\dif z F_{\mathbf{n}-2\mathbf{l}-\mathbf{2}_1,k}(y_1,\dots,y_k)
	\\&=	(n_1-1-2l_1) C^2*G(0) F_{\mathbf{n}-2\mathbf{l}-\mathbf{2}_1,k}(y_1,\dots,y_k).
\end{align*}
where $\mathbf{n}-2\mathbf{l}-\mathbf{2}_1=(n_1-2-2l_1,n_2-2l_2,\dots,n_k-2l_k)$. 
We substitute this into \eqref{zxm:1} and get the following expression as the first part of $\int C^2(y_1-z)f_{\hat{\mathbf{n}},k+1}(z,y_1,$ $\dots,y_k)\dif z$
\begin{align*}
	L_2\eqdef\sum_{l_2=0}^{\lfloor n_2/2\rfloor}
	\cdots
	\sum_{l_k=0}^{\lfloor n_k/2\rfloor}\sum_{l_1=0}^{\lfloor(n_1-1)/2\rfloor}& (-C^2*G(0))^{\sum_{i=1}^kl_i}\prod_{i=2}^k\frac{n_i!}{(n_i-2l_i)!l_i!2^{l_i}}\frac{(n_1-1)!}{(n_1-2-2l_1)!l_1!2^{l_1}}
	\\&\times C^2*G(0) F_{\mathbf{n}-2\mathbf{l}-\mathbf{2}_1,k}(y_1,\dots,y_k)
	\1_{\{ l_1\neq \frac{n_1-1}2\}}
	\\=-\sum_{l_2=0}^{\lfloor n_2/2\rfloor}
	\cdots
	\sum_{l_k=0}^{\lfloor n_k/2\rfloor}\sum_{l_1=1}^{\lfloor n_1/2\rfloor}&(-C^2*G(0))^{\sum_{i=1}^kl_i}\prod_{i=2}^k\frac{n_i!}{(n_i-2l_i)!l_i!2^{l_i}}\frac{(n_1-1)!}{(n_1-2l_1)!(l_1-1)!2^{l_1-1}}
	\\&\times F_{\mathbf{n}-2\mathbf{l},k}(y_1,\dots,y_k) ,
\end{align*}
where 
  the indicator function comes from the fact  that for $n_1$ odd and $l_1=(n_1-1)/2$ there is no $\cQ(y_1)$ to pair with $\cQ(z)$. We changed variable by rewriting $l_1+1$ as  $l_1$ in the second step.

\noindent\textbf{Case 2.} Suppose that the point $z$ pairs with $y_j$ for fixed $j\in\{2,\dots,k\}$ we also have $n_j-2l_j$ choices of different copies $y_j$. In this case the remaining points pair with each other. Hence, we replace the second line of \eqref{zxm:1} by
\begin{align*}
	&(n_j-2l_j)\int C^2(y_1-z)G(z-y_j)\dif z F_{\mathbf{n}-2\mathbf{l}-\mathbf{1}_j,k}(y_1,\dots,y_k)
	\\&=	(n_j-2l_j) C^2*G(y_1-y_j) F_{\mathbf{n}-2\mathbf{l}-\mathbf{1}_j,k}(y_1,\dots,y_k),
\end{align*}
where $$\mathbf{n}-2\mathbf{l}-\mathbf{1}_j=(n_1-2l_1-1,n_2-2l_2,\dots,n_{j-1}-2l_{j-1},n_j-2l_j-1,n_{j+1}-2l_{j+1},\dots,n_k-2l_k).$$
We plug this into \eqref{zxm:1} and get the following expression as the second part of $\int C^2(y_1-z)f_{\hat{\mathbf{n}},k+1}(z,y_1,$ $\dots,y_k)\dif z$
\begin{align*}
L_3\eqdef	\sum_{j=2}^k\sum_{\stackrel{l_i=0}{i=2,\dots,k, i\neq j}}^{\lfloor n_i/2\rfloor}\sum_{l_j=0}^{\lfloor(n_j-1)/2\rfloor}\sum_{l_1=0}^{\lfloor(n_1-1)/2\rfloor}&(-C^2*G(0))^{\sum_{i=1}^kl_i}\prod_{i=2,i\neq j}^k\frac{n_i!}{(n_i-2l_i)!l_i!2^{l_i}}
	\\&\times\frac{(n_1-1)!}{(n_1-1-2l_1)!l_1!2^{l_1}}\frac{n_j!}{(n_j-1-2l_j)!l_j!2^{l_j}}
	\\& C^2*G(y_1-y_j)\times F_{\mathbf{n}-2\mathbf{l}-\mathbf{1}_j,k}(y_1,\dots,y_k) ,
\end{align*}
 and
we used the fact that for $n_j$ even and $l_j=n_j/2$ there is no point $y_j$ to pair with the point $z$ and we take sum for $l_j$ from $0$ to $\lfloor(n_j-1)/2\rfloor$. Formally we view $F_{\mathbf{n}-2\mathbf{l}-\mathbf{1}_j,k}(y_1,\dots,y_k)$
as
$$ \E\Big(\Wick{\cQ^{n_1-1-2l_1}}(y_1)\Wick{\cQ^{n_j-1-2l_j}}(y_j)\prod_{i=2,i\neq j}^k\Wick{\cQ^{n_i-2l_i}}(y_i)\Big).$$
Combining the above calculations, we obtain
\begin{align*}
	\text{LHS of }\eqref{recur}=L_1+L_2+L_3.
\end{align*}

We also observe that $L_1$ and $L_2$ involve the same term $F_{\mathbf{n}-2\mathbf{l},k}$, but with different coefficients. Hence, we have
\begin{equation}\label{dec:1}
	\aligned
	L_1+L_2
	=&\sum_{l_2=0}^{\lfloor n_2/2\rfloor}
	\cdots
	\sum_{l_k=0}^{\lfloor n_k/2\rfloor}\sum_{l_1=1}^{\lfloor n_1/2\rfloor}(-C^2*G(0))^{\sum_{i=1}^kl_i}\prod_{i=2}^k\frac{n_i!}{(n_i-2l_i)!l_i!2^{l_i}}
	\\&\times\Big[\frac{n_1!}{(n_1-2l_1)!l_1!2^{l_1}}-\frac{(n_1-1)!}{(n_1-2l_1)!(l_1-1)!2^{l_1-1}}\Big]F_{\mathbf{n}-2\mathbf{l},k}(y_1,\dots,y_k)
	\\&+\sum_{l_2=0}^{\lfloor n_2/2\rfloor}
	\cdots
	\sum_{l_k=0}^{\lfloor n_k/2\rfloor}(-C^2*G(0))^{\sum_{i=2}^kl_i}\prod_{i=2}^k\frac{n_i!}{(n_i-2l_i)!l_i!2^{l_i}}
	 F_{\mathbf{n}-2\mathbf{l},k}(y_1,\dots,y_k) \1_{\{l_1=0\}},
	 \endaligned
\end{equation}
where the last line comes from the term $L_1$ with $l_1=0$, 
recalling that the final expression of $L_2$ does not have the term with $l_1=0$.
Using
\begin{align*}
	\frac{n_1!}{(n_1-2l_1)!l_1!2^{l_1}}-\frac{(n_1-1)!}{(n_1-2l_1)!(l_1-1)!2^{l_1-1}}=	\begin{cases}
	\frac{(n_1-1)!}{(n_1-2l_1-1)!l_1!2^{l_1}},&l_1\leq \lfloor(n_1-1)/2\rfloor\\
	0,&l_1=n_1/2, n_1 \text{ is even},
\end{cases}
\end{align*}
the first term on the RHS of \eqref{dec:1} can be written as
\begin{align}\label{dec:2}
&\sum_{l_2=0}^{\lfloor n_2/2\rfloor}
\cdots
\sum_{l_k=0}^{\lfloor n_k/2\rfloor}\sum_{l_1=1}^{\lfloor(n_1-1)/2\rfloor}(-C^2*G(0))^{\sum_{i=1}^kl_i}\prod_{i=2}^k\frac{n_i!}{(n_i-2l_i)!l_i!2^{l_i}}\frac{(n_1-1)!}{(n_1-2l_1-1)!l_1!2^{l_1}}F_{\mathbf{n}-2\mathbf{l},k}(y_1,\dots,y_k)  .
\end{align}
Note that the second term on the RHS of \eqref{dec:1} precisely corresponds to the case  $l_1=0$ in \eqref{dec:2}. Hence, 
\begin{equs}\label{zxm2}
L_1+L_2&=	\sum_{l_2=0}^{\lfloor n_2/2\rfloor}
\cdots
\sum_{l_k=0}^{\lfloor n_k/2\rfloor}\sum_{l_1=0}^{\lfloor(n_1-1)/2\rfloor}(-C^2*G(0))^{\sum_{i=1}^kl_i}
\\
&\prod_{i=2}^k\frac{n_i!}{(n_i-2l_i)!l_i!2^{l_i}}\frac{(n_1-1)!}{(n_1-2l_1-1)!l_1!2^{l_1}} F_{\mathbf{n}-2\mathbf{l},k}(y_1,\dots,y_k)  .
\end{equs}

Now we consider the RHS of \eqref{recur}. Using $2C^2=C^2*G+G$, we write the RHS of \eqref{recur}
as
\begin{align}\label{eq:nj}
	\sum_{j=2}^kn_j	G(y_1-y_j)f_{\widetilde{\mathbf{n}}_j,k}( y_1,\dots, y_k)+\sum_{j=2}^kn_j	C^2*G(y_1-y_j)f_{\widetilde{\mathbf{n}}_j,k}( y_1,\dots, y_k)
	\eqdef R_1+R_2.
\end{align}
Also, using \eqref{L1} with $\mathbf{n}$ replaced by $\widetilde{\mathbf{n}}_j$,  $f_{\widetilde{\mathbf{n}}_j,k}( y_1,\dots,y_k)$ is written as
\begin{equation}\label{zxm3}
	\aligned
	\sum_{\stackrel{l_i=0}{i=2,\dots,k,i\neq j}}^{\lfloor n_i/2\rfloor}\sum_{l_1=0}^{\lfloor(n_1-1)/2\rfloor}\sum_{l_j=0}^{\lfloor(n_j-1)/2\rfloor}& (-C^2*G(0))^{\sum_{i=1}^kl_i}\prod_{i=2,i\neq j}^k\frac{n_i!}{(n_i-2l_i)!l_i!2^{l_i}}	\\&\times\frac{(n_1-1)!}{(n_1-1-2l_1)!l_1!2^{l_1}}\frac{(n_j-1)!}{(n_j-1-2l_j)!l_j!2^{l_j}}
\\&\times F_{\mathbf{n}-2\mathbf{l}-\mathbf{1}_j,k}(y_1,\dots,y_k)  .
\endaligned
\end{equation}
We plug \eqref{zxm3} into the first term of \eqref{eq:nj} and have
\begin{align*}
R_1=\sum_{j=2}^k\sum_{\stackrel{l_i=0}{i=2,\dots,k,i\neq j}}^{\lfloor n_i/2\rfloor}\sum_{l_1=0}^{\lfloor(n_1-1)/2\rfloor}\sum_{l_j=0}^{\lfloor(n_j-1)/2\rfloor}& (-C^2*G(0))^{\sum_{i=1}^kl_i}\prod_{i=2,i\neq j}^k\frac{n_i!}{(n_i-2l_i)!l_i!2^{l_i}}	\\&\times\frac{(n_1-1)!}{(n_1-1-2l_1)!l_1!2^{l_1}}\frac{n_j!}{(n_j-1-2l_j)!l_j!2^{l_j}}
	\\G(y_1-y_j)&\times F_{\mathbf{n}-2\mathbf{l}-\mathbf{1}_j,k}(y_1,\dots,y_k) .
\end{align*}
We observe that the only difference between the expression for $f_{\widetilde{\mathbf{n}}_j,k}$ and the previous expression is the additional constant $n_j$ in the second line and the inclusion of $G(y_1-y_j)$ in the third line.

We then compare $R_1$ with $L_1+L_2$ in \eqref{zxm2}. Consider the term $F_{\mathbf{n}-2\mathbf{l},k}(y_1,\dots,y_k)$ in $L_1+L_2$, we fix one $y_1$ and pair it with $y_j$ for fixed $j=2,\dots,k$ and we have $n_j-2l_j$ choices of $y_j$ and the rest points pair with each other. For fixed $j$, we could replace $F_{\mathbf{n}-2\mathbf{l},k}(y_1,\dots,y_k)$ in \eqref{zxm2} by
$$(n_j-2l_j)G(y_1-y_j)\times F_{\mathbf{n}-2\mathbf{l}-\mathbf{1}_j,k}(y_1,\dots,y_k).$$
Hence, we could write $L_1+L_2$ as
\begin{align*}
	\sum_{j=2}^k\sum_{\stackrel{l_i=0}{i=2,\dots,k,i\neq j}}^{\lfloor n_i/2\rfloor}\sum_{l_1=0}^{\lfloor(n_1-1)/2\rfloor}\sum_{l_j=0}^{\lfloor(n_j-1)/2\rfloor}&(-C^2*G(0))^{\sum_{i=1}^kl_i}\prod_{i=2,i\neq j}^k\frac{n_i!}{(n_i-2l_i)!l_i!2^{l_i}}
	\\&\times\frac{(n_1-1)!}{(n_1-1-2l_1)!l_1!2^{l_1}}\frac{n_j!}{(n_j-1-2l_j)!l_j!2^{l_j}}
	\\G(y_1-y_j)&\times F_{\mathbf{n}-2\mathbf{l}-\mathbf{1}_j,k}(y_1,\dots,y_k),
\end{align*}
which is exactly $R_1$. We then have
 $R_1=L_1+L_2$.

For the second term in \eqref{eq:nj} we also plug \eqref{zxm3} into it and have
\begin{align*}
R_2=	\sum_{j=2}^k\sum_{\stackrel{l_i=0}{i=2,\dots,k,i\neq j}}^{\lfloor n_i/2\rfloor}\sum_{l_1=0}^{\lfloor(n_1-1)/2\rfloor}\sum_{l_j=0}^{\lfloor(n_j-1)/2\rfloor}& (-C^2*G(0))^{\sum_{i=1}^kl_i}\prod_{i=2,i\neq j}^k\frac{n_i!}{(n_i-2l_i)!l_i!2^{l_i}}	\\&\times\frac{(n_1-1)!}{(n_1-1-2l_1)!l_1!2^{l_1}}\frac{n_j!}{(n_j-1-2l_j)!l_j!2^{l_j}}
	\\C^2*G(y_1-y_j)&\times F_{\mathbf{n}-2\mathbf{l}-\mathbf{1}_j,k}(y_1,\dots,y_k),
\end{align*}
which is equal to $L_3$. Thus we have $L_1+L_2+L_3=R_1+R_2$, which verifies \eqref{recur}.
\end{proof}

\bp\label{p:u} The solutions to the recursive relation \eqref{eq:inductive2} are unique.
\ep
\begin{proof}
	For $\mathbf{n}=(1,\dots,1)$, $f_{\mathbf{n},k}$ are given in Lemma \ref{lem:2}. Now we consider general $\mathbf{n}=(n_1,\dots,n_k)$ and use \eqref{eq:inductive2} to find $f_{\mathbf{n},k}$ determined by $f_{\hat{\mathbf{n}},k+1}$ and $f_{\widetilde{\mathbf{n}}_j,k}$ with $\hat{\mathbf{n}}=(1,n_1-1,n_2,\dots,n_k)$ and $\widetilde{\mathbf{n}}_j=(n_1-1,n_2,\dots,n_{j-1},n_j-1,n_{j+1},\dots,n_k)$, $j=2,\dots,k$. For  $f_{\hat{\mathbf{n}},k+1}$ it satisfies
	\begin{align*}
		&f_{\hat{\mathbf{n}},k+1}(z,y_1,\dots,y_k)+\int C^2(z-z_1)	f_{\hat{\mathbf{n}},k+1}(z_1,y_1,\dots, y_k)\dif z_1
		\\&=
		\sum_{j=1}^k2n_j C^2(z-y_j)f_{\widetilde{\hat{\mathbf{n}}}_j,k}( y_1,\dots, y_k),
	\end{align*}
where $\widetilde{\hat{\mathbf{n}}}_1=(n_1-2,n_2,\dots,n_k)$ and $\widetilde{\hat{\mathbf{n}}}_j=\widetilde{\mathbf{n}}_j$ for $j\geq2$ and we omit test functions for notation simplicity.
By utilizing Fourier transform, we can observe that $f_{\hat{\mathbf{n}},k+1}$ is determined by $f_{\widetilde{\hat{\mathbf{n}}}j,k}$. This implies that $f_{\mathbf{n},k}$ is determined by $f_{\widetilde{\mathbf{n}}_j,k}$ and $f_{\widetilde{\hat{\mathbf{n}}}_j,k}$.
Comparing $\mathbf{n}$ with $\widetilde{\mathbf{n}}_j$ and $\widetilde{\hat{\mathbf{n}}}$, we can see that $\widetilde{\mathbf{n}}_j$ and $\widetilde{\hat{\mathbf{n}}}$ have fewer points $y_1$ or $y_j$. By repeating these steps, we can continue to reduce the number of points until we arrive at $f_{j}$ with $j\leq k$.
As a result, we conclude that $f_{\mathbf{n},k}$ is determined by $f_{j}$ with $j\leq k$. This proves the claimed uniqueness.
\end{proof}

\bt\label{th:o}  
Let $\mathbf{m}$ be as in Lemma \ref{th:m2}. For $\mathbf{n}=(n_1,\dots,n_m)\in\mN^m, m\in\mN$, $\kappa>0$
$$\Big\{\Big(\frac1{ N^{n_1/2}}\Wick{(\PPhi^2)^{n_1}},\cdots,\frac1{ N^{n_m/2}}\Wick{(\PPhi^2)^{n_m}}\Big)\Big\}_N$$
 converges in law in $(H^{-\kappa})^m$ to
 $$(\WickC{\cQ^{n_1}},\cdots,\WickC{\cQ^{n_m}}),$$
 with
 $$\WickC{\cQ^{n}}\eqdef \sum_{l=0}^{[n/2]} (-C^2*G(0))^{l}\frac{n!}{(n-2l)!l!2^{l}}\Wick{\cQ^{n-2l}},\quad n\in\mN.$$
\et
\begin{proof}
	By Proposition \ref{prop1},
	$$\Big\{\Big(\frac1{ N^{n_1/2}}\Wick{(\PPhi^2)^{n_1}},\cdots,\frac1{ N^{n_m/2}}\Wick{(\PPhi^2)^{n_m}}\Big)\Big\}_N$$
	 is tight in $(H^{-\kappa})^m$ and the $\mathbf{k}$-point correlation functions $\mathbf{k}=(k_1,\cdots,k_m)$  of every tight limit satisfies \eqref{eq:inductive2} with the vector $\mathbf{n}$ replaced by the following vector $$\mathbf{n}^{\mathbf{k}}=(\underbrace{n_1,\cdots,n_1}_{k_1},n_2,\cdots,\underbrace{n_j,\cdots,n_j}_{k_j},\cdots,\underbrace{n_m,\dots,n_m}_{k_m}).$$  By Proposition \ref{p:u} we know that the solutions to \eqref{eq:inductive2} are unique. The unique solution is given by \eqref{L1} with $\mathbf{n}$ replaced by $\mathbf{n}^{\mathbf{k}}$. By Wick's Theorem this is exactly the $\mathbf{k}$-point function of $$(\WickC{\cQ^{n_1}},\cdots,\WickC{\cQ^{n_m}}).$$
Hence, the result follows.
\end{proof}

\section{$1/N$ expansion of the $k$-points functions of $\frac1{\sqrt N}\Wick{\PPhi^2}$}\label{sec:5}

In this section we derive the $1/N$ expansion for $f_{k}^N$, $k\in \mN$, i.e. the $k$-points functions of $\frac1{\sqrt N}\Wick{\PPhi^2}$ defined in \eqref{def:fkN} and give a proof of Theorem \ref{main}.
We will frequently use the recursions obtained in Section~\ref{sec:IBP},
and our induction arguments will rely on graphic notation
to represent various terms arising from the recursions.

Our approach here is similar but more sophisticated than  
\cite{SZZ21} which studied the perturbative expansion
for (single component) $\lambda\Phi^4$ model.
In \cite{SZZ21}, one iteratively applies IBP to correlation functions
in order to either decrease the number of $\Phi$ or 
produce terms of higher orders in $\lambda$, and
this generates the perturbative  expansion (in $\lambda$) such that 
each term is an expression only depending on the Green's function
$C=(\m-\Delta)^{-1}$, except for the remainder.
One can keep track of the structure of the remainder by graphs. 
One then applies SPDE estimates to bound the  remainder 
which requires a procedure to find a spanning tree from the graph.

Here, in order to obtain a $1/N$ expansion, we first analyze the structure of each term obtained from IBP. Specifically, we apply two types of IBP or recursions (Lemma \ref{lem:IBP} and Lemma \ref{lem:32}) to classify the terms into two categories. We then reduce the number of $\PPhi^2$ in each graph by repeating the IBP procedure multiple times.  Next, we analyze the parity (evenness or oddness) of the number of $\PPhi^2$ in each term when performing the $\frac{1}{\sqrt{N}}$ expansion. Each odd graph can be reduced to the product of $\frac1{\sqrt N}\E\Wick{\PPhi^2}$ and a function only depending on $C$. Since $\E\Wick{\PPhi^2}$ is of order $1$, this analysis allows us to prove that all the odd parts are of order $\frac{1}{\sqrt{N}}$ and provides the $1/N$ expansion. 

Our graphic notation here is close to \cite{SZZ21}. We denote $C_\eps$ by a  line. We will also use wavy lines to represent the field $\Phi$. More precisely, single / double / triple wavy lines represent $\Phi_1$, $\frac1{\sqrt N}\Wick{\PPhi^2}$ and $\frac1{\sqrt N}\Wick{\Phi_1\PPhi^2}$ respectively.
With this graphic notation,  for instance,
\eqref{eq:inductive} in Lemma~\ref{lem:32} in the case $k=4$  and $n_1=\cdots =n_4=1$ 
can be  represented by

\begin{equs}[e:graph-4pt]
\begin{tikzpicture}[baseline=-15]
	\node[dot] (x1) at (0,0) {}; 
	\node[dot] (x2) at (1,0) {};
	\node[dot] (x3) at (0,-1) {};
	\node[dot] (x4) at (1,-1) {};
	\draw[Phi] (x1) -- ++(-0.1,-0.25);\draw[Phi] (x1) -- ++(0.1,-0.25);
	\draw[Phi] (x2) -- ++(-0.1,-0.25);\draw[Phi] (x2) -- ++(0.1,-0.25);
	\draw[Phi] (x3) -- ++(-0.1,0.25);\draw[Phi] (x3) -- ++(0.1,0.25);
	\draw[Phi] (x4) -- ++(-0.1,0.25);\draw[Phi] (x4) -- ++(0.1,0.25);
\end{tikzpicture}
\quad& +(1+\frac{2}{N})\;\;
\begin{tikzpicture}[baseline=-15]
	\node[dot] (x1) at (0,0) {};
	\node[dot] (x2) at (1,0) {};
	\node[dot] (x3) at (0,-1) {};
	\node[dot] (x4) at (1,-1) {};
	\node[dot] (z) at (.5,-.3) {};
	\draw[C,bend left=30] (x1) to (z);\draw[C,bend right=30] (x1) to (z);
	\draw[Phi] (z) -- ++(-0.1,-0.25);\draw[Phi] (z) -- ++(0.1,-0.25);
	\draw[Phi] (x2) -- ++(-0.1,-0.25);\draw[Phi] (x2) -- ++(0.1,-0.25);
	\draw[Phi] (x3) -- ++(-0.1,0.25);\draw[Phi] (x3) -- ++(0.1,0.25);
	\draw[Phi] (x4) -- ++(-0.1,0.25);\draw[Phi] (x4) -- ++(0.1,0.25);
\end{tikzpicture}
\\ & =\;
2\;\;\Bigg( \;\;\;
\begin{tikzpicture}[baseline=-15]
	\node[dot] (x1) at (0,0) {}; 
	\node[dot] (x2) at (1,0) {};
	\node[dot] (x3) at (0,-1) {};
	\node[dot] (x4) at (1,-1) {};
	\draw[C,bend left=30] (x1) to (x2);\draw[C,bend right=30] (x1) to (x2);
	\draw[Phi] (x3) -- ++(-0.1,0.25);\draw[Phi] (x3) -- ++(0.1,0.25);
	\draw[Phi] (x4) -- ++(-0.1,0.25);\draw[Phi] (x4) -- ++(0.1,0.25);
\end{tikzpicture}
\quad +\;
\begin{tikzpicture}[baseline=-15]
	\node[dot] (x1) at (0,0) {};
	\node[dot] (x2) at (1,0) {};
	\node[dot] (x3) at (0,-1) {};
	\node[dot] (x4) at (1,-1) {};
	\draw[C,bend left=30] (x1) to (x3);\draw[C,bend right=30] (x1) to (x3);
	\draw[Phi] (x2) -- ++(-0.1,-0.25);\draw[Phi] (x2) -- ++(0.1,-0.25);
	\draw[Phi] (x4) -- ++(-0.1,0.25);\draw[Phi] (x4) -- ++(0.1,0.25);
\end{tikzpicture}
\quad +\;
\begin{tikzpicture}[baseline=-15]
	\node[dot] (x1) at (0,0) {};
	\node[dot] (x2) at (1,0) {};
	\node[dot] (x3) at (0,-1) {};
	\node[dot] (x4) at (1,-1) {};
	\draw[C,bend left=20] (x1) to (x4);\draw[C,bend right=20] (x1) to (x4);
	\draw[Phi] (x2) -- ++(-0.1,-0.25);\draw[Phi] (x2) -- ++(0.1,-0.25);
	\draw[Phi] (x3) -- ++(-0.1,0.25);\draw[Phi] (x3) -- ++(0.1,0.25);
\end{tikzpicture}
\;\;\;\Bigg)
+ N^{-1} Q_{N,\mathbf{n}}^\eps
\end{equs}
Here, in the second graph on LHS, the ``new'' vertex (i.e. the vertex in the middle) corresponds to an integration variable over space.
The $Q_{N,\mathbf{n}}^\eps$ terms defined in \eqref{de:Q} can be also represented graphically, for instance
\begin{equ}[e:Q-graph]
Q_{N,\mathbf{n}}^{4,\eps}=\sqrt N\;\;
\begin{tikzpicture}[baseline=-15]
	\node[dot] (x1) at (0,0) {};
	\node[dot] (x2) at (1,0) {};
	\node[dot] (x3) at (0,-1) {};
	\node[dot] (x4) at (1,-1) {};
	\node[dot] (z1) at (.5,0) {};\node[dot] (z2) at (.4,-.5) {};
	\draw[C,bend left=30] (x1) to (z1);\draw[C,bend right=10] (x1) to (z2);
	\draw[Phi] (z1) -- ++(0.18,-0.18);\draw[Phi] (z1) -- ++(-0.18,-0.18);\draw[Phi] (z1) -- ++(0,-0.25);
	\draw[Phi] (z2) -- ++(0.18,-0.18);\draw[Phi] (z2) -- ++(-0.18,-0.18);\draw[Phi] (z2) -- ++(0,-0.25);
	\draw[Phi] (x2) -- ++(-0.1,-0.25);\draw[Phi] (x2) -- ++(0.1,-0.25);
	\draw[Phi] (x3) -- ++(-0.1,0.25);\draw[Phi] (x3) -- ++(0.1,0.25);
	\draw[Phi] (x4) -- ++(-0.1,0.25);\draw[Phi] (x4) -- ++(0.1,0.25);
\end{tikzpicture}
\end{equ}
Here the factor $\sqrt N$ arises from 
the factor $\frac1{N^2}$ in 
the definition of $Q_{N,\mathbf{n}}^{4,\eps}$
and the $5$ points which carry double or triple wavy lines.

In general,
given a graph $G$, we write $G=(V_G,E_G)$ or simply $G=(V,E)$ where $V$ is the set of vertices
and $E$ is the set of edges.
We denote by $|V|$, $|E|$ the cardinalities of these sets,
namely the number of vertices and edges. Here for any two distinct vertices $u,v\in V$,
we allow  multiple edges between $u$ and $v$ (namely we allow `multigraphs' in the language of graph theory). However, we will assume
throughout the paper that our graphs do not have self-loops, i.e. there is not any edge of the form $\{u,u\}$ for $u\in V$.
 For every $v\in V_G$ we denote by $\deg(v)$ (called the ``degree'' of $v$) the number of $C_\eps$-lines adjacent to $v$ (we do not count wavy lines attached to $v$).

\begin{definition}\label{def:nPhik}
For each $\ell\ge 0$, $k\in\mN$
		we define $\mathcal H_\ell^k$ to be the set of all the graphs $G=(V,E)$ such that

(1)		$|V|=\ell+k$

(2) there are  $k$ ``special points'' $\{u^*_m,m=1,\dots,k\}$ in $V$
		with $\deg(u^*_m)\in\{0,1,2\}$,
		
(3)  $\deg(v)\in \{1,2,3,4\}$ for every $v\in V \backslash \{u^*_m,m=1,\dots,k\}$ and

(4) there are  two points or no point with odd degree.

We also denote by $\mathcal H_\ell^{k,1}$
 the set of all the graphs satisfying (1)(2)(3) and such that all the points have even degrees, and $\mathcal H_\ell^{k,2}$ the set of all the graphs satisfying (1)(2)(3) and such that only two  points have odd degrees. In particular $\mathcal H_\ell^{k} = \mathcal H_\ell^{k,1} \cup \mathcal H_\ell^{k,2}$.

	 For any such graph we will write
	$V^\partial_G = \{u^*_m,m=1,\dots,k\}$ and
	$V^0_G= V_G \backslash V^\partial_G$.
	
		For $G\in \mathcal H_\ell^k$ we also write
		$$n_{\Phi}(G) := 4\ell+2k - \sum_{v\in V} \deg(v).$$
	
\end{definition}

For instance, all the graphs in  \eqref{e:graph-4pt} belong to
$\mathcal H_\ell^{k,1}$ for $k=4$, where 
$\ell=0$ except for the second graph on the LHS where $\ell=1$.
The graph \eqref{e:Q-graph}  belongs to
$\mathcal H_\ell^{k,2}$ for $k=4$ and $\ell=2$. 
The number $n_{\Phi}(G)$ simply counts
 the number of wavy lines in the graph $G$, and 
 by definition $n_{\Phi}(G)$ is even. 

We define a mapping from $\mathcal H_\ell^k$ to the set of all functions in $\{x_{u^*_m}\}_{m=1}^k$,
which maps
$G\in \mathcal H_\ell^k$ to
\begin{align}\label{def:IGk}
	\mathbf{I}_G (x_1,\dots,x_{k})=\int& \Big(\!\!\!\prod_{\{u,v\}\in E_G}\!\!\!\! C_{\eps}(x_u,x_v)\Big)
	{N^{-\frac12(\sum_{u_m^*\in V_T^{\partial}}\lfloor1-\deg(u_m^*)/2\rfloor+\sum_{z\in V_G^0}\lfloor2-\deg(z)/2\rfloor)}}\,\no
	\\&\E \Big(\prod_{u_m^*\in V^\partial_G}  \Wick{\PPhi_{\eps}^{2-\deg(u_m^*)}}(x_m)
	\prod_{z\in V^0_G} \Wick{\PPhi_{\eps}^{4-\deg(z)}}(x_z)
	\Big)
	\prod_{z\in V^0_G} \dif x_z,
\end{align}
where $x_m=x_{u_m^*}$. Here we write  $\Wick{\PPhi_\eps}=\Phi_{1,\eps}$ and $\Wick{\PPhi_\eps^3}=\Wick{\Phi_{1,\eps}\PPhi_{\eps}^2}$ as shorthand notation.

\br\label{rem} Compared to \cite{SZZ21}, the IBP for observables here is more complicated, and we need to utilize the symmetry property in the IBP. In the following discussion, we only make use of the IBP presented in Lemma \ref{lem:IBP}, Lemma \ref{lem:32}, and Lemma \ref{lem:1}. 
These lemmas provide us with terms in the form of $\mathbf{I}_G$ which involve either two or zero instances of $\Wick{\Phi_1\PPhi^2}$ or $\Phi_1$, while all the other factors consist of $\Wick{\PPhi^2}$ only. 
Specifically, for $G\in \cH^{k,1}$, $\mathbf{I}_G$ refers to a term that solely involves $\Wick{\PPhi^2}$; 
and for $G\in \cH^{k,2}$, $\mathbf{I}_G$ denotes a term that  includes not only $\Wick{\PPhi^2}$ but also two instances of $\Wick{\Phi_1\PPhi^2}$ or $\Phi_1$.
This motivates item (4) in Definition~\ref{def:nPhik}.
\er

Before proceeding, we note that both terms in the LHS of \eqref{eq:fkN} involve $f_{k,\eps}^N$. To handle this part we introduce an operator
\begin{align}\label{def:K}
	Kf\eqdef (I+C^2*\cdot)^{-1}f,\quad f\in L^2(\mT^2),
\end{align}
where $I$ is the identity operator
and its discrete version
\begin{align}\label{def:Ke}
	K_\eps f\eqdef (I+C^2_\eps*\cdot)^{-1}f,\quad f\in L^{2,\eps}(\Lambda_\eps).
\end{align}
  We could write $Kf$ as
\begin{align*}
	Kf(x)=	\int_{\mathbb{R}^2} K(x-y)f(y)\dif y, \quad K(x-y)=\delta(x-y)+L(x-y),\quad x,y\in \mT^2,
\end{align*}
with $L\in L^p(\mathbb{R}^2), p\geq 2$, where we view $f$ as periodic function on $\mR^2$. Similar for the discrete operator $K_\eps$.  These results are proved in Lemma \ref{lem:K} and Lemma \ref{lem:Kd} in Appendix.

 In addition, we introduce edges for the kernel $K$ in the graphs. Let $G = (V_G, E_G) \in \mathcal{H}_\ell^k$, and we construct a new graph $G^e = (V_{G^e}, E_{G^e})$, where $V_{G^e} = V_G \cup V_{G^e}^s$ with $V_{G^e}^s$ representing the set of points ${y_1, \dots, y_q}$, where $q \in \{0, \dots, \ell+k\}$. We choose a set of points $V^* = \{u_1, \dots, u_q\} \subset V_G$ such that each $u_i$ has a full degree. Specifically, for $u_i \in V_G^\partial$, we have $\deg(u_i) = 2$, and for $u_i \in V_G^0$, we have $\deg(u_i) = 4$. The set of edges $E_{G^e}$ is obtained by modifying two edges in $E_G$ that connect to each point $u_i$, replacing them with edges that connect to $y_i$, and adding edges in 
 $E_{G^e}^1 \eqdef \{\{u_i, y_i\} : i = 1, \dots, q\}$. More precisely:
  \begin{itemize}
  	\item For $u_i$ with $\deg u_i=2$, the edges $\{u_i,v_i\}, \{u_i,x_i\}$ for $v_i,x_i\in V_G$ are replaced by $\{u_i,y_i\}, \{y_i,v_i\},$ $\{y_i,x_i\}$;
  	\item For $u_i$ with $\deg u_i=4$, we choose two points $v_i, x_i$ connecting with $u_i$ and the edges $\{u_i,v_i\}$, $\{u_i,x_i\}$ for $v_i,x_i\in V_G$ are replaced by $\{u_i,y_i\}$, $\{y_i,v_i\},$ $\{y_i,x_i\}$.
  \end{itemize}
The following Figure \ref{Fig1} demonstrates the first change. The second one is similar by adding two other lines connecting $u_i$ (see also Figure \ref{Fig2} for an example). 
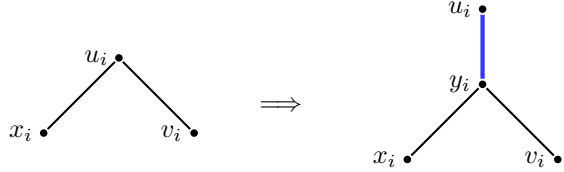
\begin{figure}[h]
	\begin{tikzpicture}[baseline=-20]
		\node[dot] (x1) at (0,0) {};		\node at (-0.3,0) {$u_i$};
		\node[dot] (y1) at (1,-1) {}; \node at (0.7,-1) {$v_i$};
		\node[dot] (y2) at (-1,-1) {};  \node at (-1.3,-1) {$x_i$};
		\draw[C] (x1) to (y1);
	        \draw[C] (x1) to (y2);
	\end{tikzpicture}
	$\qquad \Longrightarrow \qquad$
	\begin{tikzpicture}[baseline=-10]
		\node[dot](x0) at (0,1) {};	\node at (-0.3,1) {$u_i$};
		\node[dot] (x1) at (0,0) {};	\node at (-0.3,0) {$y_i$};
		\node[dot] (y1) at (1,-1) {};\node at (0.7,-1) {$v_i$};
		\node[dot] (y2) at (-1,-1) {};\node at (-1.3,-1) {$x_i$};
		\draw[K] (x0) to (x1);
		\draw[C] (x1) to (y1);
		\draw[C] (x1) to (y2);
	\end{tikzpicture}
	\caption{An illustration for changing from $G$ to $G^e$}
	\label{Fig1}
\end{figure}

The advantage of introducing $K$ can be explained with the example \eqref{e:graph-4pt}. Rewriting \eqref{e:graph-4pt} as 
\begin{equ}
\begin{tikzpicture}[baseline=-15]
	\node[dot] (x1) at (0,0) {}; 
	\node[dot] (x2) at (1,0) {};
	\node[dot] (x3) at (0,-1) {};
	\node[dot] (x4) at (1,-1) {};
	\draw[Phi] (x1) -- ++(-0.1,-0.25);\draw[Phi] (x1) -- ++(0.1,-0.25);
	\draw[Phi] (x2) -- ++(-0.1,-0.25);\draw[Phi] (x2) -- ++(0.1,-0.25);
	\draw[Phi] (x3) -- ++(-0.1,0.25);\draw[Phi] (x3) -- ++(0.1,0.25);
	\draw[Phi] (x4) -- ++(-0.1,0.25);\draw[Phi] (x4) -- ++(0.1,0.25);
\end{tikzpicture}
\quad +\quad
\begin{tikzpicture}[baseline=-15]
	\node[dot] (x1) at (0,0) {};
	\node[dot] (x2) at (1,0) {};
	\node[dot] (x3) at (0,-1) {};
	\node[dot] (x4) at (1,-1) {};
	\node[dot] (z) at (.5,-.3) {};
	\draw[C,bend left=30] (x1) to (z);\draw[C,bend right=30] (x1) to (z);
	\draw[Phi] (z) -- ++(-0.1,-0.25);\draw[Phi] (z) -- ++(0.1,-0.25);
	\draw[Phi] (x2) -- ++(-0.1,-0.25);\draw[Phi] (x2) -- ++(0.1,-0.25);
	\draw[Phi] (x3) -- ++(-0.1,0.25);\draw[Phi] (x3) -- ++(0.1,0.25);
	\draw[Phi] (x4) -- ++(-0.1,0.25);\draw[Phi] (x4) -- ++(0.1,0.25);
\end{tikzpicture}
\quad =\quad
2\;\;\Bigg( \;\;\;
\begin{tikzpicture}[baseline=-15]
	\node[dot] (x1) at (0,0) {}; 
	\node[dot] (x2) at (1,0) {};
	\node[dot] (x3) at (0,-1) {};
	\node[dot] (x4) at (1,-1) {};
	\draw[C,bend left=30] (x1) to (x2);\draw[C,bend right=30] (x1) to (x2);
	\draw[Phi] (x3) -- ++(-0.1,0.25);\draw[Phi] (x3) -- ++(0.1,0.25);
	\draw[Phi] (x4) -- ++(-0.1,0.25);\draw[Phi] (x4) -- ++(0.1,0.25);
\end{tikzpicture}
\quad +\;
\begin{tikzpicture}[baseline=-15]
	\node[dot] (x1) at (0,0) {};
	\node[dot] (x2) at (1,0) {};
	\node[dot] (x3) at (0,-1) {};
	\node[dot] (x4) at (1,-1) {};
	\draw[C,bend left=30] (x1) to (x3);\draw[C,bend right=30] (x1) to (x3);
	\draw[Phi] (x2) -- ++(-0.1,-0.25);\draw[Phi] (x2) -- ++(0.1,-0.25);
	\draw[Phi] (x4) -- ++(-0.1,0.25);\draw[Phi] (x4) -- ++(0.1,0.25);
\end{tikzpicture}
\quad +\;
\begin{tikzpicture}[baseline=-15]
	\node[dot] (x1) at (0,0) {};
	\node[dot] (x2) at (1,0) {};
	\node[dot] (x3) at (0,-1) {};
	\node[dot] (x4) at (1,-1) {};
	\draw[C,bend left=20] (x1) to (x4);\draw[C,bend right=20] (x1) to (x4);
	\draw[Phi] (x2) -- ++(-0.1,-0.25);\draw[Phi] (x2) -- ++(0.1,-0.25);
	\draw[Phi] (x3) -- ++(-0.1,0.25);\draw[Phi] (x3) -- ++(0.1,0.25);
\end{tikzpicture}
\;\;\;\Bigg)
\;\; + (\cdots)
\end{equ}
where $(\cdots)$ stands for $O(\frac1{\sqrt N})$ and
$O(\frac1N)$ terms,
we can ``solve'' the first term by applying $K$:
\begin{equ}
\begin{tikzpicture}[baseline=-15]
	\node[dot] (x1) at (0,0) {}; 
	\node[dot] (x2) at (1,0) {};
	\node[dot] (x3) at (0,-1) {};
	\node[dot] (x4) at (1,-1) {};
	\draw[Phi] (x1) -- ++(-0.1,-0.25);\draw[Phi] (x1) -- ++(0.1,-0.25);
	\draw[Phi] (x2) -- ++(-0.1,-0.25);\draw[Phi] (x2) -- ++(0.1,-0.25);
	\draw[Phi] (x3) -- ++(-0.1,0.25);\draw[Phi] (x3) -- ++(0.1,0.25);
	\draw[Phi] (x4) -- ++(-0.1,0.25);\draw[Phi] (x4) -- ++(0.1,0.25);
\end{tikzpicture}
\quad=\quad
2\;\;\Bigg( \;\;\;
\begin{tikzpicture}[baseline=-15]
	\node[dot] (x1) at (0,0) {}; 
	\node[dot] (x2) at (1,0) {};
	\node[dot] (x3) at (0,-1) {};
	\node[dot] (x4) at (1,-1) {};
	\node[dot] (y) at (.5,0) {};
	\draw[K] (x1) to (y);
	\draw[C,bend left=30] (y) to (x2);\draw[C,bend right=30] (y) to (x2);
	\draw[Phi] (x3) -- ++(-0.1,0.25);\draw[Phi] (x3) -- ++(0.1,0.25);
	\draw[Phi] (x4) -- ++(-0.1,0.25);\draw[Phi] (x4) -- ++(0.1,0.25);
\end{tikzpicture}
\quad +\;
\begin{tikzpicture}[baseline=-15]
	\node[dot] (x1) at (0,0) {};
	\node[dot] (x2) at (1,0) {};
	\node[dot] (x3) at (0,-1) {};
	\node[dot] (x4) at (1,-1) {};
	\node[dot] (y) at (0,-.5) {};
	\draw[K] (x1) to (y);
	\draw[C,bend left=30] (y) to (x3);\draw[C,bend right=30] (y) to (x3);
	\draw[Phi] (x2) -- ++(-0.1,-0.25);\draw[Phi] (x2) -- ++(0.1,-0.25);
	\draw[Phi] (x4) -- ++(-0.1,0.25);\draw[Phi] (x4) -- ++(0.1,0.25);
\end{tikzpicture}
\quad +\;
\begin{tikzpicture}[baseline=-15]
	\node[dot] (x1) at (0,0) {};
	\node[dot] (x2) at (1,0) {};
	\node[dot] (x3) at (0,-1) {};
	\node[dot] (x4) at (1,-1) {};
	\node[dot] (y) at (.5,-.5) {};
	\draw[K] (x1) to (y);
	\draw[C,bend left=20] (y) to (x4);\draw[C,bend right=20] (y) to (x4);
	\draw[Phi] (x2) -- ++(-0.1,-0.25);\draw[Phi] (x2) -- ++(0.1,-0.25);
	\draw[Phi] (x3) -- ++(-0.1,0.25);\draw[Phi] (x3) -- ++(0.1,0.25);
\end{tikzpicture}
\;\;\;\Bigg)
+ (\cdots).
\end{equ}

We denote the set of such graphs $G^e = (V_{G^e}, E_{G^e})$ as $\mathcal{H}_{\ell,q}^k$. We also set $n_\Phi(G^e) = n_\Phi(G)$. Here, $q$ represents the number of times we apply the operators $K$. Additionally, we define $\mathcal{H}_{\ell,q}^{k,i}$ as the set of graphs $G^e = (V_{G^e}, E_{G^e})$ when $G \in \mathcal{H}_{\ell}^{k,i}$ for $i=1,2$.

     Finally, we define $\mathcal{G}^k_{\ell,q}$ as the set of all graphs $G = (V, E) \in \mathcal{H}^k_{\ell,q}$ such that $n_\Phi(G) = 0$.
 %
 	We then write
$$
\mathcal G^k \eqdef \cup_{\ell\ge 0}\cup_{q\geq 0}^{k+\ell} \mathcal G_{\ell,q}^k,
\quad
\mathcal H^{k,i} \eqdef \cup_{\ell\ge 0}\cup_{q\geq 0}^{k+\ell} \mathcal H_{\ell,q}^{k,i}\;\; (i=1,2),
\quad
\mbox{and}
\quad
\cH^k=\cup_{i=1}^2\mathcal H^{k,i}.
$$

For $G^e\in \cH^{k}_{\ell,q}$ we also define
\begin{align}\label{def:IGkK}
	\mathbf{I}_{G^e} (x_1,\dots,x_{k})=\int& \Big(\!\!\!\prod_{\{u,v\}\in E_{G^e}\backslash E_{G^e}^1}\!\!\!\! C_{\eps}(x_u,x_v)\Big) \Big(\!\!\!\prod_{\{u,y\}\in E_{G^e}^1}\!\!\!\! K_{\eps}(x_u,x_y)\Big) \,\no
	\\&\times{N^{-\frac12\big(\sum_{u_m^*\in V_T^{\partial}}\lfloor1-\deg(u_m^*)/2\rfloor+\sum_{z\in V_G^0}\lfloor2-\deg(z)/2\rfloor\big)}}
	\\
	&\E \Big(\prod_{u_m^*\in V^\partial_G}  \Wick{\PPhi_{\eps}^{2-\deg(u_m^*)}}(x_m)
	\prod_{z\in V^0_G} \Wick{\PPhi_{\eps}^{4-\deg(z)}}(x_z)
	\Big)
	\prod_{z\in V^0_G} \dif x_z\prod_{y\in V^s_G} \dif x_y,\,\no
\end{align}
where $K_\eps$ is the discrete kernel of $(I+C_\eps^2*)^{-1}$. We note that for $G\in \cG^k$, $\mathbf{I}_G$ is independent of $N$ and $\Phi$. 

In the following lemma we  prove a $1/\sqrt N$ expansion for $f_{k,\eps}^N$.
Strictly speaking we will not directly use this lemma
and we will eventually be interested in the $1/N$ expansion,
but the proof of this lemma serves as important start-point for the later proofs.

\begin{lemma}\label{lem:FRk}
	For any $k \ge 1$ and $p\geq-1$
	we have the following representation for the $k$-point correlation
	\begin{equ}[e:S-exp-DSk]
		f^N_{k,\eps}  = \sum_{n=0}^p  \frac{1}{N^{n/2}} F_{n,\eps}^k +  \frac{1}{N^{(p+1)/2}} R_{p+1,\eps}^k,
	\end{equ}
	where  the graphs associated with $F_{n,\eps}^k$ belong to $\cG^k$,
	and   the graphs associated with $ R_{p+1,\eps}^k$ belong to $\cH^k$.
	The functions
	$F_{n,\eps}^k $  are independent of $N$.
\end{lemma}

We remark that in the lemma, when $p=-1$,
in which case the `empty' sum in  \eqref{e:S-exp-DSk}
is understood as $0$ by standard convention,
\eqref{e:S-exp-DSk}  trivially holds
which states that  $f^N_{k,\eps} = R_{0,\eps}^k$
where $R_{0,\eps}^k$ can be indeed associated with
a graph in $\cH_0^k$, that is, the graph with only $k$ vertices and no edge.

The proof requires Lemma \ref{lem:IBP} which we present using graphs.
For each $G\in \cH_{\ell,q}^{k,2}$, there are two vertices, denoted by $v_1,v_2$, whose degrees are odd, see 
\eqref{e:Q-graph} for an example. We define a map $\sigma: \cH_{\ell,q}^{k,2} \to \cH_{\ell,q}^{k,1}$ by
\begin{align}\label{sigma}
V_{\sigma (G)} \eqdef V_G,\qquad
E_{\sigma (G)} \eqdef E_G \sqcup \{v_1,v_2\}.
\end{align}
It is easy to see that $n_\Phi(\sigma(G))=n_\Phi(G)-2$. In other words
 for $G\in \cH^{k,2}$, $\mathbf{I}_G$ only involves two $\Wick{\Phi_1\PPhi^2}$ or $\Phi_1$, and 
 $\mathbf{I}_{\sigma(G)}$ means connecting these two vertices by the kernel $C$ and replacing $\Wick{\Phi_1\PPhi^2}$ or $\Phi_1$ by $\Wick{\PPhi^2}$ or $1$, respectively. For instance
the graph \eqref{e:Q-graph} under the map $\sigma$ becomes
 \begin{equ}
\sqrt N\;\;
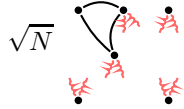
\begin{tikzpicture}[baseline=-15,scale=1.2]
	\node[dot] (x1) at (0,0) {};
	\node[dot] (x2) at (1,0) {};
	\node[dot] (x3) at (0,-1) {};
	\node[dot] (x4) at (1,-1) {};
	\node[dot] (z1) at (.5,0) {};\node[dot] (z2) at (.4,-.5) {};
	\draw[C,bend left=30] (x1) to (z1);\draw[C,bend right=10] (x1) to (z2);
	\draw[C,bend right=30] (z1) to (z2);
	\draw[Phi] (z1) -- ++(0.18,-0.18);\draw[Phi] (z1) -- ++(0,-0.25);
	\draw[Phi] (z2) -- ++(-0.1,-0.25);\draw[Phi] (z2) -- ++(0.1,-0.25);
	\draw[Phi] (x2) -- ++(-0.1,-0.25);\draw[Phi] (x2) -- ++(0.1,-0.25);
	\draw[Phi] (x3) -- ++(-0.1,0.25);\draw[Phi] (x3) -- ++(0.1,0.25);
	\draw[Phi] (x4) -- ++(-0.1,0.25);\draw[Phi] (x4) -- ++(0.1,0.25);
\end{tikzpicture}
\end{equ}

\begin{proof}[Proof of Lemma \ref{lem:FRk}]
We omit $\eps$ in the proof.
	Assume that for a fixed integer $p\ge 0$
	we have already shown that
	\begin{equ}[e:RNk11]
		f_{k}^N  = \sum_{n=0}^{p-1}  \frac1{N^{n/2}} F_{n}^k +  \frac1{N^{p/2}}R_{p}^k,
		\qquad\qquad
		R_{p}^k = \!\!\!\!\!\! \sum_{\substack{G\in \mathcal H^k \\ n_\Phi(G) \in [0,m]\cap 2\mathbb{Z}}}
		\!\!\!\!\!\!\!\! r_G \mathbf{I}_G
	\end{equ}
	for  some $m \in 2\mathbb{Z}$ which may depend on $p$,
and some coefficients $r_G\in \R$ and $r_G$ may involve $\frac1{\sqrt N}$.
	We then prove that the same holds with $p$ replaced by $p+1$, with updated values of $m$ and $r_G$.
	
Now we consider for each $G\in \cH_{\ell,q}^{k,1}$ such that $n_\Phi(G)=m$. Using Lemma \ref{lem:32} the term $ \mathbf{I}_G $ can be written as
\begin{equation}\label{eq:H1k}
	\aligned
	\mathbf{I}_G +C^2*\mathbf{I}_G =&\sum_{\substack{G'\in \cH^{k,1}_{\ell,q}\\ n_\Phi(G')=m-4}}a_{G'}\mathbf{I}_{G'}
	+\frac1{\sqrt N}\Big(\sum_{\substack{G''\in \cH^{k,2}_{\ell,q}\\ n_\Phi(G'')=m-4}}a_{G''}\mathbf{I}_{G''}+\sum_{\substack{G''\in \cH^{k,2}_{\ell+1,q}\\ n_\Phi(G'')=m}}a_{G''}\mathbf{I}_{G''}
	\\&+\sum_{\substack{G''\in \cH^{k,2}_{\ell+2,q}\\ n_\Phi(G'')=m+4}}a_{G''}\mathbf{I}_{G''}\Big)+\frac1N\sum_{\substack{G''\in \cH^{k,1}\\n_\Phi(G'')=m }} a_{G''}\mathbf{I}_{G''},
	\endaligned
\end{equation}
for some coefficients $a_{G'}, a_{G''}\in \mR$ independent of $N$, where $C^2*$ means the convolution of one $\Wick{\PPhi^2}$ with the kernel $C^2$.
Here,  the terms with coefficients $1/\sqrt N$ correspond to the terms in $Q_{N,\mathbf{n}}^\eps$  in Lemma \ref{lem:32}.
Additionally, we express $\frac{N+2}{N}$ in \eqref{eq:inductive} as $1+\frac{2}{N}$, and the terms with coefficients $1/N$ in \eqref{eq:H1k} are derived from the part in \eqref{eq:inductive} with coefficients $\frac{2}{N}$.

Applying the kernel $K$ on both sides of \eqref{eq:H1k}, we then have for $G\in \cH^{k,1}_{\ell,q}$ with $n_\Phi(G)=m$
\begin{equation}\label{e:IF1}
	\aligned
	\mathbf{I}_G  =&\sum_{\substack{G'\in \cH^{k,1}_{\ell,q+1}\\ n_\Phi(G')=m-4}}a_{G'}\mathbf{I}_{G'}
	+\frac1{\sqrt N}\Big(\sum_{\substack{G'\in \cH^{k,2}_{\ell,q+1}\\ n_\Phi(G')=m-4}}a_{G'}\mathbf{I}_{G'}+\sum_{\substack{G''\in H^{k,2}_{\ell+1,q+1} \\n_\Phi(G'')=m}}a_{G''}\mathbf{I}_{G''}
	\\&+\sum_{\substack{G''\in \cH^{k,2}_{\ell+2,q+1}\\ n_\Phi(G'')=m+4}}a_{G''}\mathbf{I}_{G''}\Big)+\frac1N\sum_{G''\in \cH^{k}} a_{G''}\mathbf{I}_{G''},
	\endaligned
\end{equation}
where we simply changed each graph from $\cH^{k,i}_{\cdot,q}, i=1,2,$ to a graph in $\cH^{k,i}_{\cdot,q+1}$, $i=1,2$. When we apply the operator $K$ to both $\mathbf{I}_{G'}$ and $\mathbf{I}_{G''}$, we select the same variable as used in the convolution with $C^2$. In the case where the degree of the point corresponding to this variable is $4$, there are a total of $4$ instances of this variable. However, we only perform the convolution with $K$ on two of these $4$ variables, which are the points from the term $\Wick{\PPhi^2}$.


We note that the first term on the RHS of \eqref{e:IF1} has smaller values of $n_\Phi$.  For $m=4i$ with $i\in\mN$, applying the above procedure $i$ times for the first term on the RHS of \eqref{e:IF1}, we get for $G\in \cH^{k,1}_{\ell,q}$ with $n_\Phi(G)=m$
\begin{align}\label{eq:IG1}
	\mathbf{I}_G  =&\sum_{G'\in \cG^{k}_{\ell,q+i}}a_{G'}\mathbf{I}_{G'}
+\frac1{\sqrt N}\sum_{G''\in \cH^{k}} r_{G''}\mathbf{I}_{G''},
\end{align}
where $a_{G'}$ are independent of $N$ and
we merge the terms with coefficients $\frac1{\sqrt N}$ and $\frac1N$, and $r_{G''}$ may depend on $\frac1{\sqrt N}$.
\footnote{Here and in the sequel, when we say that $r$ may depend on $\frac1{\sqrt N}$, we mean $r=\bar r +\frac1{\sqrt N} \tilde r$ where $\bar r$ and $\tilde r$ are independent of $N$. Statements such as ``depend on $\frac1N$'' etc are understood in similar way.}
Note that $G'\in \cG^{k}_{\ell,q+i}$ here implies that 
$\mathbf{I}_{G'}$ does not involve $\Phi$ and only depends on $C$.
For $m=4i+2$ applying the above procedure $i$ times, we get for $G\in \cH^{k,1}_{\ell,q}$
\begin{equ}
	\mathbf{I}_G  =\Big(\sum_{G'\in \cG^{k}_{\ell-1,q+i}}a_{G'}\mathbf{I}_{G'}+\sum_{G'\in \cG^{k-1}_{\ell,q+i}}a_{G'}\mathbf{I}_{G'}\Big)\frac1{\sqrt N}\E(\Wick{\PPhi^2})
	+\frac1{\sqrt N}\sum_{G''\in \cH^{k}} r_{G''}\mathbf{I}_{G''}
\end{equ}
where $a_{G'}$ are independent of $N$ and  $r_{G''}$ may depend on $1/\sqrt N$.
Using \eqref{Phi2} we have
  $$
  	\E(\Wick{\PPhi^2})=\mathbf{I}_{G'}+\frac1Na'\E(\Wick{\PPhi^2}),
  $$
with  $G'\in \cH^{1}$ and some constant $a'\in\mR$ independent of $N$. Then we can write
\begin{equ}\label{eq:IG2}
	\mathbf{I}_G  =\frac1{\sqrt N}\sum_{G''\in \cH^{k}} r_{G''}\mathbf{I}_{G''},
\end{equ}
with updated $r_{G''}$ and $G''$, where $r_{G''}$ may depend on $\frac1{\sqrt N}$ and $\frac1N$.

For each $G\in \cH_{\ell,q}^{k,2}$ with $n_\Phi(G)=m\in\mN$, using Lemma \ref{lem:IBP} we can write $ \mathbf{I}_G $  as 
\begin{equ}\label{eq:H2k}
	\mathbf{I}_G
	=\mathbf{I}_{\sigma(G)}
	+\frac1{\sqrt N}\Big(\sum_{\substack{G''\in \cH^{k,2}_{\ell,q} \\n_\Phi(G'')=m-2}}a_{G''}\mathbf{I}_{G''}+\sum_{\substack{G''\in \cH^{k,2}_{\ell+1,q}\\ n_\Phi(G'')=m+2}}a_{G''}\mathbf{I}_{G''}
	\Big)+\frac1N a_{\sigma(G)}\mathbf{I}_{\sigma(G)},
\end{equ}
with $\sigma(G)\in \cH^{k,1}_{\ell,q}$ defined in \eqref{sigma} and $a_{\sigma(G)}=2$ or $0$. 
By substituting the RHS of  \eqref{eq:IG1} and \eqref{eq:IG2} into the first term of  \eqref{eq:H2k}, we observe that equations \eqref{eq:IG1} and \eqref{eq:IG2} also remain valid for $G \in \cH^{k,2}_{\ell,q}$.
Substituting \eqref{eq:IG1} and \eqref{eq:IG2} into $R_p^k$ in \eqref{e:RNk11} and recalling again that 
for any $G\in \cG^k$ the corresponding 
$\mathbf{I}_G$ is independent of $N$ and $\Phi$, the result follows.
\end{proof}

To obtain the $\frac{1}{N}$ expansion, we need to delve deeper into IBP and analyze the parity (oddness or evenness) of the number of $\PPhi^2$ terms present in $\mathbf{I}_G$. 
To this end we denote by $n_G$ the number of $\PPhi^2$ factors in $\mathbf{I}_G$.
We also view $\Phi_1\PPhi^2$ as involving one $\PPhi^2$ term. 
It is straightforward to observe that $n_G$ can be equivalently defined as
\begin{equs}
n_G&=\frac{1}{2}n_\Phi(G),\qquad\qquad\quad \mbox{for }G\in \cH^{k,1} \quad \mbox{(i.e. $\mathbf{I}_G$ has  no $\Wick{\Phi_1\PPhi^2}$)},
\\
n_G&=\frac{1}{2}(n_\Phi(G)-2), \qquad \mbox{for }G\in \cH^{k,2}  \quad \mbox{(i.e. $\mathbf{I}_G$ has  two $\Wick{\Phi_1\PPhi^2}$ or $\Phi_1$)}.
\end{equs}

In the following proof, we will omit the subscript in $\cH^{k,i}$ and focus more on the parity 
 of $n_G$. In the sequel we also use $a_{G'}, a_{G''}$ as suitable constants independent of $N$, which may change from line to line.

\bl  \label{lem:zq1}
For $G\in \cH^{k}, k\in\mN,$ with $n_G=m\in\mN$, it holds that for every $n\in\mN$
\begin{equation}\label{zq:1}
	\aligned
	\mathbf{I}_G =&\sum_{\substack{ G'\in \cH^{k,1} \\n_{G'}=m-2}}a_{G'}\mathbf{I}_{G'}
	+\sum_{i=1}^{n+1}\frac1{ N^{i/2}}\sum_{\substack{ G''\in \cH^{k} \\n_{G''}\in 2\mZ+m-i}}a_{G''}\mathbf{I}_{G''},
	\endaligned
\end{equation}
for suitable $a_{G'}, a_{G''}\in \mR$ independent of $N$.
Here all the sums are over  finitely many terms since $k$ is fixed and at each vertex the degree plus the number of wavy lines is not allowed to exceed $4$.
\el

Remark that although the above lemma is proved for 
arbitrary $n$, we will only need $n=1,2$ later.

\begin{proof}
	We first prove that for $G\in \cH^{k,1}$ with $n_G=m$ 
	\begin{equ}\label{zq}
		\mathbf{I}_G +C^2*\mathbf{I}_G 
		=\sum_{\substack{G'\in \cH^{k,1}\\ n_{G'}=m-2}}
			a_{G'}\mathbf{I}_{G'}
		+\sum_{i=1}^{n+1}\frac1{ N^{i/2}}\!\!\!\!\!\sum_{\substack{G''\in \cH^{k,1} \\n_{G''}\in2\mZ+m-i}} \!\!\!\!\!\! a_{G''}\mathbf{I}_{G''}
		+\frac1{ N^{n/2}}\!\!\!\!\!\! \sum_{\substack{G''\in \cH^{k,2}\\ n_{G''}\in2\mZ+m-n}} \!\!\!\!\!\! a_{G''}\mathbf{I}_{G''}.
	\end{equ}
	We write \eqref{eq:H1k} for $G\in \cH^{k,1}$ with $n_G=m$ as
\begin{equation}\label{z:1}
	\aligned
		\mathbf{I}_G +C^2*\mathbf{I}_G
		=&\sum_{\substack{G'\in \cH^{k,1} \\n_{G'}=m-2}}a_{G'}\mathbf{I}_{G'}
	+\frac1{\sqrt N}\sum_{j=-1,1,3}\sum_{\substack{G''\in \cH^{k,2}\\ n_{G''}=m-j}}a_{G''}\mathbf{I}_{G''}\\&+\frac1N\sum_{\substack{G''\in \cH^{k,1}\\n_{G''}=m}} a_{G''}\mathbf{I}_{G''}.
	\endaligned
\end{equation}
Hence,  \eqref{zq} holds with $n=1$. In the following, we suppose that \eqref{zq} holds and prove that it holds with $n$ replaced by $n+1$.

We write \eqref{eq:H2k} for $G\in \cH^{k,2}$ with $n_G=m$
\begin{align}\label{z:2}
	&\mathbf{I}_G
	=\mathbf{I}_{\sigma(G)}
	+\frac1{\sqrt N}\sum_{j=-1,1}\sum_{\substack{G''\in \cH^{k,2},\\ n_{G''}=m-j}}a_{G''}\mathbf{I}_{G''}+\frac1Na_{\sigma(G)}\mathbf{I}_{\sigma(G)}.
\end{align}
From  \eqref{z:1} and \eqref{z:2}, we observe the following patterns:
For the order $\frac{1}{\sqrt{N}}$  terms, when we perform IBP we either add or eliminate an odd number of $\PPhi^2$ factors in $\mathbf{I}_{G'}$. 
This changes the parity of $n_G$.
On the other hand, for the order $\frac{1}{N}$  term, when we perform IBP we either add or eliminate an even number of $\PPhi^2$ factors in $\mathbf{I}_{G'}$. 
This does not change  the parity of $n_G$.

These observations show a relation between the order of the term and the change in the parity of $n_G$, indicating how the number of $\PPhi^2$ terms is affected during the expansion.

Using \eqref{z:2}, we write the last term in \eqref{zq}  as
\begin{equ}
\frac1{ N^{n/2}} \!\!\!\!\!\! \sum_{\substack{G''\in \cH^{k,1}\\ n_{G''}\in 2\mZ+m-n}} \!\!\!\!\!\! a_{G''}\mathbf{I}_{G''}
+\frac1{ N^{(n+1)/2}} \!\!\!\!\!\!\!\! \sum_{\substack{G''\in \cH^{k,2}\\ n_{G''}\in2\mZ+m-n-1}} \!\!\!\!\!\!\!\! a_{G''}\mathbf{I}_{G''}
+\frac1{ N^{(n+2)/2}}\!\!\!\!\!\!\!\! \sum_{\substack{G''\in \cH^{k,1}\\ n_{G''}\in 2\mZ+m-2-n}} \!\!\!\!\!\!\!\! a_{G''}\mathbf{I}_{G''},
\end{equ}
where we take sum for finite terms.
Using this to replace the last terms in \eqref{zq}, 
we obtain  \eqref{zq} with $n$ replaced by $n+1$.
Applying the operator $K$ on both sides of \eqref{zq} and noting $\cH^k=\cH^{k,1}\cup \cH^{k,2}$, 
we obtain \eqref{zq:1}  for $G\in \cH^{k,1}$.

 For $G\in \cH^{k,2}$ with $n_G=m$ we use \eqref{z:2} and similar induction as above to have
\begin{equ}
	\mathbf{I}_G =\sum_{\substack{G'\in \cH^{k,1}\\ n_{G'}=m}}a_{G'}\mathbf{I}_{G'}
	+\sum_{i=1}^{n+1}\frac1{ N^{i/2}} \!\!\!\!\!\! \sum_{\substack{G''\in \cH^{k,1}\\ n_{G''}\in 2\mZ+m-i}} \!\!\!\!\!\! a_{G''}\mathbf{I}_{G''}
	+\frac1{ N^{n/2}} \!\!\!\!\!\! \sum_{\substack{G''\in \cH^{k,2}\\ n_{G''}\in 2\mZ+m-n}} \!\!\!\!\!\! a_{G''}\mathbf{I}_{G''}.
\end{equ}
Applying \eqref{zq:1} to the $\mathbf{I}_{G'}$ in the first term of the RHS, we can lower the value of $n_{G'}$ and we obtain \eqref{zq:1} for $G\in \cH^{k,2}$.
%
\end{proof}

In the following, we will apply  \eqref{zq:1}  of Lemma~\ref{lem:zq1} with $n=1$ and $n=2$ to derive a more detailed decomposition of $\mathbf{I}_G$ for $G \in \cH^k$. This refined decomposition will be useful for the induction argument in the $1/N$ expansion.

\bl 
For $G\in \cH^{k}$ with $n_G\in 2\mZ$, one has
\begin{equ}\label{zmm:1}
\mathbf{I}_G
=\sum_{G'\in \cG^{k}}a_{G'}\mathbf{I}_{G'}
+\frac1N\sum_{\substack{G''\in \cH^{k} \\n_{G''}\in  2\mZ}}
  a_{G''}\mathbf{I}_{G''}
+\frac1{N^{3/2}} \!\!\!\!\! \sum_{\substack{G''\in \cH^{k}\\ n_{G''}\in 2\mZ+1 }}\!\!\!\!\! a_{G''}\mathbf{I}_{G''}+\frac1{N^2}\E(\Wick{\PPhi^2})\sum_{G'\in \cG}a_{G'}\mathbf{I}_{G'}
\end{equ}
for suitable $a_{G'}, a_{G''}\in \mR$ independent of $N$.
For $G\in \cH^{k}$ with $n_G\in 2\mZ-1$, one has
\begin{equ}\label{zmm:2}
	\mathbf{I}_G=\frac1{\sqrt N}\sum_{\substack{G'\in \cH^{k}\\ n_G\in  2\mZ}}a_{G'}\mathbf{I}_{G'}+\frac1{N}\sum_{\substack{G''\in \cH^{k}\\ n_{G''}\in  2\mZ+1}}a_{G''}\mathbf{I}_{G''}+\frac1{N^{3/2}}\E(\Wick{\PPhi^2})\sum_{G'\in \cG}a_{G'}\mathbf{I}_{G'}
\end{equ}
for suitable $a_{G'}, a_{G''}\in \mR$ independent of $N$.
The sums above are all over finitely many terms.
\el
\begin{proof}
Denote $m=n_G$ for the graph $G$ given in the lemma.
We  consider the case $m\in 2\mN$,
 and apply \eqref{zq:1} $m/2$ times for the first term on the RHS of \eqref{zq:1} with $n=2$
(i.e. iteratively substitute $\mathbf{I}_{G'}$ therein by the RHS of \eqref{zq:1}). 
We have
\begin{equation}\label{zq:2}
	\aligned
	\mathbf{I}_G =&\sum_{G'\in \cG^{k}}a_{G'}\mathbf{I}_{G'}
	+\sum_{i=1}^{3}\frac1{ N^{i/2}}\sum_{\substack{G''\in \cH^{k}\\n_{G''}\in  2\mZ+i}}a_{G''}\mathbf{I}_{G''}
	\endaligned
\end{equation}
Here we get  some  terms with coefficients $\frac1{N^{i/2}}, i=1,2,3$ coming from the repetitive applications of \eqref{zq:1} to the first term in \eqref{zq:1}, which can be incorporated into the corresponding terms in \eqref{zq:1} with the same coefficients $\frac1{N^{i/2}}, i=1,2,3$.
In the first term we have $G'\in \cG^{k}$ now, because each iteration decreases $n_{G'}$ by $2$.

From \eqref{zq:2} we find that for $i$ odd $n_{G''}$ is also odd for the associated $\mathbf{I}_{G''}$ with coefficient $\frac1{N^{i/2}}$.
As noted in \eqref{eq:IG2} and \eqref{eq:H2k} for $n_G$ odd, $\mathbf{I}_G$ is of order $\frac1{\sqrt N}$.   This observation is helpful because it allows us to convert these terms to the order of $\frac{1}{N^{(i+1)/2}}$.

Turning to the case  $m\in 2\mN-1$, we apply \eqref{zq:1} $(m-1)/2$ times for the first term on the RHS of \eqref{zq:1} with $n=1$  to have
\begin{equation}\label{zq:3}
	\aligned
	\mathbf{I}_G =&\sum_{\substack{G'\in \cH^{k}\\ n_G=1}}a_{G'}\mathbf{I}_{G'}
	+\sum_{i=1}^{2}\frac1{ N^{i/2}}\sum_{\substack{G''\in \cH^{k}\\ n_{G''}\in  2\mZ+i-1}}a_{G''}\mathbf{I}_{G''}.
	\endaligned
\end{equation}
Using \eqref{Phi2} we have 
\begin{equ}\label{e:z1}
\frac1{\sqrt N}	\E(\Wick{\PPhi^2})
=\frac1{\sqrt N}\mathbf{I}_{G'}  
+\frac1{N^{3/2}}a'\E(\Wick{\PPhi^2}),
\end{equ}
with $G'\in \cH^{1,2}, n_{G'}=2$ given by
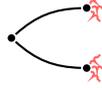
\begin{equ}
\begin{tikzpicture}
	\node[dot] (x) at (0,0) {};
	\node[dot] (y) at (1,0.4) {};
	\node[dot] (z) at (1,-0.4) {};
	\draw[C,bend right=20] (x) to (z);
	\draw[C,bend left=20] (x) to (y);
		\draw[Phi] (y) -- ++(0.18,0.18);
	\draw[Phi] (y) -- ++(0.25,0);
	\draw[Phi] (y) -- ++(0.18,-0.18);
	\draw[Phi] (z) -- ++(0.18,0.18);
	\draw[Phi] (z) -- ++(0.25,0);
	\draw[Phi] (z) -- ++(0.18,-0.18);
\end{tikzpicture}
\end{equ}
and this applied to  the first term on the RHS of \eqref{zq:3} implies \eqref{zmm:2}.
Using \eqref{zmm:2}  for the terms $\mathbf{I}_{G''}$ in \eqref{zq:2} with coefficient $\frac1{N^{1/2}}$ we obtain \eqref{zmm:1}.
\end{proof}

Now we are in a position to prove $1/N$ expansion of $f^{N}_{k,\eps}$.

\begin{lemma}\label{lem:FRkn}
	For any $k \ge 1$ and $p\geq1$
	we have the following representation for the $k$-point correlation
	\begin{equ}[ex:1]
		f^N_{k,\eps}  = \sum_{n=0}^p  \frac{1}{N^{n}} F_{n,\eps}^{k,1} +  \frac{1}{N^{p+1}} R_{p+1,\eps}^{k,1},\quad k\in 2\mN,
	\end{equ}
and
	\begin{equ}[ex:2]
	f^N_{k,\eps}  = \sum_{n=0}^p  \frac{1}{N^{n+1/2}} F_{n,\eps}^{k,2} +  \frac{1}{N^{p+3/2}} R_{p+1,\eps}^{k,2},\quad k\in 2\mN-1,
\end{equ}
	where the graphs associated with $F_{n,\eps}^{k,1}, F_{n,\eps}^{k,2}$ belong to $\cG^k$,
	and the functions
	$F_{n,\eps}^{k,1}$, $F_{n,\eps}^{k,2}$  are independent of $N$.

Moreover 	$R_{p+1,\eps}^k$ can be decomposed as the sum of finitely many terms of the form 
$b_G\mathbf{I}_G$ with $n_G\in 2\mZ$ 
and $\frac1{\sqrt N}b_{G'}\mathbf{I}_{G'}$ with $n_{G'}\in 2\mZ+1$ 
and $\frac1{N}\mathbf{I}_{G'}\E(\Wick{\PPhi^2})$. Here $b_G$ and $b_{G'}$ are constants independent of $N$. 
\end{lemma}
\begin{proof}We omit $\eps$ in the proof. By \eqref{zmm:1}, one has \eqref{ex:1} for $p=0$. Now we prove \eqref{ex:1} by induction.
Assume that for a fixed integer $p\ge 0$
we have already shown that
\begin{equs}[e:RNk1n]
	f_{k}^N  &= \sum_{n=0}^{p-1}  \frac1{N^{n}} F_{n}^{k,1} +  \frac1{N^{p}}R_{p}^{k,1},
	\\
	R_{p}^{k,1} &= \!\!\!\!\!\!\!\! 
	\sum_{\substack{G\in \mathcal H^k \\ n_G \in [0,m_1]\cap 2\mathbb{Z}}}
		\!\!\!\!\!\!\!\! b_G \mathbf{I}_G
	+\frac1{\sqrt N} \!\!\!\!\!\!\!\! 
	\sum_{\substack{G'\in \mathcal H^k \\ n_{G'}\in [0,m_2]\cap (2\mathbb{Z}+1)}}
		\!\!\!\!\!\!\!\! b_{G'} \mathbf{I}_{G'}+\frac1{N}\E(\Wick{\PPhi^2})\sum_{G'\in \cG}a_{G'}\mathbf{I}_{G'},\quad k\in 2\mN,
	\end{equs}
	for  some $m_1 \in 2\mathbb{N}, m_2\in 2\mN-1$ which may depend on $p$,
	and some coefficients $b_G, b_{G'}, a_{G'}\in \R$ independent of $N$.
	We then prove that the same holds with $p$ replaced by $p+1$, with updated values of $m_1, m_2$ and $b_G, b_{G'},  a_{G'}$.
	In \eqref{e:RNk1n}, we apply \eqref{zmm:1} to $\mathbf{I}_G$, apply   \eqref{zmm:2} to $\mathbf{I}_{G'}$, and apply \eqref{e:z1} to $\E(\Wick{\PPhi^2})$. We then obtain
	\begin{equ}
		R_{p}^{k,1} 
		= \sum_{G\in \mathcal G^k }
		 b_G \mathbf{I}_G
		 +\frac1N \!\!\!\!\!\! \sum_{\substack{G\in \mathcal H^k \\ n_G \in [0,m_1]\cap 2\mathbb{Z}}}
		\!\!\!\!\!\!\!\! b_G \mathbf{I}_G
		+\frac1{ N^{3/2}}\!\!\!\!\!\!\!\!
		 \sum_{\substack{G\in \mathcal H^k \\ n_G\in [0,m_2]\cap 2\mathbb{Z}+1}}
		\!\!\!\!\!\!\!\! b_G \mathbf{I}_G+\frac1{N^2}\E(\Wick{\PPhi^2})\sum_{G'\in \cG}a_{G'}\mathbf{I}_{G'}.
	\end{equ}
	We plug this into the first equality in \eqref{e:RNk1n} and  obtain that \eqref{e:RNk1n} holds with $p+1$. Using the fact that $\mathbf{I}_G, G\in \cG$ only depends on $C$, $K$ and is independent of $N$,  \eqref{ex:1} follows.
	
For $k$ odd, we first prove that \eqref{ex:2} holds with $p=0$. Using \eqref{zmm:2} we have
\begin{align*}
	f_{k}^N=\frac1{\sqrt N}\sum_{\substack{G'\in \cH^{k},\\ n_G\in  2\mZ}}a_{G'}\mathbf{I}_{G'}+\frac1{N}\sum_{\substack{G''\in \cH^{k},\\ n_{G''}\in  2\mZ+1}}a_{G''}\mathbf{I}_{G''}+\frac1{N^{3/2}}\E(\Wick{\PPhi^2})\sum_{G'\in \cG}a_{G'}\mathbf{I}_{G'},
\end{align*}
Applying \eqref{zmm:2} again for $\mathbf{I}_{G''}$ and using \eqref{e:z1} to replace $\E(\Wick{\PPhi^2})$, we obtain for $k$ odd
\begin{align*}
f_{k}^N
&=\frac1{\sqrt N}\sum_{\substack{G'\in \cH^{k}, \\n_G\in  2\mZ}}a_{G'}\mathbf{I}_{G'}
+\frac1{ N^{3/2}}\sum_{\substack{G'\in \cH^{k},\\ n_G\in  2\mZ}}a_{G'}\mathbf{I}_{G'}
\\
&+\frac1{N^{2}} \!\!\!\!\! \sum_{\substack{G''\in \cH^{k},\\ n_{G''}\in  2\mZ+1}}\!\!\!\!\! a_{G''}\mathbf{I}_{G''}+\frac1{N^{5/2}}\E(\Wick{\PPhi^2})\sum_{G'\in \cG}a_{G'}\mathbf{I}_{G'},
\end{align*}
with updated  cofficients $a_{G'}, a_{G''}$ independent of $N$.
Applying \eqref{zmm:1} for $\mathbf{I}_{G'}$ from the first term on the RHS, we obtain  \eqref{ex:2}  with $p=0$.

We then	assume that for a fixed integer $p\ge 0$ 
we have already shown that
\begin{equs}[e:RNk2n]
f_{k}^N  &= \sum_{n=0}^{p}  \frac1{N^{n+1/2}} F_{n}^{k,2} +  \frac1{N^{p+3/2}}R_{p}^{k,2},
		\\
R_{p+1}^{k,2} 
&= \!\!\!\!\!\!\!\! \sum_{\substack{G\in \mathcal H^k \\ n_G \in [0,m_1]\cap 2\mathbb{Z}}}
		\!\!\!\!\!\!\!\! b_G \mathbf{I}_G+\frac1{\sqrt N}
		\!\!\!\!\!\!\!\! \sum_{\substack{G'\in \mathcal H^k \\ n_{G'}\in [0,m_2]\cap (2\mathbb{Z}+1)}}
		\!\!\!\!\!\!\!\! b_{G'} \mathbf{I}_{G'}+\frac1{N}\E(\Wick{\PPhi^2})\sum_{G'\in \cG}a_{G'}\mathbf{I}_{G'},\quad k\in 2\mN,
\end{equs}
	for  some $m_1 \in 2\mathbb{N}, m_2\in 2\mN-1$ which may depend on $p$,
	and some coefficients $b_G, b_{G'}\in \R$.
Using \eqref{zmm:1}, \eqref{zmm:2} and \eqref{e:z1} again, \eqref{e:RNk2n} holds with $p$ replaced by $p+1$, with updated values of $m_1, m_2$ and $b_G, b_{G'}$. Hence, \eqref{ex:2} follows.
\end{proof}

To extend Lemma \ref{lem:FRkn} to the continuum setting, and prove Proposition~\ref{pr:N} below, we follow a similar approach as presented in \cite{SZZ21}. The proof strategy involves reducing each graph from Lemma \ref{lem:FRkn} to a tree 
and then employing inductive arguments to obtain the desired estimates. 
Similarly as in \cite{SZZ21} we repeatedly apply IBP \eqref{e:IBP-CN}; and in this procedure, we  use red color for the new line appearing for the last term in \eqref{e:IBP-CN} and use green color for the new line appearing for the first term on the LHS of \eqref{e:IBP-CN}.  However, we also have new operators $K$ here, which correspond to the new lines in $E^1_{G^e}$, and we color them in blue. As mentioned in Remark \ref{rem}, here in our case, we only apply the versions of IBP as given by  Lemma \ref{lem:IBP}, Lemma \ref{lem:32} and Lemma \ref{lem:1}.
As a consequence, 
\begin{itemize}
	\item we use red color for the new line  for $C$   corresponding to $\cI$ in  \eqref{e:defO} and \eqref{e:same} and  
	Eq.~\eqref{Phi2};
	\item  we  use green color for the other new line for  $C$ from   \eqref{e:defO} and \eqref{e:same};
	\item we use blue color for the new lines in $E_{G^e}^1$ corresponding to $K$.
\end{itemize}
Since the aforementioned lemmas are all applications of IBP \eqref{e:IBP-CN}, the coloring rule here
is essentially the same as \cite{SZZ21}, except that we have a new blue color.
We emphasize that we do not change the  color of the existing lines when transitioning from $G=(V,E)$ to $G^e=(V^{G^e},E^{G^e})$; we simply color the additional  lines in blue.
 In comparison to \cite{SZZ21}, the LHS of \eqref{eq:inductive} produces a new operator called $K$ which remains unchanged.
 
  A tree is defined as a graph that does not include the green line. By replacing $C_\eps(x-y)$ with $\E (Z_\eps'(x)Z_\eps'(y))$, where $Z_\eps' \overset{d}{=} Z_\eps$ and $Z_\eps'$ is independent of $\Phi_\eps$, we can transform each graph $G$ into $k$ trees without changing the value of $\mathbf{I}_G$. (See \cite[Lemma 3.7]{SZZ21}).

We first pass the result obtained above to continuum.

\bl\label{convergencek} It holds that for $\kappa>0$ and $k\in \mN$, $n, p\in \mN$
$$
\lim_{\eps\to0}\cE^\eps_k F_{n,\eps}^{k,i}=F_{n}^{k,i} \quad\textrm{ in }\quad (H^{-\kappa})^k,\quad i=1,2,
$$
and
$$\lim_{\eps\to0}\cE^\eps_kR_{p,\eps}^{k,i}=R_{p}^{k,i} \quad\textrm{ in } \quad (H^{-\kappa})^k,\quad i=1,2,$$
Here   $F_n^{k,i}, i=1,2,$ can be written as integrals of the Green function  $C$ of $\m-\Delta$ and the kernel $K$ from Lemma \ref{lem:K}. The functions $R_p^{ k,i}, i=1,2,$  depend on $C$, $k$ and $\Phi$. The associated graph of $F_n^{k,i}, i=1,2,$ are the same as $F_{n,\eps}^{k,i}$  with $C_{\eps}$  in \eqref{def:IGk} and \eqref{def:IGkK} replaced by the Green's function  $C$ and the sum over $\Lambda_{\eps}$ replaced by the integral over $\mT^2$. The functions $F_n^{k,i}, i=1,2$ are independent of $N$. 
\el
We give the proof after Proposition \ref{pr:N}.

Now we are in a position to state the main results of this section.

\bt\label{m:5}
	For any $k \ge 1$ and $p\geq1$
	we have the following representation for the $k$-point correlation
	\begin{equ}
		f^N_{k}  = \sum_{n=0}^p  \frac{1}{N^{n}} F_{n}^{k,1} +  \frac{1}{N^{p+1}} R_{p+1}^{k,1},\quad k\in 2\mN,
	\end{equ}
	and
	\begin{equ}
		f^N_{k}  = \sum_{n=0}^p  \frac{1}{N^{n+1/2}} F_{n}^{k,2} +  \frac{1}{N^{p+3/2}} R_{p+1}^{k,2},\quad k\in 2\mN-1,
	\end{equ}
	where $F_{n}^{k,1}, F_{n}^{k,2}$ and $R_{p+1}^{k,1}, R_{p+1}^{k,2}$ are given in Lemma \ref{convergencek}. The equality holds in $(H^{-\kappa})^k$ for $\kappa>0$.
\et
\begin{proof}
	The result follows by applying the extension opeartor $\cE_k^\eps$ defined in \eqref{def:Ek} on both sides of \eqref{ex:1} and \eqref{ex:2} and using Lemma \ref{convergencek}.
\end{proof}

Now to prove Theorem \ref{main} it remains to prove the uniform in $N$ bounds for $R_{p+1}^{k,i}, i=1,2$. 

\bp\label{pr:N} 
Let $\m$ as in Lemma \ref{th:m2}. It holds that for every $p, k\in \mN$ and $\kappa>0$
\begin{align*}
	\|R_{p+1}^{k,1}\|_{(H^{-\kappa})^k}+\|R_{p+1}^{k,2}\|_{(H^{-\kappa})^k}\lesssim1,
\end{align*}
where the implicit constant is independent of $N$.
\ep
\begin{proof}
	It suffices to prove that $\mathbf{I}_G$ for each graph $G$ showing up in the expression of $R_{p+1}^{k,1},R_{p+1}^{k,2}$ is of order $1$, 
	since $\E(\Wick{\PPhi^2})$ is of order $1$ and all the coefficients
	$b_{G'}$ are independent of $N$. 
The proof follows similarly as in \cite[Lemma 4.2]{SZZ21}. We omit $\eps$ in the proof.
	
	\medskip\medskip\medskip
	\textbf{Step 1. Reduction to trees.}
	
	We first introduce the stochastic objects when reducing each graph to trees.
	When we replace the line $C(x-y)$ by $\E(Z^{(i)}(x)Z^{(i)}(y))$ with $\{Z^{(i)}\}_{i\in \mN}$ being i.i.d. random variables and $\{Z^{(i)}\}_{i\in \mN}$ being an independent copy of $\{Z_i\}_{i\in\mN}$, we also encounter  the following stochastic objects
\begin{equs}[st:1]
Z^{(i)},\qquad\text{and}& \qquad	\Wick{Z^{(i)}\Phi_1},\quad \Wick{Z^{(i)}Z^{(j)}},
\\
\text{and}\qquad \frac1{\sqrt N}\Wick{Z^{(i)}\PPhi^2},\quad &\Wick{Z^{(i)}Z^{(j)}\Phi_1},\quad \Wick{Z^{(i)}Z^{(j)}Z^{(k)}}
\end{equs}
where $i,j,k$ are distinct
(see Figure~\ref{Fig2} for an example.)
These stochastic objects play the same role as 
$\Phi_1$, $\frac1{\sqrt N}\Wick{\PPhi^2}$ and $\frac1{\sqrt N}\Wick{\Phi_1\PPhi^2}$ in $\mathbf{I}_G$, respectively.
The Wick products here are understood as
	\begin{align*}
		\Wick{Z^{(i)}\Phi_1}\eqdef Z^{(i)}\Phi_1,\qquad
		 \Wick{Z^{(i)}Z^{(j)}} & \eqdef Z^{(i)}Z^{(j)}, \qquad 
		 \Wick{Z^{(i)}\PPhi^2}\eqdef Z^{(i)}\Wick{\PPhi^2},
		 \\
		 \Wick{Z^{(i)}Z^{(j)}\Phi_1}\eqdef Z^{(i)}Z^{(j)}\Phi_1,\qquad &\Wick{Z^{(i)}Z^{(j)}Z^{(k)}}\eqdef Z^{(i)}Z^{(j)}Z^{(k)}.
	\end{align*}
We set $\sD\eqdef\cup_{i=0}^3\sD_i$ with $\sD_0\eqdef\{1\}$  and 
\begin{align*}
	\sD_1\eqdef\{Z^{(i)}, \Phi_1\},\quad \sD_2\eqdef\Big\{\Wick{Z^{(i)}\Phi_1},\Wick{Z^{(i)}Z^{(j)}}, \frac1{\sqrt N}\Wick{\PPhi^2}\Big\},
	\\\sD_3\eqdef \Big\{\frac1{\sqrt N}\Wick{Z^{(i)}\PPhi^2}, \Wick{Z^{(i)}Z^{(j)}\Phi_1}, \Wick{Z^{(i)}Z^{(j)}Z^{(k)}},\frac1{\sqrt N}\Wick{\Phi_1\PPhi^2}\Big\}.
\end{align*}
	By similar calculations as in Proposition \ref{prop1} we have for $\kappa>0$,  $\ell\geq 1$ and $f\in \sD$
\begin{equation}\label{sto:2}
	\aligned
	\E	\|f\|_{H^{-\kappa}}^\ell\lesssim1.
	\endaligned
\end{equation}

	 We then reduce each graph $G$ to a disjoint union of $k$ trees $\sqcup_{i=1}^k T_i$ by replacing the green lines with $\E(Z^{(i)}(x)Z^{(i)}(y))$.
	 Then, 
	$\mathbf{I}_G $ is the expectation
	of a product of $k$ functions,
	and each of these  $k$ functions have the following form
	\begin{equ}[e:F_T]
		F_{T}(x_{u^*})=\int \prod_{\{u,v\}\in E_{T}}D(x_u,x_v)
		\Big(
		\prod_{v\in V_{T}}f_v(x_v) g_v^T(x_v)
		\Big)
		\prod_{v\in V_{T}\backslash \{u^*\}} \dif x_v,
	\end{equ}
	with $D=C$ or $K$,  $f_v\in \sD$, $g_v^T=1$.
	Here and below we just write $T$ for $T_i$ to simplify the notation,
	and we have introduced the function $g_v^T$ for the purpose of induction later.
	
	 For each tree we will  prove by induction that for $\kappa>0$
	 \begin{align}\label{est:FT}
	 	\|F_T\|_{H^{-\kappa}}\lesssim \prod_{v\in V_T} \|f_v\|_{H^{-\kappa}}
	 \end{align}  
 with $f_v\in\sD$. Assuming this, we have
 \begin{align*}
 	\|\mathbf{I}_G\|_{(H^{-k\kappa})^k}\lesssim \E\prod_{i=1}^k\|F_{T_i}\|_{H^{-\kappa}}\lesssim\E\prod_{i=1}^k\prod_{v\in V_{T_i}} \|f_v\|_{H^{-\kappa}}.
 \end{align*}
 The result then follows from \eqref{sto:2} and H\"older's inequality. 
	 
Figure~\ref{Fig2} explains the procedure to turn a graph into a tree, where each tiny green wavy line denotes a factor of  $Z^{(i)}$.

	\begin{figure}[h]
		\begin{tikzpicture}[baseline=-60]
			\node[dot] (x1) at (0,0) {};	\node at (-0.3,0) {$u_1^*$};
			\node[dot] (x2) at (2,0) {};\node at (2.3,0) {$u_2^*$};
			\node[dot] (y1) at (1,-1) {};
			\node[dot] (z1) at (1,-2) {};
			\node[dot] (z2) at (0,-3) {};
			\node[dot] (z3) at (2,-3) {};
			\node[dot] (z4) at (-0.5,-4) {};
			\node[dot] (z5) at (0.5,-4) {};
			\draw[Cr] (x1) to (y1);
			\draw[K] (y1) to (z1);
			\draw[Cr] (z1) to (z2);
			\draw[Cr] (z1) to (z3);
			\draw[Cr] (z2) to (z4);
			\draw[Cr] (z2) to (z5);
			\draw[Cg] (x1) to (x2);
			\draw[Cg] (x2) to (y1);
			\draw[Cg] (z2) to (z3);
			\draw[Cg] (z4) to (z5);
				\draw[Phi] (z3) -- ++(0.18,-0.18);\draw[Phi] (z3) -- ++(-0.18,-0.18);
			\draw[Phi] (z4) -- ++(0.18,-0.18);\draw[Phi] (z4) -- ++(-0.18,-0.18);
			\draw[Phi] (z5) -- ++(0.18,-0.18);\draw[Phi] (z5) -- ++(-0.18,-0.18);
		\end{tikzpicture}
		$\qquad \Longrightarrow \qquad$
		\begin{tikzpicture}[baseline=-60]
			\node[dot] (x1) at (0,0) {};	\node at (-0.4,0) {$u_1^*$};
			\node[dot] (x2) at (2,0) {};\node at (2.3,0) {$u_2^*$};
			\node[dot] (y1) at (1,-1) {};
			\node[dot] (z1) at (1,-2) {};
			\node[dot] (z2) at (0,-3) {};
			\node[dot] (z3) at (2,-3) {};
			\node[dot] (z4) at (-0.5,-4) {};
			\node[dot] (z5) at (0.5,-4) {};
			\draw[Cr] (x1) to (y1);
			\draw[K] (y1) to (z1);
			\draw[Cr] (z1) to (z2);
			\draw[Cr] (z1) to (z3);
			\draw[Cr] (z2) to (z4);
			\draw[Cr] (z2) to (z5);
			\draw[Z] (x1) -- ++(-0.18,-0.18);
			\draw[Z] (x2) -- ++(0.18,-0.18);\draw[Z] (x2) -- ++(-0.18,-0.18);
			\draw[Z] (y1) -- ++(0.18,-0.18);
			\draw[Z] (z2) -- ++(0,-0.25);
			\draw[Phi] (z3) -- ++(0.18,-0.18);\draw[Phi] (z3) -- ++(-0.18,-0.18);\draw[Z] (z3) -- ++(0,-0.25);
			\draw[Phi] (z4) -- ++(0.18,-0.18);\draw[Phi] (z4) -- ++(-0.18,-0.18);\draw[Z] (z4) -- ++(0,-0.25);
			\draw[Phi] (z5) -- ++(0.18,-0.18);\draw[Phi] (z5) -- ++(-0.18,-0.18);\draw[Z] (z5) -- ++(0,-0.25);
			
			\draw[dashed] (-1,-2.3) rectangle (3,-4.5);
			\node  at (3.1,-3) {$=g_{\bar T}$};
		\end{tikzpicture}
		\caption{An illustration for reducing graphs to trees, and inductive integrations, for the case $k=2$. In the right picture, there are three stochastic objects of the form $Z^{(i)}$, and one stochastic object of the form $\Wick{Z^{(i)}Z^{(j)}}$, and three stochastic objects of the form $\frac1{\sqrt N}\Wick{Z^{(i)}\PPhi^2}$}
		\label{Fig2}
	\end{figure}
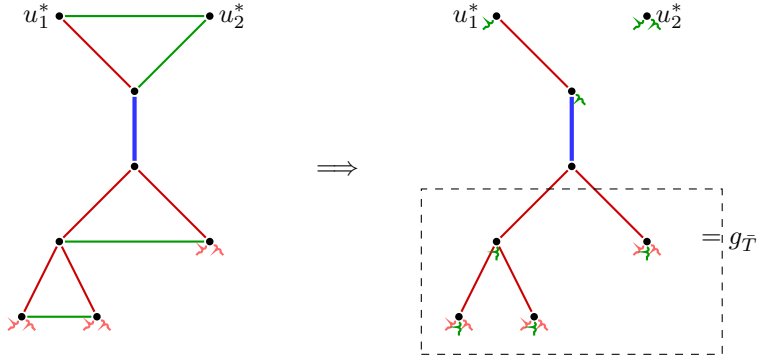
	
\medskip\medskip\medskip\medskip
	
	\textbf{Step 2. Estimate of each tree.}
	
	In this step we prove \eqref{est:FT}. 	Fixing a rooted tree $T$ as above, we will integrate the variables in $F_T$ in \eqref{e:F_T}
	from the leaves 
	of the tree $T$ and estimate the effect
	of the integrations as in \cite{SZZ21}. More precisely, we claim that
	for every subtree $\bar T$ of $T$ which contains the root $u^*$,
	\eqref{e:F_T} still holds with $T$ and $g^T$ on the RHS replaced by $\bar T$ and $g^{\bar T}$, 
	where the functions $g$  (which depend on $\bar T$) are such that
	$\|g\|_{\bC^{1-2\kappa}}$ is bounded by  $\Pi_{j\in I_g}\|f_j\|_{H^{-\kappa}}, f_j\in \sD$ for some index set $I_g\subset \mN$.  In this case, the $f_v$ in the vertex that belongs to $V_T\backslash V_{\bar T}$ has been incorporated into the definition of $g^{\bar T}$. The function $g$ is given as an example in Figure \ref{Fig2}, and $\bar T$ represents the remaining subtree obtained by subtracting the dashed box portion in the picture from the tree $T$ for $u_1^*$.
	
	We  claim that \eqref{e:F_T} still holds for the subtree $\bar T$ has $2$ vertices including $u^*$ and $v$ with $\|g^{\bar T}_v\|_{\bC^{1-2\kappa}}+\|g^{\bar T}_{u^*}\|_{\bC^{1-2\kappa}}$  bounded by  $\Pi_{j\in I_g}\|f_j\|_{H^{-\kappa}}, f_j\in \sD$ for some index set $I_g\subset \mN$. The expression for $F_T(x_{u^*})$ in this case can take one of the following forms:
	\begin{align*}
		f_{u^*},\quad f_{u^*}\cI(g^{\bar T}_vf_v),\quad g^{\bar T}_{u^*}\cI(g^{\bar T}_vf_v) ,\quad K(f_vg_v^{\bar T}).
	\end{align*}
	 Using Lemma \ref{lem:K}, Lemma \ref{lem:multi} and  \eqref{eq:sch}, we obtain \eqref{est:FT}.
	 
		The proof of the claim can be established by employing downward induction on the value of $|V_{\bar T}|$, which can be carried out similarly to the approach presented in \cite[Lemma 4.2]{SZZ21}. We will not provide the detailed proof here and only highlight the main modifications.

	The first change is that we only have the  $H^{-\kappa}$ norm of $f_v\in \sD$ by \eqref{sto:2}, while in \cite{SZZ21} we have $\bC^{-\kappa}$-norm estimates for the related terms in $\Phi^4_2$ model. Instead, for $\cI f$
	we apply the following from Lemma \ref{lem:emb}, Lemma \ref{lem:multi} and \eqref{eq:sch}
	\begin{equs}[e:I]
		\|\cI(g)\|_{\bC^{1-2\kappa}}
		&\lesssim \|g\|_{H^{-\kappa}},\quad
		\|\cI(gf)\|_{\bC^{1-2\kappa}}
		\lesssim \|g\|_{\bC^{1-2\kappa}}\|f\|_{H^{-\kappa}},
	\end{equs}
	for $0<\kappa<1/3$.

	 The second difference is that the graph involves the kernel $K$.  Since for each $K$ it only connects two points $u,y$ with $\{u,y\}\in E^1_{G^e}$ and the green line $C$ connecting the point $y$ gives $Z^{(i)}(y)$ by replacing $C(\cdot-y)$ as $\E(Z^{(i)}(\cdot)Z^{(i)}(y))$, the integration over $x_y$ is in one of the following forms:
	\begin{equation}\label{int:1}
		\aligned
		\int K(x_u-x_y)&(f_y(x_y)g^{\bar T}_y(x_y))\dif x_y,
		 \\ \int C(x_{\bar u}-x_u)&\Big(Z^{(k)}(x_u)\int K(x_u-x_y)(f_y(x_y)g^{\bar T}_y(x_y))\dif x_y\Big)\dif x_u,
		\\ \int C(x_{\bar u}-x_u)&\Big(\cI(g^{\bar T}_z)(x_u)\int K(x_u-x_y)(f_y(x_y)g^{\bar T}_y(x_y))\dif x_y\Big)\dif x_u,
		\endaligned
	\end{equation}
 where $f_y=\Wick{Z^{(i)}Z^{(j)}}$ or $Z^{(i)}$ or $1$ and $\|g^{\bar T}\|_{\bC^{1-2\kappa}}$ is controlled by the products of  $H^{-\kappa}$-norm of several  stochastic objects in $\sD$.

  
For instance,  the first case in \eqref{int:1} is exemplified by Figure \ref{Fig1} (in which case $u=u^*$).
The second case is exemplified by Figure \ref{Fig2}. The third case can arise when we replace the green wavy line in Figure \ref{Fig2} with one branch from the dashed box in Figure \ref{Fig2}.  In the second and third cases, we consider $g^{T'}_{\bar u}(x_{\bar u})$ defined by the second and third lines in \eqref{int:1}.   We aim to prove that  $\|g^{T'}_{\bar u}\|_{\bC^{1-2\kappa}}$ can be controlled by the products of the $H^{-\kappa}$-norms of several stochastic objects in $\sD$.

	For the first term in \eqref{int:1}, by Lemma \ref{lem:K},  for $\kappa>0$, $\|K(f_yg^{\bar T}_y)\|_{H^{-\kappa}}$ can be controlled by one of the following three terms
	\begin{align*}
		 \|\Wick{Z^{(i)}Z^{(j)}}\|_{H^{-\kappa}},\quad \|Z^{(i)}\|_{H^{-\kappa}}\|g^{\bar T}_y\|_{\bC^{1-2\kappa}},\quad \|g^{\bar T}_y\|_{\bC^{1-2\kappa}},
	\end{align*}
which is bounded by the products of $H^{-\kappa}$-norm of several $f_v\in \sD$. 

For the second term in \eqref{int:1}  we decompose the operator $K=I+L$ and we can write the second term as
\begin{align*}
	\cI \left(\Wick{f_yZ^{(k)}}g^{\bar T}_y \right)+\cI\left(Z^{(k)}L(f_yg^{\bar T}_y)\right),
\end{align*}
the $\bC^{1-2\kappa}$-norm of which by Lemma  \ref{lem:K} and Lemma \ref{lem:multi}, \eqref{eq:sch} is bounded by
\begin{align*}
 \Big(\|\Wick{f_yZ^{(k)}}\|_{H^{-\kappa}}+\|Z^{(k)}\|_{H^{-\kappa}}\|f_y\|_{H^{-\kappa}}\Big)\|g^{\bar T}_y\|_{\bC^{1-2\kappa}}.
\end{align*}

For the third term we  write it as
\begin{align*}
	\cI(\cI(g^{\bar T}_z)K(f_yg^{\bar T}_y)),
\end{align*}
the $\bC^{1-2\kappa}$-norm of which is bounded by
$$\|g^{\bar T}_z\|_{\bC^{1-2\kappa}}\|K(f_yg^{\bar T}_y)\|_{H^{-\kappa}}\lesssim \|g^{\bar T}_z\|_{\bC^{1-2\kappa}}\|f_y\|_{H^{-\kappa}}\|g^{\bar T}_y\|_{\bC^{1-2\kappa}}.$$
Here we use Lemma  \ref{lem:K} and Lemma \ref{lem:multi}, \eqref{eq:sch}. Hence, the $H^{-\kappa}$-norm of the first integral in \eqref{int:1} and the $\bC^{1-2\kappa}$-norm of the last two integrals in \eqref{int:1} can be bounded by products of  $H^{-\kappa}$-norm of several  stochastic objects in $\sD$.

The last difference is that we may meet $\E(\Wick{\PPhi^2})$ and we  apply \eqref{e:z1} to get some new graphs. In this case we can bound $\E(\Wick{\PPhi^2})$  by $1$. For the rest graph from  $\E(\Wick{\PPhi^2})$ in \eqref{e:z1}, we can also 
use \eqref{e:I} and Proposition \ref{prop1},  \eqref{uni:d1}, \eqref{uni:d2} to bound
\begin{align*}
	\|\mathbf{I}_{G'}\|_{\bC^{1-\kappa}}\lesssim1,
\end{align*}
with the graph $G'$ from \eqref{e:z1} and  the implicit constant independent of $N$.
 The rest is the same as in \cite[Lemma 4.2]{SZZ21} and \eqref{est:FT} follows.
\end{proof}

We finally give the proof of Lemma \ref{convergencek}.

\begin{proof}[Proof of Lemma \ref{convergencek}]
	The proof follows from the same argument as in \cite[Section 4.2]{SZZ21} and the discrete version of Proposition \ref{prop1}, i.e. \eqref{uni:d1} and \eqref{uni:d2}. The main difference is as in the proof of Proposition \ref{pr:N}, i.e.    we have a new kernel $K_\eps$ from $(I-C_\eps^2*\cdot)^{-1}$. From Lemma \ref{kcon} we have for $\kappa,\delta>0$
	\begin{align*}
		\|\cE^\eps( K_\eps*f_\eps)-K*\cE^\eps f_\eps\|_{H^{-\kappa-\delta}} &\lesssim \eps^{\delta/2}\|f_\eps\|_{H^{-\kappa,\eps}},
\\
		\|\cE^\eps( L_\eps*f_\eps)-L*\cE^\eps f_\eps\|_{H^{2-\kappa-2\delta}} &\lesssim \eps^{\delta/2}\|f_\eps\|_{H^{-\kappa,\eps}}.
	\end{align*}
	This combined with Lemma \ref{lem:Epro}, Lemma \ref{Cepro} implies the result.
\end{proof}

\section{Next order stationary dynamic}\label{sec:6}

 As in Section \ref{s:uni} we have a stationary process $(\Phi_i,Z_i)_{1\leq i\leq N}$ such that the components $\Phi_i$ and $Z_i$ are stationary solutions to \eqref{eq:Phi2d} and \eqref{eq:li1} respectively.
In stationarity, when $N\to\infty$, $\Phi_i$ converges to $Z_i$ as proved in \cite{SSZZ2d}.
In this section, we  consider the ``fluctuations'' of the stationary dynamics, namely the next order correction.
Subtracting \eqref{eq:Phi2d} and \eqref{eq:li1}, we obtain
$$
\sL (\Phi_i-Z_i) = \frac1N  \sum_{j=1}^N \Wick{\Phi_j^2 \Phi_i}.
$$
Multiplying $\sqrt N$ and setting $u_i^N=\sqrt N(\Phi_i-Z_i)=\sqrt NY_i$, we have
\begin{equation}\label{eq:u}
\sL u_i^N = \frac1{\sqrt N}  \sum_{j=1}^N \Wick{\Phi_j^2 \Phi_i}.
\end{equation}
Although the RHS ``appears'' to be of order $\frac1{\sqrt N}\times N=\sqrt{N}$, it is actually bounded as we show 
in the next theorem,
so that the above scaling by $\sqrt{N}$ is indeed the right scale to observe the next order fluctuation.

\bt\label{th:5.1} For every $k\in \mN$, $\{(u_1^N,\dots,u_k^N)\}_N$ is tight in $$\Big(L^2_{\text{loc}}(\mR^+;\bC^\kappa)\cap L^2_{\text{loc}}(\mR^+;H^\delta)\cap C(\mR^+;H^{-2\kappa})\Big)^k,$$
for some small $\kappa>0, 0<\delta<1$.
Every tight limit $u_i$ is a stationary process and satisfies the following equation
\begin{align}\label{eq:ui}
	\sL u_i=\cP_i,
\end{align}
in the analytic weak sense, where $\cP_i$ is the tight limit of $\frac1{\sqrt N}  \sum_{j=1}^N \Wick{\Phi_j^2 \Phi_i}$.
\et
\begin{proof}
By Lemma \ref{th:m2} and Lemma \ref{lem:zin}, for every $T\geq0$ and sufficiently small $\kappa>0$,
\begin{align*}
	&\sum_{i=1}^k\E\|u_i^N\|_{L^2([0,T];\bC^{2\kappa})}^2+\sum_{i=1}^k\E\|u_i^N\|_{L^2([0,T];H^1)}^2
	\\&=\sum_{i=1}^kN\E\|Y_i\|_{L^2([0,T];\bC^{2\kappa})}^2+\sum_{i=1}^kN\E\|Y_i\|_{L^2([0,T];H^1)}^2\lesssim1,
\end{align*}
with the proportional constant independent of $N$.
By Proposition \ref{prop1} the RHS of \eqref{eq:u} satisfies the following uniform bound for any $\kappa>0$
\begin{align*}
	\sum_{i=1}^k\E\Big\|\frac1{\sqrt N}  \sum_{j=1}^N \Wick{\Phi_j^2 \Phi_i}\Big\|^2_{H^{-\kappa}}
	\lesssim 1,
\end{align*}
which implies
\begin{align*}
	\sum_{i=1}^k\E\Big\|\frac1{\sqrt N}  \sum_{j=1}^N \Wick{\Phi_j^2 \Phi_i}\Big\|^2_{L^2_TH^{-\kappa}}
	\lesssim 1.
\end{align*}
By the above two estimates we obtain tightness of
$$\Big((u_1^N,\dots,u_k^N),\Big(\frac1{\sqrt N}  \sum_{j=1}^N \Wick{\Phi_j^2 \Phi_1},\dots,\frac1{\sqrt N}  \sum_{j=1}^N \Wick{\Phi_j^2 \Phi_k}\Big)\Big)$$ in
$$\Big(L^2([0,T];\bC^\kappa)\cap L^2([0,T];H^\delta)\cap C([0,T];H^{-2\kappa})\Big)^k\times L_{w}^2([0,T];H^{-2\kappa})^k,$$
for every $T\geq0$, where $L_{w}^2$ means weak topology w.r.t. time. Hence, we obtain the desired tightness.
Denoting the tightness limit by $(u_i,\cP_i)$,  we can then take limit on both sides of \eqref{eq:u} to obtain \eqref{eq:ui}.
\end{proof}

We will now give the time marginal law of $\cP_i$,  i.e. 
the law of $\cP_i(t)$ for fixed time $t\geq0$.

\subsection{Marginal law of $\cP_i$}
In this section we consider the marginal law of $\cP_i$, i.e.  large $N$ limit of $\frac1{\sqrt N}\Wick{\Phi_1\PPhi^2}$, for any fixed time.
Similarly as in \eqref{ob0}-\eqref{ob2} we first use the stationary solutions $Z_i$ and $Y_i$ to \eqref{eq:li1} and \eqref{eq:22} to precisely define $\frac1{\sqrt N}\Wick{\Phi_1\PPhi^2}$ as
\begin{align*}
	\frac1{\sqrt N}\Wick{\Phi_i\PPhi^2}\eqdef\frac1{\sqrt N}\sum_{j=1}^N\Big(Y_iY_j^2+2Y_iY_jZ_j+Y_i\Wick{Z_j^2}+Y_j^2Z_i+2Y_j\Wick{Z_iZ_j}+\Wick{Z_iZ_j^2}\Big).
\end{align*}
In the following we set $i=1$ and focus on $\frac1{\sqrt N}\Wick{\Phi_1\PPhi^2}$. Define
\begin{equ}
	g_{k,\eps}^N(y_1,\dots,y_k)\eqdef \frac1{N^{k/2}}	\E \Big(\prod_{i=1}^k\Wick{\Phi_{1,\eps}\PPhi^2_\eps}(y_i)\Big),\quad y_i\in \Lambda_\eps, i=1,\dots,k.
\end{equ}

By symmetry it is easy to find that for $k\in 2\mN-1$
\begin{align*}
	g_{k,\eps}^N(y_1,\dots,y_k)\equiv0, \quad y_i\in \Lambda_\eps, i=1,\dots,k.
\end{align*}

\bp\label{bp:1}
For $k\in2\mN$, $y_i\in \Lambda_\eps, i=1,\dots,k$, one has
\begin{equ}\label{eq:gkn}
	g_{k,\eps}^N(y_1,\dots,y_k)= C_{k,\eps}(y_1,\dots,y_k) f_{k,\eps}^N(y_1,\dots,y_k)+O_{N,\eps}.
\end{equ}
Here $C_{k,\eps}$ is the $k$-point correlation of the Gaussian free field with covariance $C_\eps$,
$f_{k,\eps}^N$ is the $k$-point correlation of $\frac1{\sqrt N}\Wick{\PPhi^2_\eps}$ defined in Section~\ref{sec:4.1},
 and $O_{N,\eps}$ is defined in \eqref{def:ON} below.
\ep

\begin{proof} The idea is similar as in the proof Lemma \ref{lem:IBP}. We again omit $\eps$ in the proof.
	We choose
	$$F_1(\Phi)=\Wick{\PPhi^2}(y_1)\prod_{i=2}^k\Wick{\Phi_1\PPhi^2}(y_i)$$
	as test function in \eqref{e:IBP-CN} to have
	\begin{align}\label{es:1k}
		&\sum_{m=2}^kC(y_1-y_m)
		\E\Big(\Wick{\PPhi^2}(y_1)\Big(\Wick{\PPhi^2}(y_m)+2\Wick{\Phi^2_1}(y_m)\Big)\prod_{j=2,j\neq m}^k\Wick{\Phi_1\PPhi^2}(y_j)\Big)
		\\
		&=\E \Big(\prod_{i=1}^k\Wick{\Phi_1\PPhi^2}(y_i)\Big)+\frac1N\E\Big(F_1(\Phi)\cI(\Wick{\Phi_1\PPhi^2})(y_1)\Big),\label{es:2k}
	\end{align}
where we used
$$\frac{\delta\Wick{\Phi_1\PPhi^2}(y_m)}{\delta \Phi_1(y_m)}=(\Wick{\PPhi^2}+2\Wick{\Phi^2_1})(y_m)\frac{e_{y_m}}{\eps^2}$$
	to derive the terms in \eqref{es:1k}.  Note that since $\Wick{\Phi_1\PPhi^2}$ is not $O(N)$ invariant, we cannot replace $\Wick{\Phi^2_1}$ by $\frac1N\Wick{\PPhi^2}$.
	Define
	$$O_N^1\eqdef \text{ the second part in }\eqref{es:1k}-\text{ the second term in }\eqref{es:2k},$$
	where the second part in \eqref{es:1k} means the part involving $2\Wick{\Phi_1^2}$. 
	In Lemmas \ref{lem:ON}+\ref{lem:On} below we will prove that $\frac1{N^{k/2}}O_N^1$ is of the order $\frac1{\sqrt N}$. 
The above IBP formula shows that one can find $y_1,y_m$ such that the relevant $\Wick{\Phi_1\PPhi^2}$ are turned into $\Wick{\PPhi^2}$ and a kernel $C(y_1-y_m)$ is created: graphically, up to the remainder $\frac1{N^{k/2}}O_N^1$,
\begin{equ}
\begin{tikzpicture}[baseline=-15]
	\node[dot] (x1) at (0,0) {}; \node at (-0.4,0) {$y_1$};
	\node[dot] (x2) at (1,0) {}; \node at (1.4,0) {$y_m$};
	\node[dot] (x3) at (0,-1) {};
	\node[dot] (x4) at (1,-1) {};
	\draw[Phi] (x1) -- ++(-0.18,-0.18);\draw[Phi] (x1) -- ++(0.18,-0.18);\draw[Phi] (x1) -- ++(0,-0.25);
	\draw[Phi] (x2) -- ++(-0.18,-0.18);\draw[Phi] (x2) -- ++(0.18,-0.18);\draw[Phi] (x2) -- ++(0,-0.25);
	\draw[Phi] (x3) -- ++(-0.18,-0.18);\draw[Phi] (x3) -- ++(0.18,-0.18);\draw[Phi] (x3) -- ++(0,-0.25);
	\draw[Phi] (x4) -- ++(-0.18,-0.18);\draw[Phi] (x4) -- ++(0.18,-0.18);\draw[Phi] (x4) -- ++(0,-0.25);
\end{tikzpicture}
\qquad
\Rightarrow
\qquad
\begin{tikzpicture}[baseline=-15]
	\node[dot] (x1) at (0,0) {}; \node at (-0.4,0) {$y_1$};
	\node[dot] (x2) at (1,0) {}; \node at (1.4,0) {$y_m$};
	\node[dot] (x3) at (0,-1) {};
	\node[dot] (x4) at (1,-1) {};
	\draw[C] (x1) to (x2);
	\draw[Phi] (x1) -- ++(-0.18,-0.18);\draw[Phi] (x1) -- ++(0,-0.25);
	\draw[Phi] (x2) -- ++(0.18,-0.18);\draw[Phi] (x2) -- ++(0,-0.25);
	\draw[Phi] (x3) -- ++(-0.18,-0.18);\draw[Phi] (x3) -- ++(0.18,-0.18);\draw[Phi] (x3) -- ++(0,-0.25);
	\draw[Phi] (x4) -- ++(-0.18,-0.18);\draw[Phi] (x4) -- ++(0.18,-0.18);\draw[Phi] (x4) -- ++(0,-0.25);
\end{tikzpicture}
\end{equ}
	
	For the first term in \eqref{es:1k} involving only $\Wick{\PPhi^2}(y_m)$ we  choose the following  test function for a fixed $i\in \{2,\cdots,k\}\backslash\{m\}$
	$$
	F_2(\Phi)=\Wick{\PPhi^2}(y_1)
	\Wick{\PPhi^2}(y_m)\Wick{\PPhi^2}(y_i)
	\prod_{j=2, j\notin\{i,m\}}^k  \!\!\!\!\!\! \Wick{\Phi_1\PPhi^2}(y_j)
	$$
	 in \eqref{e:IBP-CN} to have 
	\begin{equs}[eq:cz1]
		{}&2C(y_i-y_m)\E\Big(\Wick{\PPhi^2}(y_1)\Phi_1(y_m)\Wick{\PPhi^2}(y_i)
		\prod_{j=2, j\notin \{i,m\}}^k \!\!\!\!\!\! \Wick{\Phi_1\PPhi^2}(y_j)\Big)
		\\
		&+2C(y_i-y_1)\E\Big(\Phi_1(y_1)\Wick{\PPhi^2}(y_m)\Wick{\PPhi^2}(y_i)
		\prod_{j=2,j\notin \{i,m\}}^k \!\!\!\!\!\! \Wick{\Phi_1\PPhi^2}(y_j)\Big)
		\\
		&+\sum_{j=2,j\notin \{i,m\}}^k  \!\!\!\!\! C(y_i-y_j)
		\E\Big(\prod_{\ell\in \{1,m,i\}} \!\!\!\! \Wick{\PPhi^2}(y_\ell)\Big(\Wick{\PPhi^2}(y_j)+2\Wick{\Phi^2_1}(y_j)\Big)\prod_{p\notin\{1,m,i,j\}}\!\!\!\!\Wick{\Phi_1\PPhi^2}(y_p)\Big)
		\\&=
		\E\Big(\Wick{\PPhi^2}(y_1)\Wick{\PPhi^2}(y_m)\prod_{j=2,j\neq m}^k \!\!\!\! \Wick{\Phi_1\PPhi^2}(y_j)\Big)+\frac1N\E\Big(F_2(\Phi)\cI(\Wick{\Phi_1\PPhi^2})(y_i)\Big).
	\end{equs}
	Define
	\begin{align*}
		O_{N}^2\eqdef &\sum_{m=2}^kC(y_1-y_m)\Big(\text{first two terms in \eqref{eq:cz1} }
		\\&+\text{the second part in the third term \eqref{eq:cz1} }-\text{last term in \eqref{eq:cz1}}\Big),
	\end{align*}
where the second part in the third term \eqref{eq:cz1} refers to the part involving $2\Wick{\Phi_1^2}$. 
	Again we will prove below that $\frac1{N^{k/2}}O_{N}^2$ is of order $\frac1{\sqrt N}$.
So we have found two more points $y_i,y_j$ such that the corresponding 
$\Wick{\Phi_1\PPhi^2}$ are turned into $\Wick{\PPhi^2}$ and $C(y_i-y_j)$ is created: namely, up to the remainder $\frac1{N^{k/2}} O_{N}^2$,
\begin{equ}
\begin{tikzpicture}[baseline=-15]
	\node[dot] (x1) at (0,0) {}; \node at (-0.4,0) {$y_1$};
	\node[dot] (x2) at (1,0) {}; \node at (1.4,0) {$y_m$};
	\node[dot] (x3) at (0,-1) {}; \node at (-0.4,-1) {$y_i$};
	\node[dot] (x4) at (1,-1) {}; \node at (1.4,-1) {$y_j$};
	\draw[C] (x1) to (x2);
	\draw[Phi] (x1) -- ++(-0.18,-0.18);\draw[Phi] (x1) -- ++(0,-0.25);
	\draw[Phi] (x2) -- ++(0.18,-0.18);\draw[Phi] (x2) -- ++(0,-0.25);
	\draw[Phi] (x3) -- ++(-0.18,-0.18);\draw[Phi] (x3) -- ++(0.18,-0.18);\draw[Phi] (x3) -- ++(0,-0.25);
	\draw[Phi] (x4) -- ++(-0.18,-0.18);\draw[Phi] (x4) -- ++(0.18,-0.18);\draw[Phi] (x4) -- ++(0,-0.25);
\end{tikzpicture}
\qquad
\Rightarrow
\qquad
\begin{tikzpicture}[baseline=-15]
	\node[dot] (x1) at (0,0) {}; \node at (-0.4,0) {$y_1$};
	\node[dot] (x2) at (1,0) {}; \node at (1.4,0) {$y_m$};
	\node[dot] (x3) at (0,-1) {}; \node at (-0.4,-1) {$y_i$};
	\node[dot] (x4) at (1,-1) {}; \node at (1.4,-1) {$y_j$};
	\draw[C] (x1) to (x2);\draw[C] (x3) to (x4);
	\draw[Phi] (x1) -- ++(-0.18,-0.18);\draw[Phi] (x1) -- ++(0,-0.25);
	\draw[Phi] (x2) -- ++(0.18,-0.18);\draw[Phi] (x2) -- ++(0,-0.25);
	\draw[Phi] (x3) -- ++(-0.18,-0.18);\draw[Phi] (x3) -- ++(0,-0.25);
	\draw[Phi] (x4) -- ++(0.18,-0.18);\draw[Phi] (x4) -- ++(0,-0.25);
\end{tikzpicture}
\end{equ}
	Combining the above calculation we find 
	\begin{align*}
		\E \Big(\prod_{n=1}^k\Wick{\Phi_1\PPhi^2}(y_n)\Big)
		=&\sum_{\substack{m,j=2 \\ i,j,m \mbox{ \tiny distinct}}}^k  C(y_1-y_m)C(y_i-y_j)
		\\
		&\E\Big(\prod_{\ell\in \{1,m,i,j\}}\!\!\!\!\! \Wick{\PPhi^2}(y_\ell)\prod_{p\notin\{1,m,i,j\}} \!\!\!\!\! \Wick{\Phi_1\PPhi^2}(y_p)\Big)
		+\sum_{q=1}^2O_{N}^q.
	\end{align*}
	
In general we  have the following term for $2\leq \ell\leq k-2$, $\ell\in 2\mN$,
	\begin{align*}
		I\eqdef	\E\Big(\prod_{m=1}^\ell\Wick{\PPhi^2}(y_{\pi(m)})\prod_{j=\ell+1}^k\Wick{\Phi_1\PPhi^2}(y_{\pi(j)})\Big),
	\end{align*}
	with $\pi$ being a permutation of $\{1,\dots,k\}$. We then choose
	\begin{align*}
		F_3(\Phi)=\prod_{m=1}^{\ell+1}\Wick{\PPhi^2}(y_{\pi(m)})\prod_{j=\ell+2}^k\Wick{\Phi_1\PPhi^2}(y_{\pi(j)})
	\end{align*}
	as the test function $F$ and $y_{\pi(\ell+1)}$ as $x$ in \eqref{e:IBP-CN}  to obtain
	\begin{equation}\label{eq:cz2}
		\aligned
		&2\sum_{p=1}^\ell C(y_{\pi(\ell+1)}-y_{\pi(p)})\E\Big(\Phi_1(y_{\pi(p)})\prod_{m=1,m\neq p}^{\ell+1}\Wick{\PPhi^2}(y_{\pi(m)})\prod_{j=\ell+2}^k\Wick{\Phi_1\PPhi^2}(y_{\pi(j)})\Big)
		\\&+\sum_{p=\ell+2}^k C(y_{\pi(\ell+1)}-y_{\pi(p)})\E\Big(\prod_{m=1}^{\ell+1}\Wick{\PPhi^2}(y_{\pi(m)})
		\Big(\Wick{\PPhi^2}(y_{\pi(p)})+2\Wick{\Phi^2_1}(y_{\pi(p)})\Big)
		\\&
		\qquad\qquad\qquad\qquad\qquad\qquad
		\times\prod_{j=\ell+2,j\neq p}^k\Wick{\Phi_1\PPhi^2}(y_{\pi(j)})\Big)\\&=
		I+\frac1N\E\Big(F_3(\Phi)\cI(\Wick{\Phi_1\PPhi^2})(y_{\pi(\ell+1)})\Big).
		\endaligned
	\end{equation}
The first term, the part involving $2\Wick{\Phi^2_1}(y_{\pi(p)})$ in the second term and the last term in \eqref{eq:cz2} together is defined as $O_{N}^{\frac\ell2+1}$. The first part in the second term can be viewed as choosing two points $y_{\pi(\ell+1)}$ and $y_{\pi(p)}$, turning the corresponding $\Wick{\Phi_1\PPhi^2}$ into $\Wick{\PPhi^2}$ and adding $C(y_{\pi(\ell+1)}-y_{\pi(p)})$.
	We iterate the above procedure and obtain
		\begin{equ}
		\frac1{N^{k/2}}\E \Big(\prod_{i=1}^k\Wick{\Phi_1\PPhi^2}(y_i)\Big)
		=\frac1{N^{k/2}}\Big(\sum_\pi\prod_{j=1}^{k/2}C(y_{\pi(2j-1)}-y_{\pi(2j)})\Big)
		\E\Big(\prod_{i=1}^k\Wick{\PPhi^2}(y_i)\Big)+O_{N},
	\end{equ}
	where $\pi$ runs over all the pairing permutations of $\{1,\dots,k\}$ and
	\begin{align}\label{def:ON}
		O_{N}\eqdef\sum_{i=1}^{k/2}\frac1{N^{k/2}}O_{N}^i.
	\end{align}
	Noting that 
	$$
	C_k(y_1,\cdots,y_k) \eqdef \sum_\pi\prod_{j=1}^{k/2}C(y_{\pi(2j-1)}-y_{\pi(2j)})
	$$ 
	gives the $k$-point function of Gaussian free field with covariance $C$, the result follows.
\end{proof}

For $k\in\mN$ we  define the $k$-point function for  $\frac1{\sqrt{N}}\Wick{\Phi_1\PPhi^2}$:
\begin{align*}
	\<g_k^N,\varphi\>=\lim_{\eps\to0}\int \E\Big(\prod_{i=1}^k\cE^\eps\frac{1}{\sqrt N}\Wick{\Phi_{1,\eps}\PPhi^2_\eps}(y_i)\Big)\varphi(y_1,\dots,y_k)\prod_{i=1}^k\dif y_i.
\end{align*}
for every $\varphi\in \cS(\mT^{2k})$. Since for fixed $N$ the law of $\frac{1}{\sqrt N}\Wick{\Phi_1\PPhi^2}$ is uniquely determined, we have
\begin{align*}
	\<g_k^N,\otimes_{i=1}^k\varphi_i\>= \E\prod_{i=1}^k\Big\<\frac{1}{\sqrt N}\Wick{\Phi_1\PPhi^2},\varphi_i\Big\>,
\end{align*}
for $\varphi_i\in \cS(\mT^2)$, $i=1,\dots, k$.
Now we send $\eps\to0$ on the both sides of \eqref{eq:gkn}, and write $g_{k}^N,C_{k},f_{k}^N,O_{N}$ for the limits
of the terms appearing therein.
 The same argument as in the proof Lemma \ref{le:fkN} implies the following result which provides more explicit expressions for $O_N$.

\bl\label{lem:ON} 
It holds that $g_{2m-1}^N\equiv0, m\in \mN$. For $k\in 2\mN$, $\varphi_i\in \cS(\mT^2)$, one has 
\begin{align}\label{e:gkn}
	\<g_{k}^N,\otimes_{i=1}^k \varphi_i\>= \<C_{k}f_{k}^N,\otimes_{i=1}^k \varphi_i\>+\<O_{N},\otimes_{i=1}^k \varphi_i\>, \end{align}
where $O_N=\sum_{i=1, i \in 2\mN}^{k}\frac1{N^{k/2}}O_N^{i/2}$.  Here, $O_N^{1}=O_N^{1,1}+O_N^{1,2}$ with
\begin{align*}
	\<O_N^{1,1},\otimes_{i=1}^k \varphi_i\>
	&\eqdef 2\E\Big(\<\Wick{\PPhi^2}\cI(\Wick{\Phi_1^2}\varphi_m),\varphi_1\>\prod_{j=2,j\neq m}^k\<\Wick{\Phi_1\PPhi^2},\varphi_{j}\>\Big),
\\
	\<O_N^{1,2},\otimes_{i=1}^k \varphi_i\>
	&\eqdef-\frac1N \E\Big(\<\Wick{\PPhi^2}\cI(\Wick{\Phi_1\PPhi^2}),\varphi_1\>\prod_{j=2}^k\<\Wick{\Phi_1\PPhi^2},\varphi_{j}\>\Big).
\end{align*}
Moreover, for $2\leq \ell\leq k-2$, $\ell\in 2\mN$,  $O_N^{\frac\ell2+1}=\sum_{i=1}^3O_N^{\frac\ell2+1,i}$ with
\begin{align*}
	\<O_N^{\frac\ell2+1,1},\otimes_{i=1}^k \varphi_i\>
	&\eqdef
	2\sum_{\pi} \E\Big(\<\cI(\Wick{\PPhi^2}\varphi_{\pi(2)})\cI(\Wick{\PPhi^2}\varphi_{\pi(3)})\Phi_1,\varphi_{\pi(1)}\>
	\\&\qquad \times
	\prod_{i=2}^{\ell/2}\<\cI(\Wick{\PPhi^2}\varphi_{\pi(2i)})\Wick{\PPhi^2},\varphi_{\pi(2i+1)}\>\prod_{j=\ell+2}^k\<\Wick{\Phi_1\PPhi^2},\varphi_{\pi(j)}\>\Big),
\end{align*}
\begin{align*}
	\<O_N^{\frac\ell2+1,2},\otimes_{i=1}^k \varphi_i\>
	&\eqdef
	2\sum_{\pi}\sum_{p=\ell+2}^k \E\Big(
	\prod_{i=1}^{\ell/2}\<\cI(\Wick{\PPhi^2}\varphi_{\pi(2i)})\Wick{\PPhi^2},\varphi_{\pi(2i-1)}\>
	\\
	&\qquad\times\<\cI(\Wick{\Phi_1^2}\varphi_{\pi(p)})\Wick{\PPhi^2},\varphi_{\pi(\ell+1)}\>
	\prod_{j=\ell+2,j\neq p}^k\<\Wick{\Phi_1\PPhi^2},\varphi_{\pi(j)}\>\Big),
\end{align*}
and
\begin{align*}
	\<O_N^{\frac\ell2+1,3},\otimes_{i=1}^k \varphi_i\> 
	& \eqdef-\frac1N\sum_{\pi} \E\Big(\<\cI(\Wick{\Phi_1\PPhi^2})\Wick{\PPhi^2},\varphi_{\pi(1)})\>
	\\
	&\qquad\times\prod_{i=1}^{\ell/2}\<\cI(\Wick{\PPhi^2}\varphi_{\pi(2i)})\Wick{\PPhi^2},\varphi_{\pi(2i+1)}\>\prod_{j=\ell+2}^k\<\Wick{\Phi_1\PPhi^2},\varphi_{\pi(j)}\>\Big).
\end{align*}
Here each $\pi$ is a  permutation of $\{1,\dots,k\}$ modulo possibly swapping the values of $\pi(j)$ in each $p_i, i=0, \dots,\ell/2+1$ with $p_0=\{\pi(1)\}$,    $p_i=\{\pi(2i),\pi(2i+1)\}$ for $i=1,\dots,\ell/2$ and $p_{\ell/2+1}=\{\pi(\ell+2),\dots,\pi(k)\}$. The sums above run  over all such permutations. 
\el

Now we show that $O_N$ is of order $\frac1{\sqrt N}$. Recall our notation $B_i$ in \eqref{no:B}.

\bl\label{lem:On} 
Let $\m$ be as in Lemma \ref{th:m2}.  For $\varphi_i\in \cS(\mT^2)$ it holds that
$$|\<O_N,\otimes_{i=1}^k\varphi_i\>|\lesssim \frac1{\sqrt N}.$$
\el

\begin{proof}
	Using Lemma \ref{lem:zin} we have for $\ell\geq1$ 
\begin{align*}
	\E\|\Wick{\Phi_1^2}\|_{H^{-\kappa}}^\ell\lesssim 1.
\end{align*}
	Hence, by Proposition \ref{prop1},
\begin{align*}
	\frac1{N^{k/2}}|\<O_N^{1,1},\otimes_{i=1}^k \varphi_i\>|\lesssim	\frac1{N^{k/2}}\E \Big(\|\Wick{\Phi_1^2}\|_{H^{-\kappa}}B_2B_3^{k-2}\Big)\lesssim \frac1{N^{1/2}}.
\end{align*}
	For $O_N^{1,2}$, Lemma \ref{lem:multi} and  \eqref{eq:sch} imply that
	\begin{align*}
		|\<\Wick{\PPhi^2}\cI(\Wick{\Phi_1\PPhi^2}),\varphi_1\>|\lesssim B_2B_3.
	\end{align*}
	Hence, by Proposition \ref{prop1},
	\begin{align*}
		\frac1{N^{k/2}}|\<O_N^{1,2},\otimes_{i=1}^k \varphi_i\>|\lesssim	\frac1{N^{k/2+1}}\E (B_2B_3^k)\lesssim \frac1{N^{1/2}}.
	\end{align*}
	Similarly we use Lemma \ref{lem:multi} and \eqref{es} to have
	\begin{align*}
		|\<\cI(\Wick{\PPhi^2}\varphi_{\pi(2)})\cI(\Wick{\PPhi^2}\varphi_{\pi(3)})\Phi_1,\varphi_{\pi(1)}\>|\lesssim B_1B_2^2,
	\end{align*}
	and
	\begin{equ}
		|\<\cI(\Wick{\PPhi^2}\varphi_{\pi(2i)})\Wick{\PPhi^2},\varphi_{\pi(2i+1)}\>|\lesssim B_2^2, \qquad
		|\<\cI(\Wick{\Phi_1\PPhi^2})\Wick{\PPhi^2},\varphi_{\pi(1)})\>|\lesssim B_2B_3,
	\end{equ}
	which combined with Proposition \ref{prop1} implies that for $\ell\geq2$
	\begin{equs}
		\frac1{N^{k/2}} & |\<O_N^{\frac\ell2+1},\otimes_{i=1}^k \varphi_i\>|
		\lesssim \frac1{N^{k/2}}\E( B_1B_2^\ell B_3^{k-\ell-1})
		\\
		&+\frac1{N^{k/2}}\E\Big(\|\Wick{\Phi_1^2}\|_{H^{-\kappa}} B_2^{\ell+1} B_3^{k-\ell-2}\Big)
		+	\frac1{N^{k/2+1}}\E( B_2^{\ell+1}B_3^{k-\ell})
		\lesssim \frac1{N^{1/2}}.
	\end{equs}
\end{proof}

\bt\label{th:3} 
Let $\m$ as in Lemma \ref{th:m2}. 
The marginal law of $\cP_i, i=1,\dots,k,$ is given by  the random field $X_1\cQ$, where $X_1$ and $\cQ$ are independent, 
$X_1=^dZ$, and $\cQ$ is the Gaussian field obtained in Theorem \ref{th:1}. Here, 
$$
X_1\cQ\eqdef \lim_{\eps\to0} \Big( (\rho_\eps*X_1) \cdot (\rho_\eps*\cQ)\Big)
$$  
where $\rho_\eps$ is a smooth mollifier
 as in Proposition \ref{pro:1}, and the limit is in $\bC^{-\kappa}$ $\bP$-a.s..
\et
\begin{proof}
Proposition \ref{prop1} implies tightness of $\{\frac1{\sqrt N}\Wick{\Phi_i\PPhi^2}\}_N$ in $H^{-\kappa}$, $\kappa>0$. Now the result follows by taking limit on both sides of \eqref{e:gkn} in Lemma~\ref{lem:ON} and applying Lemma \ref{lem:On} and the fact that $C_kf_k$ gives the $k$-point function of $X_1\cQ$.
\end{proof}

In the following we consider marginal law of $(\cP_1,\dots,\cP_m)$ for $m\in\mN$.
For $\mathbf{k}=(k_1,\dots,k_m)\subset \mN^m$, write
\begin{equ}
	g_{\mathbf{k},\eps}^N(y_1^1,\dots,y_{k_m}^m)\eqdef \frac1{N^{\sum_{j=1}^mk_j/2}}	\E \Big(\prod_{j=1}^m\prod_{i=1}^{k_j}\Wick{\Phi_{j,\eps}\PPhi^2_\eps}(y_i^j)\Big),
\end{equ}
$y_i^j\in \Lambda_\eps, i=1,\dots,k_j, j=1,\dots,m$.

By symmetry it is easy to find that for  $\mathbf{k}\subset \mN^m$ with $k_j\in 2\mN-1$ for some $j$
\begin{align*}
	g_{\mathbf{k},\eps}^N\equiv0.
\end{align*}

Similar as in the proof of Proposition \ref{bp:1} we derive the following results.

\bp Let $\m$ be as in Lemma \ref{th:m2}. 
It holds that for $\mathbf{k}\subset \mN^m$ with $k_\ell\in2\mN$ for each $\ell=1,\dots,m$
\begin{align}\label{eq:gknn}
	g_{\mathbf{k},\eps}^N(y_1^1,\dots,y_{k_m}^m)= \Big(\prod_{j=1}^mC_{k_j,\eps}(y_1^j,\dots,y_{k_j}^j)\Big)f_{\sum_{j=1}^mk_j,\eps}^N(y_1^1,\dots,y_{k_m}^m)+O_{N,\eps}^{\mathbf{k}},
\end{align}
$y_i^j\in \Lambda_\eps, i=1,\dots,k_j, j=1,\dots,m$, where  $O_{N,\eps}^{\mathbf{k}}$ is of order $\frac1{\sqrt N}$.
\ep
\begin{proof}
	The proof is similar as that of Proposition  \ref{bp:1} so we only point out the difference.  Omit $\eps$ in the proof.
	We choose
	$$F_1(\Phi)=\Wick{\PPhi^2}(y_1^1)\Big(\prod_{i=2}^{k_1}\Wick{\Phi_1\PPhi^2}(y_i^1)\Big)\prod_{j=1}^m\prod_{i=1}^{k_j}\Wick{\Phi_{j}\PPhi^2}(y_i^j)$$
	as test function in \eqref{e:IBP-CN} to have similar terms as in \eqref{es:1k} and \eqref{es:2k}. The main difference is that we need to consider  $\frac{\delta\Wick{\Phi_{j}\PPhi^2}}{\delta \Phi_1}$, which gives $2\Wick{\Phi_1\Phi_j}$ for $j\neq1$. Since $$\Wick{\Phi_1\Phi_j}=Y_1Y_j+Y_1Z_j+Z_1Y_j+\Wick{Z_1Z_j},$$
	by Lemma \ref{lem:zin} we have for $\ell\geq1$
	\begin{align*}
		\E\|\Wick{\Phi_1\Phi_j}\|_{H^{-\kappa}}^\ell\lesssim1,
	\end{align*}
	 which combined with Proposition \ref{prop1} and similar calculation as in the proof of Lemma \ref{lem:On} implies that the part including $2\Wick{\Phi_1\Phi_j}$ for $j\neq 1$ corresponds to the term of order $\frac1{\sqrt N}$. Hence, after $k_1/2$ steps, we get
	\begin{align*}
		&g_{\mathbf{k},\eps}^N(y_1^1,\dots,y_{k_m}^m)
		\\&= C_{k_1,\eps}(y_1^1,\dots,y_{k_1}^1)\frac1{N^{\sum_{j=1}^mk_j/2}}	\E \Big(\prod_{\ell=1}^{k_1}\Wick{\PPhi^2}(y_\ell^1)\prod_{j=2}^m\prod_{i=1}^{k_j}\Wick{\Phi_{j}\PPhi^2}(y_i^j)\Big)+O(\frac1{\sqrt N}).
	\end{align*}
	Compared to the definition of $g_{\mathbf{k},\eps}^N$, we change all the $\Wick{\Phi_1\PPhi^2}$ to $\Wick{\PPhi^2}$. We then repeat the above procedures for each $\Wick{\Phi_j\PPhi^2}$ and the result follows.
\end{proof}

Hence, as in the proof of Theorem \ref{th:3} we derive the following result.

\bt\label{th:4} Let $\mm$ as in Lemma \ref{th:m2}. For $m\in\mN$, the marginal law of $(\cP_1,\dots,\cP_m)$ are given by  the random field $$(X_1\cQ,X_2\cQ,\dots,X_m\cQ),$$
 where $X_1, \dots, X_m, \cQ$ are independent and $X_i=^dZ, i=1,\dots,k,$ and  $\cQ$ is the Gaussian field obtained in Theorem \ref{th:1}.
\et

 \appendix
\renewcommand{\appendixname}{Appendix~\Alph{section}}
\renewcommand{\theequation}{A.\arabic{equation}}

\section{Besov spaces}\label{A.1}
In this appendix we summarize some basic results regarding Besov spaces on the continuum spaces as well as on lattices.  

\subsection{Besov spaces on $\mT^d$}

Let $(\Delta_{i})_{i\geq -1}$ be the Littlewood--Paley blocks for a dyadic partition of unity.
Besov spaces on $\mT^d$ with indices $\alpha\in \R$, $p,q\in[1,\infty]$ are defined as
the completion of $C^\infty(\mT^d)$ with respect to the norm
$$
\|u\|_{B^\alpha_{p,q}}:=
\Big(\sum_{j\geq-1}(2^{jq\alpha}\|\Delta_ju\|_{L^p}^q)\Big)^{1/q},
$$
and the H\"{o}lder-Besov space $\bC^\alpha$ is given by $\bC^\alpha=B^\alpha_{\infty,\infty}$.  We write $\|\cdot\|_{\bC^\alpha}\eqdef\|\cdot\|_{B^\alpha_{\infty,\infty}}$.

The following embedding results are frequently used. 
Recall the spaces $H_p^s$ defined in Section~\ref{sec:intro}.

\bl\label{lem:emb} (i) Let $1\leq p_1\leq p_2\leq\infty$ and $1\leq q_1\leq q_2\leq\infty$, and let $\alpha\in\mathbb{R}$. Then $B^\alpha_{p_1,q_1} \subset B^{\alpha-d(1/p_1-1/p_2)}_{p_2,q_2}$. (cf. \cite[Lemma~A.2]{GIP15})

(ii) Let $s\in \R$, $1<p<\infty$, $\epsilon>0$. Then $H^s_2=B^s_{2,2}$, and
$B^s_{p,1}\subset H^{s}_p\subset B^{s}_{p,\infty}\subset B^{s-\epsilon}_{p,1}$. (cf. \cite[Theorem 4.6.1]{Tri78})



Here  $\subset$ means  continuous and dense embedding.
\el

%
%
%
%


\bl\label{lem:multi}
 Let $\alpha,\beta\in\mathbb{R}$ and $p, p_1, p_2, q\in [1,\infty]$ be such that $\frac{1}{p}=\frac{1}{p_1}+\frac{1}{p_2}$.
The bilinear map $(u, v)\mapsto uv$
extends to a continuous map from ${B}^\alpha_{p_1,q}\times {B}^\beta_{p_2,q}$ to ${B}^{\alpha\wedge\beta}_{p,q}$  if $\alpha+\beta>0$. (cf. \cite[Corollary~2]{MW17})

\el

We recall the following  smoothing effect of the heat flow $S_t=e^{t(\Delta-\mm)}$, $\mm\geq0$ (e.g. \cite[Lemma~A.7]{GIP15}, \cite[Proposition~5]{MW17}).
\vskip.10in
\bl\label{lem:heat}  Let $u\in B^{\alpha}_{p,q}$ for some $\alpha\in \mathbb{R}, p,q\in [1,\infty]$. Then for every $\delta\geq0$ and $t\in [0,T]$
$$\|S_tu\|_{B^{\alpha+\delta}_{p,q}}\lesssim t^{-\delta/2}\|u\|_{B^{\alpha}_{p,q}}.$$
\el

\bl\label{lem:sch} Let $I(f)(t)\eqdef\int_0^tS_{t-r}f\dif r$. Then for any $p\geq1,\alpha\in\mR,\beta\in(0,2)$
$$\|I(f)\|_{L^p_T\bC^{\alpha+\beta}}\lesssim \|f\|_{L^p_T\bC^{\alpha}}\;.$$
\el
\begin{proof} We use Lemma \ref{lem:heat} and H\"older's inequality to have
	\begin{align*}
	\|I(f)\|^p_{L^p_T\bC^{\alpha+\beta}}
	&\lesssim
	\int_0^T\Big(\int_0^t (t-s)^{-\beta/2}\|f\|_{\bC^\alpha}\dif s \Big)^p\dif t
	\lesssim_T
	\int_0^T\int_0^t (t-s)^{-\beta/2}\|f\|^p_{\bC^\alpha}\dif s \dif t
	\\
	&=\int_0^T\int_s^T (t-s)^{-\beta/2}\dif t\|f\|^p_{\bC^\alpha} \dif s
	\lesssim \|f\|^p_{L^p_T\bC^{\alpha}}
	\end{align*}
	where we used Fubini's Theorem in the third step.
\end{proof}

The following result can be proved similarly as in \cite[Lemma A.2]{SZZ21}.

\bl\label{lem:el}
 For $\alpha\in \mR, p,q\in [1,\infty]$, $d=2$ and the operators $\cI,\cI_1$ introduced in \eqref{def:I} one has
\begin{equs}
\|\cI f\|_{H^{\alpha+2}} &\lesssim \|f\|_{H^{\alpha}},		\label{eq:sch}
\\
\| \cI_1 f\|_{H^{\alpha+2-\gamma}} &\lesssim \|f\|_{H^{\alpha}}, \quad \forall\gamma>0.	\label{eq:sch2}
\end{equs}
\el
\begin{proof}
	 The first result follows from the Fourier transform  $\widehat{C}(k)=\frac1{\m+4\pi^2|k|^2}, k\in \mZ^2$ and the definition of Sobolev space. The second result follows from
$\widehat{C^2}(k)=\widehat{C}*\widehat{C}(k),k\in\mZ^2.$
\end{proof}

The following result is the regularity estimate for the opearator $K$ introduced in \eqref{def:K}.

\bl\label{lem:K} For $f\in C^\infty(\mT^2)$, $d=2$
\begin{align*}
	Kf(x)=	\int_{\mathbb{R}^2} K(x-y)f(y)\dif y, \quad K(x-y)=\delta(x-y)+L(x-y),\quad x,y\in\mT^2,
\end{align*}
with $L\in L^p(\mathbb{R}^2), p\geq 2$, where we view $f$ as periodic function on $\mathbb{R}^2$.
Furthermore, it holds that for $\gamma\in\mR$, $\delta>0$
\begin{align*}
	\|Kf\|_{H^\gamma}\lesssim \|f\|_{H^\gamma},\quad	\|Lf\|_{H^{2+\gamma-\delta}}\lesssim \|f\|_{H^{\gamma}},
\end{align*}
where $Lf(x)=\int L(x-y)f(y)\dif y$.
\el
\begin{proof}
	The proof follows from Fourier transform (see also \cite[Lemma 2]{MR578040}). In the proof we will view the kernel $K$ as a function on $\mR^2$ and $f$ periodic function on $\mR^2$. We have for $k\in\mR^2$
	$$0\leq\cF_{\mR^2} K(k)=\frac1{1+\cF_{\mR^2} {C^2}(k)}\leq 1,$$
	where $\cF_{\mR^2}$ denotes the Fourier transform on $\mR^2$. 
	Thus $\|Kf\|_{H^\gamma}\lesssim \|f\|_{H^\gamma}$. It is easy to see that for $k\in\mR^2$
	$$\cF_{\mR^2}{L}(k)=-\frac{\cF_{\mR^2} {C^2}(k)}{1+\cF_{\mR^2} {C^2}(k)},$$
	which implies for any $\delta>0$
	$$|\cF_{\mR^2}{L}(k)|\lesssim |k|^{-2+\delta}\wedge1.$$
	Then we obtain $L\in L^p$ for any $p\geq2$ by Hausdorff--Young's inequality and the second estimate follows.
\end{proof}

\subsection{Besov spaces on discrete torus}
We also introduce  Besov spaces on the lattice $\Lambda_{\eps}=\eps \mZ^d\cap \mT^d$ where $\eps=2^{-N}, N\in \mN\cup\{0\}.$ We can view functions on the lattice  $\Lambda_{\eps}$ as periodic functions on $\eps \mathbb{Z}^d$. The definition of Besov spaces on the lattice $\eps \mathbb{Z}^d$ is from \cite{MP19,GH18a}. For $f\in L^{1,\eps}(\eps \mathbb{Z}^d)$ and $g\in L^{1}(\eps^{-1}\mT^d)$ we define the Fourier and the inverse Fourier transform as
$$\cF f(\xi)=\eps^d\sum_{x\in \eps \mZ^d}f(x)e^{-2\pi \iota \xi\cdot x},\quad \cF^{-1} g(x)=\int_{\eps^{-1}\mT^d} g(\xi)e^{2\pi \iota \xi\cdot x}\dif \xi,$$
for $\xi\in \eps^{-1}\mT^d, x\in \eps\mZ^d$, where 
$\|f\|_{ L^{1,\eps}(\eps\mZ^d)}=\eps^d\sum_{x\in\eps \mZ^d}|f(x)|$.
When we write  $\eps=0$  we refer to the quantities in the continuous setting with $\cF_{\mR^d}$ and $\cF_{\mR^d}^{-1}$ being the usual Fourier transform and its inverse on $\mR^d$.
Let $(\varphi_j)_{j\geq-1}$ be a dyadic partition of unity on $\mathbb{R}^d$. We define the dyadic partition of unity for $x\in \eps^{-1}\mathbb{T}^d$:
\begin{align}\label{dyadic}\varphi_j^\eps(x)=\begin{cases}\varphi_j(x), \qquad & j<j_\eps,\\
		1-\sum_{j<j_\eps}\varphi_j(x), \qquad & j=j_\eps.
	\end{cases}
\end{align}
Here $j_\eps:=\inf\{ j:\textrm{supp} \varphi_j\cap \partial (\eps^{-1}\mathbb{T}^d)\neq \emptyset\}$.

Now we define the Littlewood-Paley blocks for distributions on $\Lambda_{\eps}$ by
$$\Delta_j^\eps f=\cF^{-1}(\varphi_j^\eps\cF f),$$
where $f:\Lambda_\eps\to\mR$ is viewed as periodic functions on $\eps\mZ^d$. This
 leads to the definition of Besov spaces. For $\alpha\in \mR, p, q\in [1,\infty]$ and $\eps\in [0,1]$ we define the  Besov spaces on $\Lambda_{\eps}$ given by the norm
$$\|f\|_{B^{\alpha,\eps}_{p,q}}=\Big(\sum_{-1\leq j\leq j_\eps}2^{\alpha jq}\|\Delta_j^\eps f\|_{L^{p,\eps}}^q\Big)^{1/q}<\infty,$$
where $\|f\|_{ L^{p,\eps}}=(\eps^d\sum_{x\in\Lambda_\eps}|f(x)|^p)^{1/p}$, $p\in[1,\infty]$.
If $\eps=0$, $B^{\alpha,\eps}_{p,q}$ is the classical Besov space
$B^{\alpha}_{p,q}$ on $\mT^d$. We also set $H^{\alpha,\eps}=B^{\alpha,\eps}_{2,2}$ and $\bC^{\alpha,\eps}=B^{\alpha,\eps}_{\infty,\infty}$ for $\alpha\in\mR$.

To compare with the definition of Besov spaces on $\eps\mZ^d$, we also introduce the following weighted Besov spaces on
$\eps\mZ^d$  given by the norm
$$\|f\|_{B^{\alpha,\eps}_{p,q}(\eps\mZ^d,\rho)}=\Big(\sum_{-1\leq j\leq j_\eps}2^{\alpha jq}\|\Delta_j^\eps f\|_{L^{p,\eps}(\eps\mZ^d,\rho)}^q\Big)^{1/q}<\infty,$$
for $\alpha\in \mR, p, q\in [1,\infty]$, where $\rho$ is a polynomial weight of the form $\rho(x)=(1+|x|^2)^{-\delta/2}$ for some $\delta\geq0$ and $\|f\|_{ L^{p,\eps}(\eps\mZ^d,\rho)}=(\eps^d\sum_{x\in\eps\mZ^d}|(f\rho)(x)|^p)^{1/p}$, $p\in[1,\infty]$. If $\eps=0$, $B^\alpha_{p,q}(\rho)$ is the usual weighted Besov space. Then it is easy to see that for  functions $f$ on $\Lambda_\eps$ when $\delta p>d$
\begin{align}\label{equi}\|f\|_{B^{\alpha,\eps}_{p,q}}\simeq \|f\|_{B^{\alpha,\eps}_{p,q}(\eps\mZ^d,\rho)},\end{align}
which implies that useful results for functions on $\eps \mZ^d$ from \cite{MP19} can be directly extended to functions on $\Lambda_\eps$.

With the above notations at hand we  extend Lemma \ref{lem:K} to the following discrete version by similar argument as in \cite[Prop. 3.6]{MP19}.

\bl\label{lem:Kd} For $f\in L^{2,\eps}$, $d=2$
\begin{align*}
	K_\eps f(x)=\eps^2	\sum_{y\in \eps \mathbb{Z}^2} K_\eps(x-y)f(y), \quad K_\eps(x-y)=\delta(x-y)+L_\eps(x-y),\quad x,y\in \Lambda_\eps,
\end{align*}
with $L_\eps\in L^{p,\eps}(\eps \mathbb{Z}^2), p\geq 2$, where we view $f$ as periodic function on $\eps \mathbb{Z}^2$.
Furthermore, it holds that for $\gamma\in\mR$, $\delta>0$
\begin{align*}
	\|K_\eps f\|_{H^{\gamma,\eps}}\lesssim \|f\|_{H^{\gamma,\eps}},\quad	\|L_\eps f\|_{H^{2+\gamma-\delta,\eps}}\lesssim \|f\|_{H^{\gamma,\eps}},
\end{align*}
where $L_\eps f(x)=	\eps^2\sum_{y\in \eps \mathbb{Z}^2} L_\eps(x-y)f(y)$, $x\in\Lambda_\eps$.
\el

\subsection{Extension operator}

We follow \cite{MP19} and introduce the following extension operator. Recall $\varphi_j^\eps$ given in \eqref{dyadic}.  We choose a symmetric function $\psi\in C_c(\mathbb{R}^d)$ satisfying the following property:
\begin{itemize}
	\item $\sum_{k\in\mathbb{Z}^d}\psi(\cdot-k)=1,$
	\item $\psi(\eps\cdot)=1$ on supp$\varphi_j$ for $j<j_\eps,$
	\item $(\textrm{supp} \psi\cap \textrm{supp} (\varphi_j)_{\textrm{ext}})\backslash \eps^{-1}\mathbb{T}^d\neq\emptyset \Rightarrow j=j_\eps.$
\end{itemize}
Here  $(\varphi_j)_{\textrm{ext}}$ means periodic extension to $\mathbb{R}^d$ of  $\varphi_j$ on $\mathbb{T}^d$.
In \cite{MP19} such a function $\psi$ is called a smear function. Set $w^\eps=\mathcal{F}^{-1}_{\mathbb{R}^d}\psi(\eps\cdot)$ and  define
\begin{align}\label{def:E}
	\cE^\eps f(x)\eqdef\eps^d\sum_{y\in \eps\mZ^d} w^\eps(x-y)f(y), \quad f\in B^{\alpha,\eps}_{p,q}.\end{align}
Here we extend functions on $\Lambda_\eps$ to functions on $\eps\mZ^d $ by periodic extension. We also introduce the following extension operators for functions  $f:(\eps\mZ^d)^k\to \mR$:
\begin{align}\label{def:Ek}
	\cE_k^\eps f(x_1,\dots,x_k)\eqdef\eps^{dk}\sum_{\substack{y_i\in \eps\mZ^d,\\i=1,\dots,k}}\Big(\prod_{i=1}^k w^\eps(x_i-y_i)\Big)f(y_1,\dots,y_k).\end{align}

By \cite[Lemma 2.24]{MP19} $\cE^\eps$ and $\cE^\eps_k$ are operators which are uniformly  bounded in $\eps$  from discrete Besov spaces $B^{\alpha,\eps}_{p,q}$ to continuous Besov spaces $B^{\alpha}_{p,q}$. 

We also recall the following two results from \cite[Lemma A.9, A.10]{SZZ21}, which are useful for 
passing from functions on $\Lambda_\eps$ to distributions in $\cS'(\mT^d)$
as $\eps\to0$. In fact the results in \cite[Lemma A.9, A.10]{SZZ21} hold for functions in $B^{\alpha,\eps}_{p,q}(\eps\mZ^d,\rho)$. By \eqref{equi} the results also hold for functions in $B^{\alpha,\eps}_{p,q}$.

\bl\label{lem:Epro}
For $p\in [1,\infty]$, $\gamma<0<\alpha$ with $\alpha+\gamma>0$ and $\beta>0$ it holds that
$$\|\cE^\eps(fg)-\cE^\eps f\cE^\eps g\|_{B^{\gamma,\eps}_{p,\infty}}\lesssim o(\eps)\|f\|_{B^{\alpha,\eps}_{p,\infty}}\|g\|_{\bC^{\gamma+\beta,\eps}}.$$
\el

\bl\label{Cepro} Assume that $(\m-\Delta_\eps)u_\eps=f_\eps$.
Then for $\alpha\in \mR$ and $\delta>0$
$$\|\cE^\eps u_\eps-(\m-\Delta)^{-1}\cE^\eps f_\eps\|_{\bC^{\alpha+2-\delta}}\lesssim \eps^\delta \|f_\eps\|_{\bC^{\alpha,\eps}}.$$
\el

We also have similar result for $K_\eps$.

\bl\label{kcon} For $d=2$, $\alpha\in\mR$ and $0<\delta<1$, one has
\begin{equs}
	\|\cE^\eps( K_\eps*_\eps f_\eps)-K*\cE^\eps f_\eps\|_{H^{\alpha-\delta}}&\lesssim \eps^{\delta/2}\|f_\eps\|_{H^{\alpha,\eps}},
\\
	\|\cE^\eps( L_\eps*_\eps f_\eps)-L*\cE^\eps f_\eps\|_{H^{\alpha+2-\delta}}&\lesssim \eps^{\delta/2}\|f_\eps\|_{H^{\alpha,\eps}},
\end{equs}
where $*_\eps$ is the discrete convolution on $\eps\mZ^2$, i.e. $f*_\eps g=\eps^2 \sum_{y\in \eps\mZ^2}f(\cdot-y)g(y)$. 
\el
\begin{proof} 
For the first result, by \eqref{equi} and \cite[Lemma 2.24]{MP19} we only need to prove
\begin{align*}
	\|\cE^\eps( K_\eps*_\eps f_\eps)-K*\cE^\eps f_\eps\|_{B^{\alpha-\delta}_{2,2}(\rho)}\lesssim \eps^{\delta/2}\|\cE^\eps f_\eps\|_{B^{\alpha}_{2,2}(\rho)},
\end{align*}
where $\rho=(1+|x|^2)^{-1}$. 
From the definition of Besov space we have
\begin{align*}
	&\|\cE^\eps( K_\eps*_\eps f_\eps)-K*\cE^\eps f_\eps\|^2_{B^{\alpha-\delta}_{2,2}(\rho)}
\\&\lesssim \sum_j2^{2(\alpha-\delta)j}\|\Delta_j(\cE^\eps( K_\eps*f_\eps)-K*\cE^\eps f_\eps)\|_{L^2(\rho)}^2.
\end{align*}
By the definition of $\cE^\eps$ only $j\lesssim j_\eps$ contributes. 
By \cite[(33)]{MP19},
$$
\mathcal{F}_{\mR^2}(\cE^\eps (K_\eps*f_\eps))=\psi(\eps\cdot)(\mathcal{F}K_\eps\mathcal{F}f_\eps)_{\textrm{ext}},\qquad \mathcal{F}_{\mR^2}( K*\cE^\eps f_\eps)=\mathcal{F}_{\mR^2}K\psi(\eps\cdot)(\mathcal{F}f)_{\textrm{ext}}.
$$
Here $g_{\textrm{ext}}\in \mathcal{S}'(\mR^2)$ is the periodic extension of $g\in\mathcal{S}'(\eps^{-1}\mathbb{T}^2)$ (see \cite[(11)]{MP19}).

	 Moreover, for $j\lesssim j_\eps$
\begin{align*}
\Delta_j(\cE^\eps( K_\eps*f_\eps)-K*\cE^\eps f_\eps)
&=\mathcal{F}^{-1}_{\mR^2}\big(\varphi_j\psi(\eps\cdot)[(\mathcal{F}K_\eps)_{\textrm{ext}}-\mathcal{F}_{\mR^2}K](\mathcal{F}f_\eps)_{\textrm{ext}}\big)
\\
&=\mathcal{F}^{-1}_{\mR^2}[\phi_j(\mathcal{F}K_\eps)_{\textrm{ext}}-\phi_j\mathcal{F}_{\mR^2}K]*\Delta_j\cE^\eps f_\eps.
\end{align*}
Here $\phi_j=\sum_{i:|i-j|\leq1}\varphi_i=:\phi(2^{-j}\cdot)$. Then it suffices to prove
\begin{equation}\label{zjj1}
	\|(1+|x|^2)\mathcal{F}^{-1}_{\mR^2}[\phi_j(\mathcal{F}K_\eps)_{\textrm{ext}}-\phi_j\mathcal{F}_{\mR^2}K]\|_{L^1}\lesssim \eps^{\delta/2} 2^{j\delta}.\end{equation}
By the following calculations 
\begin{align*}
&\|(1+|x|^2)\mathcal{F}^{-1}_{\mR^2}
	[\phi_j(\mathcal{F}K_\eps)_{\textrm{ext}}-\phi_j\mathcal{F}_{\mR^2}K]\|_{L^1}
\\ =\,&\|(1+|2^{-j}x|^2)\mathcal{F}^{-1}_{\mR^2}
[\phi(\mathcal{F}K_\eps)_{\textrm{ext}}(2^j\cdot)-\phi\mathcal{F}_{\mR^2}K(2^j\cdot)]\|_{L^1}
\\\lesssim \, & \|(1+\Delta)^3[\phi(\mathcal{F}K_\eps)_{\textrm{ext}}(2^j\cdot)-\phi\mathcal{F}_{\mR^2}K(2^j\cdot)]\|_{L^\infty},
\end{align*}
\eqref{zjj1} follows from
\begin{align}\label{estc} |D^m[\phi(\mathcal{F}K_\eps)_{\textrm{ext}}(2^j\cdot)-\phi\mathcal{F}_{\mR^2}K(2^j\cdot)]|\lesssim \eps^{\delta/2} 2^{j\delta}, \qquad m\leq 6,m\in\mN. 
\end{align}
Now we prove \eqref{estc}. It is easy to see that
	$$\mathcal{F}K_\eps(\xi_1,\xi_2)= \frac1{1+\mathcal{F}C^2_\eps(\xi_1,\xi_2)},\quad \mathcal{F}_{\mR^2}K(\xi_1,\xi_2)= \frac1{1+\mathcal{F}_{\mR^2}C^2(\xi_1,\xi_2)},$$
with
$$\mathcal{F}C_\eps^2=\cF C_\eps*_\eps\cF C_\eps,\quad \cF_{\mR^2} C^2=\cF_{\mR^2} C*\cF_{\mR^2} C,$$
and
	$$\mathcal{F}C_\eps(\xi_1,\xi_2)= \frac1{\mm+4(\sin^2(\eps\pi\xi_1)+\sin^2(\eps\pi\xi_2))/\eps^2},\quad \mathcal{F}_{\mR^2}C(\xi)=\frac1{\mm+4\pi^2|\xi|^2}.$$
 	Then, for any $0<\eta<1$ on the support of $\cE^\eps$, 
	\begin{equation}\label{A.9}
		\aligned
		&\Big|\mathcal{F}C_\eps(\xi_1,\xi_2)-\frac1{\mm+4\pi^2|\xi|^2}\Big|
		\\=\,&\Big|\frac{4(\sin^2(\eps\pi\xi_1)+\sin^2(\eps\pi\xi_2))/\eps^2-4\pi^2|\xi|^2}{(\mm+4(\sin^2(\eps\pi\xi_1)+\sin^2(\eps\pi\xi_2))/\eps^2)(\mm+4\pi^2|\xi|^2)}\Big|
		\\\lesssim\,&\frac{ |\eps\xi|^\eta|\xi|^2}{(\m+4(\sin^2(\eps\pi\xi_1)+\sin^2(\eps\pi\xi_2))/\eps^2)(\m+4\pi^2|\xi|^2)}
		\\\lesssim\,&\frac{ |\eps\xi|^\eta}{\m+4\pi^2|\xi|^2}.
		\endaligned
	\end{equation}
	Here in the second last  step we used $|\frac{\sin ^2x}{x^2}-1|\lesssim |x|^{\eta}$ and in the last step we used the fact that on the support of $\cE^\eps$, $(\sin^2(\eps\pi\xi_1)+\sin^2(\eps\pi\xi_2))/\eps^2\gtrsim |\xi|^2$ by \cite[Lemma 3.5]{MP19}. \eqref{A.9} implies that on the support of $\cE^\eps$
	$$\Big|\mathcal{F}C_\eps^2(\xi)-\mathcal{F}_{\mR^2}C^2(\xi)\Big|\lesssim\frac{ |\eps\xi|^\eta|\xi|^\kappa}{\m+4\pi^2|\xi|^2},$$
	for $\kappa>0$, which implies \eqref{estc} for $m=0$. 
	We can use similar arguments as in the proof of \cite[Lemma 3.2]{ZZ18} to estimate further derivatives, which implies \eqref{estc}  and the  result for $K$ holds. Regarding $L$, by  Fourier transform as in Lemma \ref{lem:K} the second result follows similarly.
\end{proof}

\bibliographystyle{alphaabbr}
\bibliography{Reference}

\end{document}